\documentclass[12pt]{article}
\pdfoutput=1
\usepackage{amsmath,amsthm,amssymb}
\usepackage[left=3cm,right=3cm, top=2.5cm,bottom=2.5cm,bindingoffset=0cm]{geometry}

\usepackage{amscd}

\usepackage{hyperref}

\usepackage{mathtools}

\usepackage[all,cmtip]{xy}

\usepackage{MnSymbol} 

\usepackage[table]{xcolor}

\usepackage{wrapfig}
\usepackage{graphicx}

\makeatletter
\def\th@plain{%
  \thm@notefont{}
  \itshape 
}
\def\th@definition{%
  \thm@notefont{}
  \normalfont 
}
\makeatother

\newtheorem{lemma}{Lemma}[section]
\newtheorem{proposition}[lemma]{Proposition}

\newtheorem{remark-definition}[lemma]{Remark-Definition}
\newtheorem{theorem}[lemma]{Theorem}
\newtheorem{corollary}[lemma]{Corollary}

\newtheorem{proposition-conjecture}[lemma]{Proposition-conjecture}
\newtheorem{problem}[lemma]{Problem}
\newtheorem{question}[lemma]{Question}
\newtheorem{hypothesis}[lemma]{Hypothesis}

\theoremstyle{definition}
\newtheorem{example}[lemma]{Example}

\newtheorem{definition}[lemma]{Definition}
\newtheorem{remark}[lemma]{Remark}

\newcommand{\Ker}{\mathrm{Ker}\,\,}

\newcommand{\Imm}{\mathrm{Im}\,\,}

\newcommand\Pen{\mathcal{P}}

\def\minus{\hbox{-}}   

\usepackage{multirow}  

\usepackage[normalem]{ulem} 

\usepackage{tikz}
\usetikzlibrary{calc}

\definecolor{block}{RGB}{0,162,232}

\def\blockaux#1(#2,#3)#4(#5,#6){%
  \draw[fill={#1}]
  let \p1=(#2,#3),
      \p2=(#5,#6),
      \p3=(#2+#5,#3+#6),
      \p4=(#2+#5/2,#3+#6/2)
  in
    (\p1) rectangle (\p3)
    (\p4) node {$#4$}
  ;%
}

\begin{document}

\title{Realization of Jordan-Kronecker invariants \\ by Lie algebras}
\author{I.\,K.~Kozlov\thanks{No Affiliation, Moscow, Russia. E-mail: {\tt ikozlov90@gmail.com} }
}
\date{}

\maketitle

\begin{abstract} We study what Jordan-Kronecker invariants of Lie algebras, introduced by A.\,V.~Bolsinov and P.~Zhang, are possible. We completely solve this problem in the Jordan and the Kronecker cases. We prove that any JK invariants that contain the Kronecker $3 \times 3$ block or several Kronecker $1 \times 1$ blocks are possible. For other JK invariants, with Kronecker indices $k_1, \dots, k_q$, we give a partial answer:

\begin{itemize}

\item all Jordan--Kronecker invariants with no more than $\sum_i k_i$ Jordan tuples with multiple maxima are possible. 

\item the Jordan--Kronecker invariants with more than $\sum_i k_i$ unique Jordan tuples with multiple maxima are impossible. 

\end{itemize} 
We also desribe all JK invariants that can be realized by compatible Poisson brackets with non-constant eigenvalues.\end{abstract}

\tableofcontents

\section{Introduction}

Jordan--Kronecker invariants of finite-dimensional Lie algebras were introduced in \cite{BolsZhang}. They arise as a combination of two simple facts:

\begin{itemize}

\item Let  $\mathfrak{g}$ be a finite-dimensional Lie algebra and  $\mathfrak{g}^*$ be its dual space. Any element $x\in \mathfrak{g}^*$ defines a skew-symmetric bilinear form $\mathcal{A}_x = \left( c_{ij}^k x_k\right)$, where $c_{ij}^k$ are the structural constants of $\mathfrak{g}$.  

\item There is a canonical form for a pencil of skew-symmetric bilinear forms given by the Jordan--Kronecker theorem (Theorem~\ref{T:Jordan-Kronecker_theorem}, see, e.g., \cite{Thompson, Gantmaher88}). 

\end{itemize}

The Jordan--Kronecker invariants of $\mathfrak{g}$ are determined by the canonical form of the pencil $\mathcal A_{x+ \lambda a}$ for a generic pair $(x,a)\in \mathfrak{g}^*\times \mathfrak{g}^*$ (see Section~\ref{S:JKLieDef}). In short, there are two kinds of blocks in the Jordan--Kronecker theorem: Kronecker blocks and Jordan blocks. The Jordan--Kronecker invariants of $\mathfrak{g}$ are 

\begin{itemize}

\item sizes of Kronecker blocks,

\item sizes of Jordan blocks grouped by eigenvalues

\end{itemize}

for a generic pencil $\mathcal{A}_{x + \lambda a}$. The JK invariants of $\mathfrak{g}$ are a collection of Jordan tuples and Kronecker sizes: \begin{equation} \label{Eq:NotationJK} J_{\lambda_1}(2n_{11}, \dots, 2n_{1s_1}), \quad  \dots, \quad  J_{\lambda_p}(2n_{p1}, \dots, 2n_{ps_p}), \quad 2k_1 -1, \quad \dots,  \quad 2k_q - 1. \end{equation} Each \textbf{Kronecker size} $2k_j-1$ represent a  Kronecker $(2k_j-1)\times(2k_j-1)$ block. Each \textbf{Jordan tuple} $J_{\lambda_i}(2n_{i1}, \dots, 2n_{is_i})$ represent Jordan blocks with sizes  $2n_{i1} \times 2n_{i1}, \dots, 2n_{is_i} \times 2n_{is_i}$ and the same eigenvalue $\lambda_i$. In \eqref{Eq:NotationJK} we regard $\lambda_i$ as (distinct) formal variables, since the eigenvalues $\lambda_i$ depend on $(x,a)$. Nevertheless, the number and sizes of blocks are the same for a generic pair $(x,a)\in \mathfrak{g}^*\times \mathfrak{g}^*$.

We always assume that the numbers in the Jordan tuples $J_{\lambda_i}(2n_{i1}, \dots, 2n_{is_i})$ are sorted in the descending order: \[n_{i1} \geq n_{i2} \geq \dots \geq n_{is_i}.\] We say that a Jordan tuple \textbf{has a unique maximum}, if $n_{i1} > n_{i2}$ or $s_i=1$. Otherwise, if $n_{i1} = n_{i2}$, then the Jordan tuple \textbf{has multiple maxima}. A Jordan tuple $J_{\lambda_i}\left(2n_{i1}, \dots, 2n_{is_i}\right)$ is \textbf{unique}, if \[ \left(2n_{i1}, \dots, 2n_{is_i}\right) \not = \left(2n_{j1}, \dots, 2n_{js_j}\right), \qquad i \not = j.\]

Apart from  \cite{BolsZhang} and references therein, the Jordan--Kronecker invariants for various Lie algebras were calculated in \cite{Vor1}, \cite{Vor2}, \cite{Vor3}, \cite{Vor4}, \cite{Gar1}, \cite{Gar2}. Yet the following question, asked in \cite{BolsIzosKonOsh},  \cite{BolsIzosTson} and \cite{BolsMatvMirTab} remains open. 

\begin{question} \label{Q:JKLie} What Jordan--Kronecker invariants can be realized by complex finite-dimensional Lie algebras?\end{question}

We completely solve this question in the pure Kronecker and pure Jordan case, when there is only one type of blocks. In the mixed case we give a partial answer. Here is a short summary of our results:

\begin{enumerate}

\item (Kronecker case) Any collection of Kronecker sizes $2k_1-1, \dots, 2k_q-1$ is possible (Theorems~\ref{Th:Sum} and \ref{T:KronInv}).

\item (Jordan case) A collection of Jordan tuples $J_{\lambda_i} (2n_{i1}, \dots 2n_{i s_i})$, $i=1,\dots, p$,  is possible if and only if each of them has a unique maximum, i.e. $n_{i1} > n_{ij}$ (Theorem~\ref{Th:JordanCase}).

\item (Mixed case, ``Positive result'') All JK invariants with Kronecker $3\times 3$ block or several $1 \times 1$ blocks are possible (Theorems~ \ref{Th:TrivialKron} and \ref{Th:Kron3}).

\item (Mixed case, $1\times 1$ Kronecker block) A collection of Jordan tuples $J_{\lambda_i} (2n_{i1}, \dots 2n_{i s_i})$, $i=1,\dots, p$, and one Kronecker $1\times 1$ block is possible if and only there is no more than $1$ Jordan tuple with multiple maxima, i.e. $n_{i_0 1} = n_{i_0 2} \geq n_{i_0 j}$ for some $i_0 \in \left\{1,\dots, i\right\}$  (Theorem~\ref{Th:TrivialKron}). 

\item (General mixed case) Consider the remaining mixed case, i.e. there are no Kronecker $3 \times 3$ blocks and no more than one $1 \times 1$ Kronecker block. Assume that the Kronecker sizes are $2k_1-1, \dots, 2k_q-1$. 

\begin{enumerate} 

\item (``Positive result'') All JK invariants that have no more than $k_1 + \dots + k_q$ Jordan tuples with multiple maxima are possible (Theorem~\ref{Th:RealOneKronSeveralEigen}). 

\item (``Negative result'') The JK invariants that have more than $k_1 + \dots + k_q$ unique Jordan tuples with multiple maxima are impossible (Theorem~\ref{Th:FinalObs}).

\end{enumerate}

\item (1 Kronecker block, no semi-invariants) Let $\mathfrak{g}$ be Lie algebra without (proper) semi-invariants\footnote{$\mathfrak{g}$ has no (proper) semi-invariants if its radical is nilpotent, see Remark~\ref{Rem:NilRad}. It is also well-known that if $\mathfrak{g}$ is perfect, i.e. $[\mathfrak{g}, \mathfrak{g}] = \mathfrak{g}$, then its radical is nilpotent and, hence, $\mathfrak{g}$ has no semi-invariants.} and  with only one Kronecker block (i.e. $\operatorname{ind} \mathfrak{g} = 1$), the size of the Kronecker block is $(2k-1) \times(2k-1)$. Either there are no Jordan tuples, or there are $k$ equal Jordan tuples (see Theorem~\ref{T:Ind1Th}). 

\end{enumerate}

The JK invariants of a Lie algebra $\mathfrak{g}$ are the invariants of its Lie--Poisson pencil (see Section~\ref{S:JKLiePoisDef}). For any Poisson pencil $\mathcal{P} = \left\{ \mathcal{A} + \lambda \mathcal{B} \right\}$ on a manifold $M$ we can define the JK invariants and algebraic type of the pencil $\mathcal{P}$ at each point $x \in M$ (see Section~\ref{S:CompPoisBrack}). Simply speaking, if the JK invariants are \[J_{\lambda_1(x)}(2n_{11}, \dots, 2n_{1s_1}), \quad  \dots, \quad  J_{\lambda_p(x)}(2n_{p1}, \dots, 2n_{ps_p}), \quad 2k_1 -1, \quad \dots,  \quad 2k_q - 1, \] then the algebraic type is the multiset \[ \left\{ (2n_{11}, \dots, 2n_{1s_1}), \quad \dots, \quad (2n_{p1}, \dots, 2n_{ps_p}), \quad 2k_1-1, \quad \dots, \quad 2k_q -1\right\}.\] Obviously, the algeraic type is the same in a neighbouhood of a generic point $x \in M$. Instead of Question~\ref{Q:JKLie} we can ask the following more general question.

\begin{question} \label{Q:RealPoisBr} What JK invariants can be realized by compatible Poisson brackets with non-constant eigenvalues?
More formally, let $\mathcal{P}$ be a Poisson pencil on a manifold $M$ such that

\begin{enumerate}

\item the algebraic type of $\mathcal{P}$ at all points $x \in M$ is the same,

\item $d \lambda_i(x) \not = 0$, where $\lambda_i(x)$ are the eigenvalues of $\mathcal{P}$.

\end{enumerate}

What is the algebraic type of $\mathcal{P}$ (at any point $x\in M$)?
\end{question}

\begin{remark} Obviously, we can realize all JK invariants by constant Poisson brackets. Therefore we add the condition $d \lambda_i(0) \not =0$. This is a simple natural condition that is satisfied for the  Lie--Poisson pencils (Lemma~\ref{L:LieEigenNonConst}) and is the main obstruction for realization of JK invariants in the Jordan case (Section~\ref{S:ObstJord}). \end{remark}

Simply speaking, Question~\ref{Q:RealPoisBr} asks: what are differential-geometric obstructions to realization of JK invariants? It turns out that there is only one obstruction: in the Jordan case (i.e. when there are no Kronecker blocks) each Jordan tuple must have a unique maximum.  This follows from Turiel's theorems about compatible symplectic structures (see Section~\ref{S:ObstJord}). In the Jordan case the JK invariants are realized in Theorem~\ref{Th:JordanCase}. As stated above, for Lie algebras there are other obstructions in the mixed case. But for Poisson pencils with non-constant eigenvalues all JK invariants with at least one Kronecker block are possible.

\begin{theorem} \label{T:RealPenc} For any collection (i.e. multiset) of tuples and numbers \begin{equation} \label{Eq:CollMult1} \left\{ (2n_{11}, \dots, 2n_{1s_1}), \dots, (2n_{p1}, \dots, 2n_{ps_p}), \quad 2k_1-1, \dots, 2k_q -1\right\} \end{equation} with at least one number (i.e. $q \geq 1$) there exists a a manifold $M$ and a Poisson pencil  $\mathcal{P}$ on it such that:

\begin{enumerate}

\item  the JK invariants of $\mathcal{P}$ at all points $x \in M$ are the Jordan tuples $J_{\lambda_i(x)}(2n_{i1}, \dots, 2n_{is_i})$, $i=1, \dots, p$ and the Kronecker sizes $2k_j -1$, $j=1, \dots, q$,

\item all eigenvalues $\lambda_i(x)$ are distinct and $d \lambda_i(x) \not = 0$ for any point $x \in M$. 

\end{enumerate}

\end{theorem}

\begin{remark} Theorem~\ref{T:RealPenc} formally holds if there are no Jordan tuples, i.e. $p =0$. \end{remark}

We prove Theorem~\ref{T:RealPenc} in Section~\ref{S:ProofThrealPoiss}.  We cannot realize some JK invariants by Lie algebras, because the local structure of Poisson pencils sometimes clashes with the properties of Lie algebras. The restrictions on JK invariants of Lie algebras, that we discuss in this paper, mainly come from the following $2$ important results of F.\,J.~Turiel about the local structure of Poisson pencils:

\begin{itemize}

\item  Local form for nondegenerate Poisson pencils (=compatible symplectic structures), \cite{turiel} (see Section~\ref{S:ObstJord} and Appendix~\ref{S:TurielSympApp}).

\item Local decomposition of Poisson pencils into the Jordan--Kronecker product, \cite{Turiel11} (see Theorem~\ref{T:TurielDecompTrivKronOneJord} and Appendix~\ref{S:TurielDecompApp}). 

\end{itemize}

Also, in this paper we prove several interesting results that can be useful in a general study of Poisson pencils or Lie algebras:

\begin{itemize}

\item  In Section~\ref{S:CoreMantleSection} we introduce some useful local coordinates for Poisson pencils $\mathcal{A} + \lambda \mathcal{B}$ (Theorems~\ref{T:BiPoissRedCoreMantle} and \ref{T:TrivKronFact}).

\item In Lemma~\ref{L:EigenDiff} we prove a well-known statement that eigenvalues $\lambda_i(x)$ of a Poisson pencil  $\mathcal{A} + \lambda \mathcal{B}$ satisfy \[  (\mathcal{A} - \lambda_i (x) \mathcal{B}) d\lambda_i (x) =0\] We also prove a similar statement for semi-invariants of Lie algebras in Lemma~\ref{L:EigenSemiInv}.
\end{itemize}

All Lie algebras are finite-dimensional and complex (although the results about Lie algebras remain true for an arbitrary algebraically closed field $\mathbb{K}$ with $\operatorname{char} \mathbb{K} = 0$). For a manifold $M$ there are two cases:

\begin{itemize}

\item $M$ is complex, and all the functions, Poisson brackets, differential forms, etc. are complex analytic and holomorphic.

\item $M$ is real $\mathbb{C}^{\infty}$-smooth, and all the functions, Poisson brackets, differential forms, etc. are $\mathbb{C}^{\infty}$-smooth.

\end{itemize}

Since we study complex Lie algebras, all the statements hold in the complex case. The majority of statements hold in the real $\mathbb{C}^{\infty}$-case, and we specify if a statement doesn't hold or has to be modified in the real case\footnote{In the real case we often need to add the condition that all the eigenvalues $\lambda_i$ are real. The case of complex conjugate eigenvalues $\lambda_{\pm}  = \alpha \pm i \beta$ often can be dealt with, since there is a natural almost complex structure $J$. The construction of such structure $J$ is discussed, e.g. in \cite{turiel}, \cite{Kozlov15} or \cite{BolsinovN1}. We don't bother with the complex conjugate eigenvalues in this paper.}. We try to state results in such a way that they hold in both complex and real case. In the complex case by ``smooth'' we mean complex analytic and holomorphic. 

All the manifolds are connected. Moreover, the statements about manifolds are local. Therefore, without loss of generality, one may consider all manifolds to be open disks in $\mathbb{C}^n$ (or in $\mathbb{R}^n$ in the real case). We say that some property holds for a \textit{generic} point of a manifold $M$, if it holds on an open dense subset of $M$. For short, we use JK as an abbreviation of Jordan-Kronecker.

\par\medskip
 
\textbf{Acknowledgements.} The author would like to thank A.\,V.~Bolsinov and A.\,M.~Izosimov for useful comments.
 
\section{Basic definitions}

For completeness sake we briefly recall the definition of JK invariants and some important facts about them. For details, see \cite{BolsZhang}. 

\subsection{Jordan--Kronecker theorem} 

First, let us recall the canonical form for a pair of skew-symmetric forms.

\begin{theorem}[Jordan--Kronecker theorem,  \cite{Thompson}]\label{T:Jordan-Kronecker_theorem}
Let $A$ and $B$ be skew-symmetric bilinear forms on a
finite-dimension vector space $V$ over a field $\mathbb{K}$ with $\textmd{char }  \mathbb{K} =0$. If the field $\mathbb{K}$
is algebraically closed, then there exists a basis of the space $V$
such that the matrices of both forms $A$ and $B$ are block-diagonal
matrices:

{\footnotesize
$$
A =
\begin{pmatrix}
A_1 &     &        &      \\
    & A_2 &        &      \\
    &     & \ddots &      \\
    &     &        & A_k  \\
\end{pmatrix}
\quad  B=
\begin{pmatrix}
B_1 &     &        &      \\
    & B_2 &        &      \\
    &     & \ddots &      \\
    &     &        & B_k  \\
\end{pmatrix}
$$
}

where each pair of corresponding blocks $A_i$ and $B_i$ is one of
the following:

\begin{itemize}

\item Jordan block with eigenvalue $\lambda_i \in \mathbb{K}$: {\scriptsize  \begin{equation} \label{Eq:JordBlockL} A_i =\left(
\begin{array}{c|c}
  0 & \begin{matrix}
   \lambda_i &1&        & \\
      & \lambda_i & \ddots &     \\
      &        & \ddots & 1  \\
      &        &        & \lambda_i   \\
    \end{matrix} \\
  \hline
  \begin{matrix}
  \minus\lambda_i  &        &   & \\
  \minus1   & \minus\lambda_i &     &\\
      & \ddots & \ddots &  \\
      &        & \minus1   & \minus\lambda_i \\
  \end{matrix} & 0
 \end{array}
 \right)
\quad  B_i= \left(
\begin{array}{c|c}
  0 & \begin{matrix}
    1 & &        & \\
      & 1 &  &     \\
      &        & \ddots &   \\
      &        &        & 1   \\
    \end{matrix} \\
  \hline
  \begin{matrix}
  \minus1  &        &   & \\
     & \minus1 &     &\\
      &  & \ddots &  \\
      &        &    & \minus1 \\
  \end{matrix} & 0
 \end{array}
 \right)
\end{equation}} \item Jordan block with eigenvalue $\infty$ {\scriptsize \begin{equation} \label{Eq:JordBlockInf}
A_i = \left(
\begin{array}{c|c}
  0 & \begin{matrix}
   1 & &        & \\
      &1 &  &     \\
      &        & \ddots &   \\
      &        &        & 1   \\
    \end{matrix} \\
  \hline
  \begin{matrix}
  \minus1  &        &   & \\
     & \minus1 &     &\\
      &  & \ddots &  \\
      &        &    & \minus1 \\
  \end{matrix} & 0
 \end{array}
 \right)
\quad B_i = \left(
\begin{array}{c|c}
  0 & \begin{matrix}
    0 & 1&        & \\
      & 0 & \ddots &     \\
      &        & \ddots & 1  \\
      &        &        & 0   \\
    \end{matrix} \\
  \hline
  \begin{matrix}
  0  &        &   & \\
  \minus1   & 0 &     &\\
      & \ddots & \ddots &  \\
      &        & \minus1   & 0 \\
  \end{matrix} & 0
 \end{array}
 \right)
 \end{equation} } \item   Kronecker block {\scriptsize \begin{equation} \label{Eq:KronBlock} A_i = \left(
\begin{array}{c|c}
  0 & \begin{matrix}
   1 & 0      &        &     \\
      & \ddots & \ddots &     \\
      &        & 1    &  0  \\
    \end{matrix} \\
  \hline
  \begin{matrix}
  \minus1  &        &    \\
  0   & \ddots &    \\
      & \ddots & \minus1 \\
      &        & 0  \\
  \end{matrix} & 0
 \end{array}
 \right) \quad  B_i= \left(
\begin{array}{c|c}
  0 & \begin{matrix}
    0 & 1      &        &     \\
      & \ddots & \ddots &     \\
      &        &   0    & 1  \\
    \end{matrix} \\
  \hline
  \begin{matrix}
  0  &        &    \\
  \minus1   & \ddots &    \\
      & \ddots & 0 \\
      &        & \minus1  \\
  \end{matrix} & 0
 \end{array}
 \right)
 \end{equation} }
 \end{itemize}

\end{theorem}

A proof of the Jordan--Kronecker theorem can be found in \cite{Thompson}, 
which is based on \cite{Gantmaher88}.

\begin{remark} A Jordan $\infty$-block is a Jordan $0$-block, where the matrices $A_i$ and $B_i$ are swapped. \end{remark}

\begin{remark} Each Kronecker block is a $(2k_i-1) \times (2k_i-1)$ block, where
$k_i \in \mathbb{N}$. If $k_i=1$, then the blocks are $1\times 1$
zero matrices
\[
A_i =
\begin{pmatrix}
0
\end{pmatrix} \quad  B_i=
\begin{pmatrix}
0
\end{pmatrix}
\]
\end{remark}

\begin{remark} \label{Rem:JordSign} Note that in \cite{BolsZhang} the matrix $B_i$ for Jordan blocks is taken with the opposite sign. This is a question of convention --- whether $\lambda_i$ or $-\lambda_i$ should be called an eigenvalue. Here is the difference:

\begin{enumerate}

\item In this paper the Jordan block with eigenvalue $\lambda_i$ is given by \[ A_i= \left( \begin{matrix} 0 & J_{\lambda_i} \\ -J_{\lambda_i}^T & 0 \end{matrix}  \right), \qquad B_i = \left( \begin{matrix} 0 & E \\ -E & 0 \end{matrix}  \right).\] Then the matrix of recursion operator consists of 2 blocks with eigenvalue $\lambda_i$: \[P_i = B_i^{-1} A_{i} =\left( \begin{matrix}  J_{\lambda_i}^T & 0  \\ 0 & J_{\lambda_i} \end{matrix}  \right).  \] But the rank of $A+\lambda B$ drops for \underline{minus $\lambda_i$}: \[ \operatorname{rk} \left( A - \lambda_i B\right) < \max_{\lambda} \,\, \operatorname{rk} \left( A + \lambda B\right).  \]

\item In \cite{BolsZhang}  the Jordan block with eigenvalue $\lambda_i$ is given by \[ A_i= \left( \begin{matrix} 0 & J_{\lambda_i} \\ -J_{\lambda_i}^T & 0 \end{matrix}  \right), \qquad B_i = \left( \begin{matrix} 0 & -E \\ E & 0 \end{matrix}  \right).\]  Then the rank drops for $\lambda_i$: \[ \operatorname{rk} \left( A + \lambda_i B\right) < \max_{\lambda}  \,\, \operatorname{rk} \left( A + \lambda B\right).  \] But the matrix of recursion operator consists of 2 blocks with eigenvalue \underline{minus $\lambda_i$}: \[P_i = B_i^{-1} A_{i} =\left( \begin{matrix}  -J_{\lambda_i}^T & 0  \\ 0 & -J_{\lambda_i} \end{matrix}  \right).  \] 

\end{enumerate}

We chose the first variant, so that a pair of forms $A, B$  and the recursion operator $P = B^{-1}A$ would have the same eigenvalues.  Also, $B_i = \left( \begin{matrix} 0 & E \\ -E & 0 \end{matrix}  \right)$ means that the basis for a Jordan block is a Darboux basis for the symplectic form $B_i$, which is ``more natural''. We pay for it with ``an ugly minus'' in $\operatorname{Ker} (A- \lambda_i B)$. Oh well, one cannot get all the pluses. 
\end{remark}

\subsubsection{Linear Poisson pencils}

Let $A$ and $B$ be skew-symmetric forms on a complex finite-dimensional vector space $V$. We use the following notations: \[A_{\lambda} = A + \lambda B, \qquad A_{\infty} = B.\] Let us intoduce several definitions. 

\begin{itemize}

\item We call the one-parametric family of skew-symmetric forms \[\mathcal{P} = \left\{ A + \lambda  B \,\,  \bigr| \,\,\lambda \in \mathbb{C} \cup  \left\{  \infty \right\} \right\} \] a \textbf{linear Poisson pencil} on $V$. Note that the parameter $\lambda$ of the pencil belongs to the complex projective line $\lambda \in \bar{\mathbb{C}} = \mathbb{CP}^1 = \mathbb{C} \cup \left\{ \infty \right\}$. 

\item A vector space with a linear Poisson pencil on it $(V, \Pen)$ is a \textbf{bi-Poisson vector space}. We may denote $(V, \Pen)$ by $V$ if the pencil is clear from the context. A subspace of a bi-Poisson space $W \subset (V, \Pen)$ is considered with the induced pencil $\Pen \bigr|_{W}$ on it.

\item The \textbf{rank} of a pencil $\mathcal{P}$ is \[ \operatorname{rk} \mathcal{P} = \max_{\lambda \in \bar{\mathbb{C}}} \, \operatorname{rk} A_{\lambda}.\] 

\item A form $A_\lambda$ is \textbf{regular} if $\operatorname{rk}A_{\lambda} = \operatorname{rk} \mathcal{P}$. 

\item The eigenvalues of Jordan blocks $\lambda_i$ are also called  \textbf{characteristic numbers} of the pencil $\mathcal{P}$. The set of characteristic numbers $\lambda_i$ of the pencil $\mathcal{P}$ is denoted by \[\Lambda_{\Pen} = \left\{ \lambda \in \mathbb{CP}^1 \quad \bigr| \quad \operatorname{rk} A_{-\lambda} < \operatorname{rk} \mathcal{P} \right\}. \] In other words, a form $A_{-\lambda}= A - \lambda B$ is regular if and only if $\lambda \not \in \Lambda_{\Pen}$. 

\item We call a decomposition of $V$ into a sum of subspaces corresponding to Jordan and Kronecker blocks a \textbf{Jordan-Kronecker decomposition}:  \begin{equation} \label{Eq:JKDecomp} \left( V, \Pen\right) = \bigoplus_{j=1}^N \left( V_{J_{\lambda_j, 2n_j}}, \Pen_{J_{\lambda_j, 2n_j}} \right)\oplus  \bigoplus_{i=1}^q \left(V_{K_i}, \Pen_{K_i} \right).\end{equation} 

Since the pencils on the subspaces $V_{J_{\lambda_j, 2n_j}}$ and $V_{K_i}$ are induced by $\Pen$, we often omit them and denote the JK decomposition \eqref{Eq:JKDecomp} as \[ V = \left( \bigoplus_{j=1}^N  V_{J_{\lambda_j, 2n_j}} \right) \oplus  \left( \bigoplus_{i=1}^q V_{K_i}\right).\]

\begin{itemize}

\item $\left( V_{J_{\lambda_j, 2n_j}}, \Pen_{J_{\lambda_j, 2n_j}} \right)$ corresponds to a Jordan $2n_j \times 2n_j $ block with eigenvalue $\lambda_j$ and is called a \textbf{Jordan subspace} with Jordan index $2n_j$ and eigenvalue $\lambda$.

\item $\left(V_{K_i}, \Pen_{K_i} \right)$ corresponds to a Kronecker $\left(2k_i - 1 \right) \times \left(2k_i - 1 \right)$ block and is called a \textbf{Kronecker subspace} with Kronecker index $k_i$.

\end{itemize}

\item A basis of a Jordan or a Kronecker subspace is a \textbf{standard basis} if the forms $A$ and $B$ are as in the JK theorem (i.e. they have the form \eqref{Eq:JordBlockL}, \eqref{Eq:JordBlockInf} or \eqref{Eq:KronBlock}). 

\end{itemize}

The JK decomposition is not unique, just as the basis in the JK theorem. But the sizes and types of blocks in the JK theorem are uniquely defined.

\subsection{Core and mantle subspaces}

\subsubsection{Definition of the core and the mantle}

Denote by $V_J$ and $V_K$ the sum of all Jordan and all Kronecker subspaces respectively: \[V_J = \bigoplus_{j=1}^N V_{J_{\lambda_j, 2n_j}}, \qquad V_K =  \bigoplus_{i=1}^q V_{K_i}.\] The decomposition \[ V = V_J + V_K \] for a linear pencil $A + \lambda B$ is not as natural as it seems. Rather, intuitively speaking, the Jordan part $V_J$ is ``sandwiched'' between ``the greater and lesser halfs'' of Kronecker blocks. In order to formalize this intuition, let us introduce two important invariantly defined subspaces, that we will need below.

\begin{definition} Consider a pencil of skew-symmetric forms $\left\{ A_{\lambda} = A + \lambda B\right\}$.

\begin{enumerate}

\item The \textbf{core} subspace\footnote{This neat terminology is due to A.\,V.~Bolsinov. } is \[ K = \sum_{\lambda - regular} \operatorname{Ker} A_{\lambda}. \] 

\item The \textbf{mantle} subspace is the skew-orthogonal complement to the core (w.r.t. any regular form $A_{\mu}$) \[ M = K^{\perp}. \] 

\end{enumerate}

\end{definition} 

Now fix any basis from the JK theorem and let us describe the core and mantle subspaces in it. First, let us describe $\operatorname{Ker} A_{\lambda}$ for each type of block.

\begin{proposition} \label{Prop:KernelJordKronBlocks}

\begin{enumerate}

\item For one Jordan block with eigenvalue $\lambda_0$ \[ A + \lambda B = \left( \begin{matrix} 0 & J_{\lambda_0 + \lambda} \\ -J_{\lambda_0 + \lambda}^T & 0 \end{matrix}  \right), \qquad J_{\mu} = \left(\begin{matrix} \mu & 1 & &  \\  & \ddots & \ddots & & \\   &  & \ddots & 1 \\   & & & \mu \end{matrix}  \right)\] we have $\operatorname{Ker} (A + \lambda B ) = 0$ unless $\lambda = - \lambda_0$. The same holds if $\lambda_0 = \infty$. 

\item  And for one Kronecker block \[ A + \lambda B = \left( \begin{matrix} 0 & P \\ -P^T & 0 \end{matrix}  \right), \qquad P = \left(\begin{matrix} 1 & \lambda & & \\ & 1 & \ddots & & \\ & & \ddots & \ddots & \\ & & & 1 & \lambda  \end{matrix}  \right)\]  with a standard basis $e_1, \dots, e_{k-1}, f_1, \dots, f_k$  we have \begin{equation} \label{Eq:KerKronBlock} \operatorname{Ker} \left( A + \lambda B \right) = \operatorname{span} \left(f_k - \lambda f_{k-1} + \lambda^2 f_{k-2} - \dots \right). \end{equation}

\end{enumerate}

\end{proposition}

As a consequence we get another easy description of the core and mantle.

\begin{corollary} \label{Cor:CoreMantle} For any JK decomposition \eqref{Eq:JKDecomp} we have the following. 
\begin{enumerate}

\item The core subspace $K$ of $V$ is the direct sum of core subspaces of Kronecker subspaces $V_{K_i}$: \[ K = \bigoplus_{i=1}^q \left(K \cap V_{K_i}\right).\] If $e_1, \dots, e_{k_i-1},  f_1, \dots, f_{k_i}$ is a standard basis of $V_{K_i}$, then the core subspace of $V_{K_i}$ is \[K \cap V_{K_i} = \operatorname{span} \left(f_1, \dots, f_{k_i} \right).\]

\item The mantle subspace $M$ is the core plus all Jordan blocks: \[ M = K \oplus V_J. \]

\end{enumerate}

\end{corollary} 

Simply speaking, the core subspace $K$ is spanned by the basis vectors corresponding to the down-right zero matrices of Kronecker blocks, like this one:  \[ A_i + \lambda B_i = \left(
\begin{array}{c|c}
  0 & \begin{matrix}
   1 & \lambda      &        &     \\
      & \ddots & \ddots &     \\
      &        & 1    & \lambda  \\
    \end{matrix} \\
  \hline
  \begin{matrix}
  \minus1  &        &    \\
  \minus \lambda   & \ddots &    \\
      & \ddots & \minus1 \\
      &        & \minus\lambda  \\
  \end{matrix} &\cellcolor{blue!25} 0 
 \end{array}
 \right).
 \]

Corollary~\ref{Cor:CoreMantle} shows the following:

\begin{enumerate}

\item the mantle subspace $M$ is correctly defined;

\item the core $K$ is a \textbf{bi-isotropic} subspace (i.e. it is isotropic w.r.t. any regular form $A+\lambda B$);

\item the restrictions of $A, B$ to $M/K$ are well-defined. 

\end{enumerate}

\begin{remark} Now, the analogy with the Earth in Fig.~\ref{Fig:CoreMantle} becomes obvious. The mantle $M$ ``wraps around'' the core $K$, the rest of Kronecker blocks is ``kind of a crust''.  \begin{figure}[h!]
  \centering
    \includegraphics[width=0.3\textwidth]{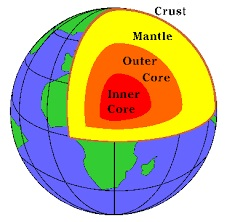}
      \caption{Earth structure.}
      \label{Fig:CoreMantle}
\end{figure} For those who are interested in this ``entertaining geology'', the mantle and the core also have some natural structures, although we do not need them.  For example:

\begin{enumerate}

\item The core $K$ has layers, namely, sums of ``bigger half'' of Kronecker blocks with size not greater than $(2k-1) \times (2k-1)$. Invariantly, each such layer is generated by polynomial solutions of \[(A + \lambda B) v(\lambda)  = 0\] with $\deg v(\lambda) < k$.

\item The quotient $M/K \approx V_J$ naturally decomposes into the sums of Jordan blocks corresponding to different eigenvalues $\lambda_i$: \[ V_J = \bigoplus_{i=1}^p  V_{J_{\lambda_i}}, \qquad   V_{J_{\lambda_i}} = \bigoplus_{j=1}^{s_i} V_{J_{\lambda_i, 2n_{ij}}}.\]

\end{enumerate}

\end{remark}

\subsubsection{Kernels for linear Poisson pencils}

Below, we need several simple corollaries from Proposition~\ref{Prop:KernelJordKronBlocks} about linear independence of vectors $v_i \in \Ker A_{\lambda_i}$. For one Kronecker block $\Ker A_{\lambda}$ is given by  \eqref{Eq:KerKronBlock} and we get the following.

\begin{proposition} \label{P:LinIndKronBlocks} Let $A + \lambda B$ be the $(2k-1) \times (2k-1)$ Kronecker block. Then non-zero vectors $v_i \in \Ker A_{\lambda_i}$ for $M$ distinct values $\lambda_i \in \mathbb{CP}^1$ are linearly independant if and only if $M \leq k$.  \end{proposition}

Recall that a $(2k-1) \times (2k-1)$ Kronecker block has Kronecker index $k$. In general case we get the following.

\begin{proposition} \label{P:KronckerKernels} Consider a linear pencil $A+\lambda B$ on a vector space $V$. 

\begin{enumerate}

\item $\operatorname{Ker} A_{\lambda} \cap \operatorname{Ker} A_{\mu}$,  for any $\lambda \not = \mu$, is the sum of Kronecker subspaces with Kronecker index $1$.

\item If $v_i \in \operatorname{Ker} A_{\lambda_i}$, where $i =1, \dots, M$ and $\lambda_i \not = \lambda_j$, are linearly dependent, then all $v_i$ belong to the sum of Kronecker subspaces with Kronecker indices $k_j < M$.

\end{enumerate}

\end{proposition}

Subspaces from Proposition~\ref{P:KronckerKernels}  can also be described as follows:

\begin{itemize}

\item  $\operatorname{Ker} A_{\lambda} \cap \operatorname{Ker} A_{\mu}$ corresponds to the sum of trivial $1\times 1$ Kronecker blocks.

\item Kronecker subspaces with Kronecker indices $k_j < M$ correspond to the sums of Kronecker blocks with size not bigger than $(2M-3) \times (2M-3)$.  

\end{itemize}

\begin{proof}[Proof of Proposition~\ref{P:KronckerKernels}] 
\begin{enumerate}

\item Easily follows from the JK theorem.

\item Let $A_{\lambda}$ be a regular form. $\Ker A_{\lambda}$ is the sum of its intersection with Kronecker blocks:  \[ \Ker A_{\lambda} =  \Ker A_{\lambda} \cap  \bigoplus_{i=1}^q V_{K_i}. \] Let $\sum_{i} c_i v_i =0$. Without loss of generality all $c_i \not = 0$. The same equality must hold for each Kronecker block. By Proposition~\ref{P:LinIndKronBlocks} if the Kronecker index $k_j \geq M$, then the components of all $v_i$ in this Kronecker subspace must be zero vectors. Thus, all $v_i$ belong to the sum of Kronecker subspaces with Kronecker indices $k_j < M$, as required. \end{enumerate}

Proposition~\ref{P:KronckerKernels} is proved.  \end{proof}

\subsection{Jordan-Kronecker invariants}

\begin{definition} \label{Def:JK} The \textbf{Jordan-Kronecker invariants} of a pair of skew-symmetric bilinear forms $A$ and $B$ are 

\begin{itemize}

\item sizes of Kronecker blocks,

\item eigenvalues $\lambda_i$, and sizes of Jordan blocks for each eigenvalue $\lambda_i$. 

\end{itemize}

\end{definition}

Thus, we represent every Kronecker $(2k-1)\times (2k-1)$ block with its \textbf{Kronecker size} $2k-1$ (the number $k$ is its \textbf{Kronecker index}). For each eigenvalue $\lambda_i$ the corresponding Jordan blocks with sizes $2n_{i1} \times 2n_{i1},  \dots, 2n_{is_i} \times 2n_{is_i}$ are represented by the corresponding \textbf{Jordan tuple}  $J_{\lambda_i}(2n_{i1}, \dots, 2n_{is_i})$. The JK invariants are the collection of all Jordan tuples and Kronecker sizes. 

\begin{example} Assume that in the JK theorem the blocks for forms $A$ and $B$ are

\begin{itemize}

\item one Jordan $4 \times 4$ block with eigenvalue $\lambda_1$,

\item two Jordan $2 \times 2$ blocks with eigenvalue $\lambda_2 \not = \lambda_1$,

\item two $1\times 1$ Kronecker blocks,

\item one $3 \times 3$ Kronecker block,

\item and one Kronecker $5 \times 5$ block.

\end{itemize} 

Then the JK invariants for $A$ and $B$ consist of $2$ Jordan tuples and $4$ Kronecker sizes:\begin{equation} \label{Eq:JKNote1} J_{\lambda_1} (4), \qquad J_{\lambda_2}(2, 2), \qquad 1, \qquad 1, \qquad 3, \qquad 5.\end{equation} The Kronecker indices (i.e. halves of Kronecker sizes rounded up) are  \[ k_1 = 1, \qquad k_2 = 1, \qquad k_3 = 2, \qquad k_4 = 3. \] 
\end{example} 

\begin{remark} The notation~\eqref{Eq:JKNote1} is not ideal, since it does not empasize that the numbers are Kronecker sizes. We suggest to denote the tuple (i.e. a sequence) of Kronecker sizes as \[K(2k_1 -1, \dots, 2k_q -1).\] Then, in practice, we can write JK invariants~\eqref{Eq:JKNote1} as  \[ J_{\lambda_1} (4), \qquad J_{\lambda_2}(2, 2), \qquad K(1, 1, 3, 5).\] 

\end{remark} 

\begin{remark}  A Kronecker $(2k-1) \times (2k-1)$ block can be represented by any of the following 3 numbers:

\begin{itemize}

\item Kronecker size $2k-1$,

\item Kronecker index $k$,
 
\item minimal index $m = k-1$. 

\end{itemize}

It is a question of convention --- which number to choose for the JK invariants. We use the Kronecker sizes and always write them as $2k-1$. To avoid confusion between the indices $k$, $m = k-1$ and the dimension $2k-1$, we may specify the size  $(2k-1) \times(2k-1)$ of the Kronecker block. For instance, we may say ``JK invariants with a Kronecker $(2k-1) \times(2k-1)$ block''.  \end{remark}

\begin{remark} In terms of \cite{Gantmaher88} (and \cite{Thompson}):

\begin{itemize}

\item Kronecker indices are \textbf{minimal indices} plus 1,

\item halves of sizes of Jordan blocks are \textbf{degrees of elementary divisors}\footnote{More precisely, each Jordan $2n \times 2n$ block with eigenvalue $\lambda_0$ corresponds to a pair of elementary divisors $\left\{ (\lambda - \lambda_0)^n, (\lambda - \lambda_0)^n\right\}$, see also Lemma~\ref{L:JordNum}.}.

\end{itemize} 

There is not much difference between Kronecker indices and minimal indices. Kronecker indices $k_i$ appear as bounds to the degrees of $\operatorname{Ad}^*$-invariant polynomials: \[\operatorname{deg} f_i \geq k_i,\] in A.\,S.~Vorontsov's theorem from \cite{Voron}.

\end{remark}

The characteristic numbers (i.e. eigenvalues) $\lambda_i$ are invarians of pencils $A+ \lambda B$. But we don't need their exact values for the JK invariants. We only care whether $\lambda_i$ are equal or not for different Jordan blocks.  This will be useful for compatible Poisson brackets, when characteristic numbers $\lambda_i(x)$ depend on the point $x$.

\subsubsection{Algebraic type of a pencil}

Below, when considering Poisson pencils on manifolds, we work with linear Poisson pencils $\mathcal{P}_x$ that depend on a paramater $x$. Ofther the number and sizes of all blocks are the same but the eigenvalues depend on the parameter $\lambda_i = \lambda_i(x)$. It is convenient to formalize the notation of JK invariants that are equal``up to exact values of eigenvalues''. 

Let $\mathcal{P}$ be a linear Poisson pencil on a vector space $V$ and let  \begin{equation} \label{Eq:JKInvLinDef} J_{\lambda_1} \left(2n_{11}, \dots, 2n_{1 s_1} \right), \quad \dots, \quad J_{\lambda_p} \left(2n_{p1}, \dots, 2n_{p s_p} \right), \quad 2k_1-1, \quad  \dots, \quad  2k_q-1. \end{equation} be its JK invariants.

\begin{definition} \label{D:AlgTypeLin} The \textbf{algebraic type} of a linear Poisson pencil $\mathcal{P}$ is the multiset of tuples and numbers \[ \operatorname{Type} \left(\mathcal{P} \right) =  \left\{\left(2n_{11}, \dots, 2n_{1 s_1} \right), \quad \dots, \quad \left(2n_{p1}, \dots, 2n_{p s_p} \right), \quad 2k_1-1, \quad  \dots, \quad  2k_q-1 \right\}.\]
\end{definition}

 Each tuple $\left(2n_{i1}, \dots, 2n_{i s_i} \right)$ is an ordered sequence, where \[n_{i1} \geq n_{i2} \geq \dots \geq n_{i s_i}.\]  Note that $\operatorname{Type} \left(\mathcal{P} \right)$ is a multiset, i.e. multiple instances of tuples and numbers are allowed. A tuple with one element $(2n)$ and the number $2n$ are distinguished. For example, the following four algebraic types are pairwise different: \[\left\{ (2, 2), 2 \right\}, \qquad  \left\{ (2), (2), 2 \right\},  \qquad  \left\{ (2), 2, 2 \right\},  \qquad \left\{ 2, 2, 2 \right\}.   \] 
 
Two linear Poisson pencils $\mathcal{P}$ and $\mathcal{P}'$ have the same algebraic type if and only if they have the same Kronecker sizes and there is one-to-one correspondence between their eigenvalues such that the sizes of Jordan blocks for any corresponding characteristic numbers are the same.

\subsection{Computation of Jordan-Kronecker invariants}

Let us state some simple facts that will help us to check that the Lie algebars that we construct below have the desired JK invariants. We provide only some basic techniques to calculate JK invariants, for more information about them see, e.g. \cite{BolsZhang}.

\subsubsection{Jordan blocks}

We start with Jordan blocks. First, if B is nondegenerate $\operatorname{Ker} B  =0$, then we can consider the \textbf{recursion operator} $P = B^{-1}A$. There are no Kronecker blocks, and each Jordan block of  $A, B$ with eigenvalue $\lambda_0$ corresponds to a pair of Jordan blocks of the recursion operator $P = B^{-1} A$ with the same eigenvalue:  \[ A_i + \lambda B_i = \left( \begin{matrix} 0 & J_{\lambda_0 + \lambda} \\ -J_{\lambda_0 + \lambda}^T & 0 \end{matrix}  \right) \qquad \Rightarrow \qquad P_i = B_i^{-1} A_i = \left( \begin{matrix} J_{\lambda_0}^T & 0 \\ 0 & J_{\lambda_0}  \end{matrix}  \right).\]  

The general case can be easily reduced to the nondegenerate case, since the sum of Jordan blocks for $A,B$ is the quotient of the mantle by the core \[ V_J \approx M/K\] and if $B$ is a regular form, then the restriction $B\bigr|_{M/K}$ is nondegenerate.

\begin{lemma} \label{L:JordNum} Assume that $B$ is a regular form of a linear pencil $\mathcal{P} = \left\{ A+\lambda B\right\}$ on $V$.  Take any $\lambda \in \mathbb{C}$ and $n \in \mathbb{N}$. Consider two numbers: 

\begin{itemize}

\item $N$ is the number of Jordan $2n \times 2n$ block with eigenvalue $\lambda$ for the Jordan--Kronecker decomposition of $A, B$.

\item $M$ is the number of Jordan $n \times n$ blocks with eigenvalue $\lambda$ in the Jordan canonical form (JCF) of the recursion operator $P = \left(B\bigr|_{M/K}\right)^{-1} A\bigr|_{M/K}$. 

\end{itemize}

Then $M = 2N$.
\end{lemma}

 In \cite{BolsZhang} the \textbf{characteristic polynomial} $p_{\mathcal{P}}(\lambda)$ of $\mathcal{P} = \left\{ A+\lambda B\right\}$  is defined as follows. Consider all diagonal minors $\Delta_{I}$ of the matrix $A+\lambda B$ of order $ \operatorname{rank} \mathcal{P}$ and take the Pfaffians $\operatorname{Pf}(\Delta_{I})$, i.e. square roots, for each of them. The characteristic polynomial  is the greatest common divisor of all these
Pffaffians: \[ p_{\mathcal{P}}(\lambda) = \operatorname{gcd} \left( \operatorname{Pf}(\Delta_{I})\right).\] The next statement follows from Lemma~\ref{L:JordNum}.

\begin{proposition} Let $\mathcal{P} = \left\{ A+\lambda B\right\}$ be a linear pencil and let

\begin{itemize}

\item $p_{\mathcal{P}}(\lambda)$  be the characteristic polynomial of $\mathcal{P}$,

\item $\det (P - \lambda I)$ be the characteristic polynomial of the recursion operator $P = \left(B\bigr|_{M/K}\right)^{-1} A\bigr|_{M/K}$. 

\end{itemize}

Then $\det (P - \lambda I) = p_{\mathcal{P}}(\lambda)^2$. \end{proposition}

Note that the degree of the characteristic polynomial of $\Pen$ is \[\dim p_{\mathcal{P}}(\lambda) = \frac{1}{2} V_J.\] Below we use the following simple statement.

\begin{proposition} \label{P:DimKM} Let $\mathcal{P} = \left\{ A+\lambda B\right\}$ be a linear pencil on $V$ with core $K$, mantle $M$ and characteristic polynomial $p_{\mathcal{P}}(\lambda)$. Then \[\operatorname{dim} K = \dim V - \frac{1}{2}\operatorname{rk} \mathcal{P} - \deg p_{\mathcal{P}},\qquad  \operatorname{dim} M = \dim V - \frac{1}{2}\operatorname{rk} \mathcal{P} + \deg p_{\mathcal{P}}. \] \end{proposition}

In all the examples below the core $K$, the mantle $M$ and the recursion operator $P$ will be quite easy to find. In general, this may be a nontrivial task.

\begin{remark} Sizes of Jordan blocks can also be defined as elementary divisors  through g.c.d. of minors of $A+\lambda B$ (or Pfaffians of minors, since everything is skew-symmetric). See~\cite{Gantmaher88} for more details.  \end{remark}

\subsubsection{Kronecker blocks} 

If Jordan tuples are, roughly speaking, ``one half of the Jordan canonical form of the recursion operator $P$'',  the Kronecker blocks are all about ``the spread'' of $\operatorname{Ker} \left(A + \lambda B\right)$ for various $\lambda$ and the polynomial solutions of \begin{equation} \label{Eq:KerPencil} \left(A + \lambda B\right)v(\lambda) = 0.\end{equation} Let us denote the set of all its polynomial solutions as $\operatorname{Ker} \mathcal{P}$ (i.e. $\operatorname{Ker} \mathcal{P} \subset \mathbb{C}[\lambda]$). Recall that we know the structure of $\operatorname{Ker} A_{\lambda}$ for each form $A_{\lambda}$ from Proposition~\ref{Prop:KernelJordKronBlocks}. As a consequence we get the following 2 propositions.

\begin{proposition} The number of Kronecker blocks for a linear pencil $\mathcal{P}$ on $V$ is \[\dim V - \operatorname{rk} \mathcal{P} = \dim \operatorname{Ker} A_{\lambda}, \qquad \forall \lambda \text{ - regular}. \]\end{proposition}

\begin{proposition} \label{P:GenerPolyn} For each Kroncker block consider its corresponding polynomial solution $v_i(\lambda)$ of \eqref{Eq:KerPencil} from Proposition~\ref{Prop:KernelJordKronBlocks}. Then $v_i(\lambda)$ are minimal degree generators of the $\mathbb{C}[\lambda]$-module of all polynomial solutions of \eqref{Eq:KerPencil}: \[ \mathbb{C}[\lambda]  v_1(\lambda) + \dots + \mathbb{C}[\lambda] v_{\operatorname{corank} \mathcal{P}}(\lambda)  = \operatorname{Ker} \mathcal{P}.\] \end{proposition}

Hence, we can find the Kronecker indices as in \cite{Gantmaher88} by iteratively taking solutions $x_1(\lambda), x_2(\lambda), \dots$ of \eqref{Eq:KerPencil} so that at each step:

\begin{itemize}

\item $x_i(\lambda)$ are linearly independent,

\item $\operatorname{deg} x_i (\lambda)$ is as small as possible.

\end{itemize}

Then the Kronecker indices are $\operatorname{deg} x_i (\lambda) + 1$. We also have the following simple statement, which is a slight generalization of Corollary 2 from \cite{BolsZhang}.

\begin{corollary} \label{C:DegPol} Let $k_1 \leq k_2 \leq \dots \leq k_q$ be the Kronecker indices of a linear Poisson pencil $\mathcal{P} = \left\{ A + \lambda B\right\}$ on $V$. Suppose that $N$ polynomials \[ v_i(\lambda) = \sum_{j=0}^{m_i} v_{i}^{(j)} \lambda^j, \qquad v_i^{(j)} \in V, \quad i =1, \dots, N,\] are solutions of \eqref{Eq:KerPencil} such that their initial vectors $v_i(0)$ are linearly independant. Assume that the numbers $m_i = \operatorname{deg} v_i(\lambda)$ are ordered so that $m_1 \leq m_2 \leq \dots \leq m_N$. Then \[ m_i \geq k_i, \qquad i=1, \dots, \min\left(q, N\right).\] \end{corollary}

If this way to calculate Kronecker indices turns out to be too hard, sometimes the next corollary from Proposition~\ref{Prop:KernelJordKronBlocks} can help.

\begin{proposition} Take a sufficiently large number of distinct regular values $\lambda_0, \dots, \lambda_N \in \mathbb{C} \cup \left\{ \infty\right\}$. Denote \[ m_k = \dim \left( \sum_{i=0}^k \operatorname{Ker} A_{\lambda_i} \right). \] Then $m_{j+1} - m_j$ is the number of Kronecker blocks with size greater than $(2j +1) \times (2j+1)$. Thus, the number of Kronecker indices $k_i$ equal to $j$ is \[2m_j - m_{j+1} - m_{j-1}.\] Here, formally, $m_{-1} = 0$.\end{proposition}

\subsection{Poisson manifolds}

For the definition and basic facts about Poisson brackets see any standard book on Poisson manifolds, e.g.  \cite{DufourZung05}, \cite{SilvaWeinstein99}. One of the most important facts about Poisson brackets is the well-known Darboux-Weinstein theorem (or Weinstein splitting theorem).

\begin{theorem}[Darboux--Weinstein theorem] \label{T:DarbouxWeinstein} Let $x$ be a point of rank $2k$ of a Poisson manifold $(M^n, \mathcal{A})$. There exist local coordinates $(p^1, \dots, p^k, q^1, \dots, q^k, z^1, \dots, z^{n-2k})$ such that the Poisson bracket is given by 
\[\{p^i, q^j \} = \delta^i_j, \qquad \{z^i, z^j \} =
\varphi^{ij}(z^1, \dots, z^{n-2k}), \] where all other brackets of coordinate functions are zero and $\varphi^{ij} (z(x)) = 0$.
\end{theorem}

In other words, the Poisson bivector has the form \[ \mathcal{A} = \sum_{i=1}^k \frac{\partial}{\partial p^i} \wedge \frac{\partial}{\partial q^i} + \sum_{i<j} \varphi^{ij} (z) \frac{\partial}{\partial z^i} \wedge \frac{\partial}{\partial z^j} \] and the corresponding matrix is \[ \left( \mathcal{A}^{ij} \right) = \left( \begin{matrix} 0 & I_k & 0 \\ -I_k & 0 & 0 \\ 0 & 0 & \Phi(z) \end{matrix} \right), \qquad \Phi(z(x))=0. \] 

\begin{remark}
If $\operatorname{rk} \mathcal{A} = 2k$ in a neighbourhood of $x$, then $ \Phi(z) \equiv  0$ and the Poisson structure is locally flat \begin{equation} \label{Eq:ConstPoissonDarboux} \{p^i, q^j \} = \delta^i_j, \qquad \{z^i, z^j \} =
0. \end{equation}
\end{remark}

Two functions $f$  and $g$ \textbf{commute} or are \textbf{in involution} w.r.t. a Poisson bracket if $ \{f,g\}=0$. A function $f$ is a \textbf{Casimir function} of a Poisson bracket if $\left\{f, g \right\} = 0$ for any function $g$. Note that $z^i$ are local Casimir functions of the Poisson bracket \eqref{Eq:ConstPoissonDarboux}.

\subsubsection{Symplectic leaves.} Any Poisson manifold $(M, \mathcal{A})$ is a union of symplectic manifolds. For any point $x \in (M, \mathcal{A})$ the image $\displaystyle \operatorname{Im} \mathcal{A}_x$ is called the \textbf{characteristic space} at x of the Poisson structure $\mathcal{A}$. In the local coordinates from Theorem~\ref{T:DarbouxWeinstein} the characteristic space is \[  \operatorname{Im} \mathcal{A}_x= \operatorname{span}\left\{ \frac{\partial}{\partial p^1}, \quad \dots, \quad \frac{\partial}{\partial p^k}, \quad \frac{\partial}{\partial q^1}, \quad \dots,  \quad \frac{\partial}{\partial q^k}\right\}. \] For short we denote the characteristic foliation $\displaystyle \mathcal{A}\left(T^*M\right)$ as $\operatorname{Im} \mathcal{A}$. 

\begin{definition} An \textbf{integral submanifold} for $\operatorname{Im} \mathcal{A}$ is a path-connected submanifold $S\subseteq M$ satisfying $T_x S = \mathcal{C}_x$ for all $x\in S$. Integral submanifolds of $\mathcal{A}$ are regularly immersed symplectic manifolds with the symplectic structure given by \[\omega_S(X_f, X_g) (x) =  - \left\{f, g \right\}(x).\] Maximal integral submanifolds of $\mathcal{A}$ are called the \textbf{symplectic leaves} of $\mathcal{A}$.  \end{definition}
 
The symplectic leaves of the bracket \eqref{Eq:ConstPoissonDarboux} (locally) are the level surfaces of the Casimir functions: \[ S = \left\{ z^j = \operatorname{const}, \quad j = 1, \dots, n - 2k\right\}.\]

\subsection{Compatible Poisson brackets} \label{S:CompPoisBrack}

\begin{definition}
Two Poisson structures $\mathcal{A}$ and $\mathcal{B}$ are \textbf{compatible} if any their linear combination $\alpha \mathcal{A} + \beta \mathcal{B}$ with constant coeffitients is also a Poisson structure. 
\end{definition}

If $\mathcal{A}$ and $\mathcal{B}$ are compatible Poisson structures, then we denote \[\mathcal{A}_{\lambda} = \mathcal{A} + \lambda \mathcal{B}, \qquad \mathcal{A}_{\infty} = \mathcal{B}.\] We call a one-parametric family of compatible Poisson brackets \[\mathcal{P} = \left\{ \mathcal{A} + \lambda \mathcal{B} \,\,  \bigr| \,\,\lambda \in \mathbb{K} \cup  \infty  \right\} \] a \textbf{Poisson pencil} or a pencil of Poisson brackets.  Let $\mathcal{P} = \left\{ A_{\lambda}\right\}$ be a Poisson pencil on $M$. The \textbf{rank} of $\mathcal{P}$  \[  \operatorname{rk} \mathcal{P} = \max_{\lambda} \operatorname{rk}\mathcal{A}_{\lambda}.\] Similarly, the rank of $\mathcal{P}$ at a point $x \in M$ is the maximal rank of $\mathcal{A}_\lambda$ at the point $x$. A bracket $A_{\lambda} \in \mathcal{P}$ is \textbf{regular} at a point $x$ if \[ \operatorname{rk} \mathcal{A}_{\lambda}(x) = \operatorname{rk} \mathcal{P}(x).\] Below we use the following trivial fact.

\begin{proposition} \label{P:LocRegBracketPenc}  Let $\Pen = \left\{ \mathcal{A}_{\lambda}\right\}$ be a Poisson pencil on $M$. If $\operatorname{rk} \mathcal{P} = 2k$ in a neighbourhood of $x \in (M, \mathcal{P})$ and a bracket $\mathcal{A}_{\lambda}$ is regular at $x$, then $\mathcal{A}_{\lambda}$ is regular in a sufficiently small neighbourhood of $x$. \end{proposition}

\subsubsection{Regular points for compatible Poisson brackets} \label{S:RegPointPoiss}

Let $\Pen = \left\{ A_{\lambda} = \mathcal{A} + \lambda \mathcal{B}\right\}$ be a Poisson pencil on $M$. At each point $x \in M$ we have a pair of $2$-forms on the cotangent bundle: \[ \mathcal{A}_x, \mathcal{B}_x \in \Lambda^2(T_xM).\] The \textbf{JK invariants} of $\mathcal{P}$ at $x \in M$ are the JK invarians of the linear Poisson pencil $\mathcal{A}_x + \mathcal{B}_x $. 

\begin{definition} \label{D:JKReg} A point $x_0 \in (M,\mathcal{P})$ is \textbf{JK-regular} in a neighbourhood of $Ox_0$ the pencils $\mathcal{P}_x$ have the same algebraic type. \end{definition}

Simply speaking, locally the JK invaraints are ``equal up to eigenvalues'' (see Definition~\ref{D:AlgTypeLin}). In other words, $x_0\in (M, \mathcal{P})$  is JK-regular if in a  neighbourhood of $x_0$ there exists a local frame $v_1(x), \dots,
v_n(x)$  such that the matrices of $\mathcal{A}$ and $\mathcal{B}$ have the block-diagonal form as in the  JK theorem, but the eigenvalues $\lambda_i(x)$ depend on $x\in M$: \begin{equation} \renewcommand*{\arraystretch}{1.2} A_i = \left(\begin{array}{c|c} 0 & J(\lambda_i(x)) \\
\hline - J^{T}(\lambda_i(x)) &
0\end{array} \right), \quad  B_i = \left(\begin{array}{c|c} 0 & E \\
\hline - E& 0\end{array} \right).\end{equation} Note that for JK-regular points the number of distinct eigenvalues $\lambda_i(x)$ locally remains the same. Eigenvalues that are equal at $x_0$ remain equal in a neighbourhood $Ox_0$:\[ \lambda_i(x_0) = \lambda_j(x_0) \qquad \Rightarrow \qquad \lambda_i(x) = \lambda_j(x), \quad x \in Ox_0.\]

\subsection{Jordan--Kronecker invariants of Lie algebras} \label{S:JKLieDef}

Let $\mathfrak{g}$ be a finite-dimensional complex Lie algebra and $\mathfrak{g}^*$ be its dual space.  Fix a basis of $\mathfrak{g}$ and denote by $c^i_{jk}$ the structure constants of $\mathfrak{g}$ in that basis. For a pair of points $x, a \in \mathfrak{g}^*$ consider the skew-symmetric forms \begin{equation} \label{Eq:PoissonMat} \mathcal{A}_x = \left( \sum_i c^i_{jk} x_i \right), \qquad  \mathcal{A}_a = \left( \sum_i c^i_{jk} a_i \right)\end{equation} and the linear Poisson pencil \[ \mathcal{P}_{x, a} = \left\{ \mathcal{A}_x + \lambda \mathcal{A}_a \right\}\] generated by them. There exists a non-empty Zariski open subset \[ U \subset \mathfrak{g}^* \times \mathfrak{g}^*\] such that the algebraic type of the pencil $\mathcal{P}_{x, a}$ is the same for all $(x, a) \in U$. Recall that we defined the algebraic type as a multiset in Definition~\ref{D:AlgTypeLin}. Denote the algebraic type of $\mathcal{P}_{x, a}$ for $(x, a) \in U$ as \begin{equation} \label{Eq:JKSetLie} \operatorname{JK}(\mathfrak{g}) = \left\{\left(2n_{11}, \dots, 2n_{1 s_1} \right), \,\, \dots, \, \,\left(2n_{p1}, \dots, 2n_{p s_p} \right), \, \, 2k_1-1, \, \, \dots, \, \, 2k_q-1 \right\}.\end{equation} In other words, the JK invariants of a pair of skew-symmetric form $\mathcal{A}_x$ and $\mathcal{A}_a$ (given by Definition~\ref{Def:JK}) have the form \begin{equation} \label{Eq:JKInvLie} J_{\lambda_1} \left(2n_{11}, \dots, 2n_{1 s_1} \right), \quad \dots, \quad J_{\lambda_p} \left(2n_{p1}, \dots, 2n_{p s_p} \right), \quad 2k_1-1, \quad  \dots, \quad  2k_q-1. \end{equation} We say that $(x, a) \in U$ is a \textbf{JK-generic pair}. The corresponding pencil $\left\{ \mathcal{A}_x + \lambda \mathcal{A}_a \right\}$ is called \textbf{JK-generic} too.

\begin{definition} \label{D:JKInvLie} The \textbf{Jordan-Kronecker invariants} of a Lie algebra $\mathfrak{g}$ are 
\begin{itemize}

\item sizes of Kronecker blocks,

\item sizes of Jordan blocks grouped by eigenvalues

\end{itemize}
for the JK-generic pencil $ \mathcal{P}_{x, a} =\left\{ \mathcal{A}_x + \lambda \mathcal{A}_a\right\}$, given by \eqref{Eq:PoissonMat}. \end{definition}

In other words, the JK invariants of $\mathfrak{g}$ are the numbers $2k_j-1$ and the tuples $\left(2n_{i1}, \dots, 2n_{i s_i} \right)$. The algebraic type of a generic pencil $\operatorname{JK}(\mathfrak{g})$, given by \eqref{Eq:JKSetLie}, consists of the JK invariants of $\mathfrak{g}$. 

\begin{remark} \label{Rem:JKGenPenc} It is possible to denote the JK invariants of $\mathfrak{g}$ as \eqref{Eq:JKSetLie}. Yet, we use the notation~\ref{Eq:JKInvLie} and regard $\lambda_i$ as distinct formal variables.  We find the expression~\eqref{Eq:JKInvLie} more descriptive.  In particular, it emphasizes that Jordan tuples correspond to different characteristic numbers $\lambda_i$. Also, we often want to work with the JK invariants of JK-generic pencils $ \mathcal{P}_{x, a} =\left\{ \mathcal{A}_x + \lambda \mathcal{A}_a\right\}$. Roughly speaking, \[\text{``JK invariants of $ \mathcal{P}_{x, a}$ = JK invariants of $\mathfrak{g}$ + eigenvalues $\lambda_i(x,a)$''.}\] The eigenvalues $\lambda_i = \lambda_i(x,a)$, which are locally analytic functions for JK-generic pairs $(x,a)$, often come in handy.  So, we need the notation \eqref{Eq:JKInvLie}  anyway.  
\end{remark}

\begin{remark} Note the difference between Definitions~\ref{Def:JK} and \ref{D:JKInvLie}. 

\begin{itemize}

\item The JK invariants of \emph{a pair of forms} $A, B$ contain eigenvalues $\lambda_i$. The eigenvalues $\lambda_i$ are values in $\mathbb{CP} \cup \left\{\infty\right\}$. 

\item The JK invariants of \emph{a Lie algebra} $\mathfrak{g}$ have no eigenvalues $\lambda_i$. Jordan tuples correspond to distinct eigenvalues $\lambda_i$, but they ``do not remember'' their exact values. Each JK-generic pencil $\left\{ \mathcal{A}_x + \lambda \mathcal{A}_a\right\}$ has eigenvalues $\lambda_i= \lambda_i(x,a)$, but they depend on the pair $(x, a)$.

\end{itemize}

\end{remark}

\begin{remark} It seems like the definition of JK invariants is not well established yet. Different authors may change it slightly. 

\begin{itemize}

\item In \cite{BolsZhang} the algebraic type of a generic pencil $\mathcal{A}_x + \lambda \mathcal{A}_a$ was called the JK invariant of $\mathfrak{g}$. Note that the word ``invariant'' is singular, not plural here. In our notations this is the multiset \eqref{Eq:JKSetLie}. This definition, where an ``invariant'' is a ``type'', is valid but slightly counter-intuitive. The Kronecker sizes are also ``invariants'' of $\mathfrak{g}$. It is a question of preference, but the multiset \eqref{Eq:JKSetLie} is rather a ``collection of invariants'', then an ``invariant''. 

\item In \cite{BolsIzosTson} the JK invariants are Kronecker indices $k_i$ and half-sizes of Jordan blocks $n_{ij}$ for each eigenvalue. This definition is very close to Definition~\ref{D:JKInvLie}. We use the term ``tuple'' to group sizes of Jordan blocks together. Unlike \cite{BolsIzosTson}, we use sizes of blocks and not halves of sizes. Otherwise, it would be a hassle to keep in mind that a Jordan $2\times 2$ block is given by the index $1$, and that the Kronecker $3 \times 3$ block is given by the index $2$. All Jordan blocks are even-dimensional and all Kronecker blocks are odd-dimensional --- much easier to remember!

\item In \cite{Vor1}, \cite{Vor2}, \cite{Vor3} and \cite{Vor4} the JK invarians are the sizes of blocks. This is not ideal, there is no need to lose the valuable information --- which Jordan blocks have the same eigenvalue. It is better to group Jordan sizes in tuples and not ``lump them together''.

\end{itemize}

It seems like usually by the JK invariants various authors mean the discrete invariants, i.e. the sizes of Jordan and Kronecker blocks (which they describe ``in various shape or form''). What they probably mean is the multiset \eqref{Eq:JKSetLie}, but the lack of terminology and notations hindered the progress a little bit. \end{remark}

\subsubsection{Index of a Lie algebra}

For Lie algebras we can also describe the number Kronecker blocks as follows.

\begin{itemize}

\item Let $a \in \mathfrak{g}^*$. Denote by $\mathfrak{g}_a$ the the stationary subalgebra
of $a$ in the sense of the coadjoint representation: \[ \mathfrak{g}_a = \left\{\xi \in \mathfrak{g} \,\, \bigr| \,\, \operatorname{ad}^*_{\xi} a = 0 \right\}. \]

\item  In terms of the Lie--Poisson pencil \[ \mathfrak{g}_a = \operatorname{Ker} \mathcal{A}_a.\] 

\item The index of Lie algebra is \[ \operatorname{ind} \mathfrak{g} = \min_{x \in \mathfrak{g}^*} \dim \mathfrak{g}_x. \]

\item  The index of $\mathfrak{g}$ is the number of functionally independent local analytic coadjoint invariants (i.e., Casimir functions). 

\end{itemize}

\begin{proposition}[A.\,V.~Bolsinov, P.~Zhang, \cite{BolsZhang}] \label{P:IndGNumKron} Consider JK invariants of a Lie algebra $\mathfrak{g}$. The number of Kronecker blocks is equal to $\operatorname{ind} \mathfrak{g}$. 

\end{proposition}

\subsubsection{Characteristic numbers and singular set} \label{S:CharNumSing}

Let us recall some facts  from \cite{BolsZhang} about Jordan blocks and characteristic numbers for the pencils $\left\{ \mathcal{A}_x + \lambda \mathcal{A}_a\right\}$, given by \eqref{Eq:PoissonMat}. For a Lie algebra $\mathfrak{g}$, existence and structure of Jordan invariants are closely related to the structure of the \textbf{singular set} $\operatorname{Sing} \subset \mathfrak{g}^*$. Or, more precisely, to the union  $\operatorname{Sing}_0$ of its codimension 1 components. 

\begin{proposition}[A.\,V.~Bolsinov, P.~Zhang, \cite{BolsZhang}] Consider the pencil $\left\{ \mathcal{A}_x + \lambda \mathcal{A}_a\right\}$, given by \eqref{Eq:PoissonMat}. Then $\lambda_i$ is its characteristic number iff  $x - \lambda_i a \in \operatorname{Sing}$. \end{proposition}

\begin{remark} See Remark~\ref{Rem:JordSign} why we take $x- \lambda_i a$ and not $x + \lambda_i a$. \end{remark}

The singular set $\operatorname{Sing} \subset \mathfrak{g}$ is an algebraic variety given by the
system of homogeneous polynomial equations of the form: \[ \operatorname{Pf} C_{i_1,\dots, i_{2k}}  =0, \quad 1 \leq i_1 < i_2 < \dots < i_{2k} \leq \dim \mathfrak{g},\] where 
where $ \operatorname{Pf}$ denotes the Pffafian, and $C_{i_1,\dots, i_{2k}}$ is the diagonal submatrix of the skewsymmetric matrix $\mathcal{A}_x = \left( c_{ij}^k x_k\right)$, related to the rows and columns with numbers
$i_1,\dots, i_{2k}$, where $2k = \dim \mathfrak{g} - \operatorname{ind} \mathfrak{g}$. 

The subset $\operatorname{Sing}_0$ is an algebraic variety defined by one homogeneous polynomial equation --- the greatest common divisor of the  Pfaffians $ \operatorname{Pf} C_{i_1,\dots, i_{2k}} $: \[ \operatorname{Sing}_0 = \left\{ p_{\mathfrak{g}}(x) = 0 \right\}.\] 

\begin{definition}  $p_{\mathfrak{g}}(x)$ is called the \textbf{fundamental semi-invariant}\footnote{In some papers, e.g. in \cite{Ooms10}, the term ``fundamental semi-invariant'' is used for $p_{\mathfrak{g}}(x)^2$.}.  \end{definition} 

We can decompose $p_{\mathfrak{g}}(x)$ into a product of irreducible components\[p_{\mathfrak{g}} (x) = f_1^{s_1}(x) \cdot \dots \cdot f_k^{s_k}(x) \] and consider its reduced version \begin{equation} \label{Eq:FRed}  p_{\mathfrak{g}, \operatorname{red}} (x) = f_1(x) \cdot \dots \cdot f_k(x). \end{equation}

The characteristic numbers can be partitioned into $k$
groups $\Lambda_1, \dots, \Lambda_k$ according to \eqref{Eq:FRed}. Namely, $\Lambda_j$ is is the set of roots of $p_j(\lambda) = f_j( x - \lambda a)$. We reformulate the following statement from \cite{BolsZhang} in terms of Jordan tuples. We say that two Jordan tuples $J_{\lambda_i} \left(2n_{i1}, \dots,  2n_{is_i}\right) $ and $J_{\lambda_j} \left(2n_{j1}, \dots, 2n_{js_j}\right) $ are \textbf{equal} if \[ s_i = s_j, \qquad \text{ and} \qquad n_{it} = n_{jt}, \qquad  t = 1, \dots, s_i. \]

\begin{proposition}[A.\,V.~Bolsinov, P.~Zhang,  \cite{BolsZhang}] \label{Prop:ReducEigen}

\begin{enumerate}

\item The number of distinct characteristic numbers $\lambda_i$  of $\mathfrak{g}$ equals the degree of $p_{\mathfrak{g}, \operatorname{red}}  (x)$. Similarly, the degree of $p_{\mathfrak{g}} (x)$ is the number of characteristic numbers with multiplicities.

\item More precisely, the number of characteristic numbers in each group $\Lambda_i$ is equal to the degree
of $f_i$. 

\item \label{Item:EqualJordTup} For characteristic numbers from the same group $\lambda_j \in \Lambda_i$  the number and sizes of Jordan blocks are the same. In other words, all Jordan tuples $J_{\lambda_j} (2n_{1}, \dots, 2n_{s})$ for $\lambda_j \in \Lambda_i$ are equal.

\item The multiplicity $n_1 + \dots + n_s$  of a characteristic number  $\lambda_j \in \Lambda_i$  is equal to the multiplicity $s_i$ of
$f_i$ in the decomposition~\eqref{Eq:FRed}. In particular, all characteristic numbers within a group $\Lambda_i$  have
the same multiplicity.

\item If some of the characteristic numbers have different Jordan tuples, then the set (variety) of
singular elements $\operatorname{Sing}_0$ is reducible.
\end{enumerate}

\end{proposition}

\begin{remark} All items in Proposition~\ref{Prop:ReducEigen} except for Item~\ref{Item:EqualJordTup} are from Proposition~5 in \cite{BolsZhang}. Item~\ref{Item:EqualJordTup} is claimed in \cite{BolsZhang} after Proposition~5. It is stated that ``it is not hard to show that for characteristic numbers from a fixed group $\Lambda_i$, the structure of Jordan blocks is the same too''.  Roughly speaking, consider one polynomial $f_j(x)$ and the corresponding set $S_j = \left\{ f_j(x) = 0\right\}$. Consider the  lines $y - \lambda a$, where $y \in S_j$, and $a \in \mathfrak{g}$. They divide $S_j$ into a finite number of subsets accoding to the number and sizes of Jordan blocks with eigenvalue $\lambda =0$. One of the subsets contains a non-empty Zariski open subset of $S_j \times \mathfrak{g}$. Thus, for a generic line $y - \lambda a$, where $y \in S_j$ and $a \in \mathfrak{g}$, the pencil  $\mathcal{A}_y - \lambda \mathcal{A}_a$ has the same number and sizes of Jordan blocks with eigenvalue $\lambda =0$ (i.e. ``at point $y$''). We call such pairs $(y, a)$ ``nice''. For a generic line $x - \lambda a$, that intersects $S_j$ at points $ y_i = x - \lambda_i a, (i=1,\dots, d)$, all the pairs $(y_i, a)$ are ``nice''. Note that the lines $x-\lambda a$ and $y_i - \lambda a$ coincide. The Jordan blocks for $\mathcal{A}_x - \lambda \mathcal{A}_a$ with eigenvalue $\lambda_i$ correspond to the Jordan blocks for $\mathcal{A}_{y_i} - \lambda \mathcal{A}_a$ with eigenvalue $0$. By definition, all ``nice'' pairs have the same number and sizes of Jordan blocks with eigenvalue $0$. Thus, all the tuples of Jordan indices $J_{\lambda_i}(2n_{1}, \dots, 2n_{s})$ (i.e. ``at the points $y_i$'') are equal.  \end{remark}

\subsection{Lie--Poisson pencil} \label{S:JKLiePoisDef}

For any finite-dimenstional Lie algebra $\mathfrak{g}$ we have some natural compatible Poisson brackets on its dual space $\mathfrak{g}^*$:

\begin{itemize}

\item A linear bracket (\textbf{Lie-Poisson bracket}): \begin{equation} \label{Eq:LiePoisson}  \{ f, g \}(x) := \langle x, \left[\mathrm{d} f|_{x}, \mathrm{d} g|_{x} \right] \rangle. \end{equation}

\item A constant bracket (so called \textbf{``frozen argument bracket''}). Fix any $a \in \mathfrak{g}^*$. Then the bracket is given by \begin{equation} \label{Eq:Frozen} \{ f, g \}_{a} (x) := \langle a, \left[\mathrm{d} f|_{x}, \mathrm{d} g|_{x} \right] \rangle. \end{equation}

\end{itemize}

We use the canonical isomorphism $(\mathfrak{g}^*)^* = \mathfrak{g}$ in \eqref{Eq:LiePoisson} and \eqref{Eq:Frozen}. We also denote the Poisson structures \eqref{Eq:LiePoisson} and \eqref{Eq:Frozen} as $\mathcal{A}_x$ and $\mathcal{A}_a$. The matrices of these Poisson brackets $\mathcal{A}_x =  \left( \sum_i c^i_{jk} x_i \right)$ and $  \mathcal{A}_a = \left( \sum_i c^i_{jk} a_i \right)$ look similar to the forms \eqref{Eq:PoissonMat}. Yet, there is a difference:

\begin{itemize}

\item In Section~\ref{S:JKLieDef} the elements $x, a \in \mathfrak{g}^*$ are fixed. \eqref{Eq:PoissonMat} defines two skew-symmetric forms.

\item The brackets  \eqref{Eq:LiePoisson} and \eqref{Eq:Frozen} are Poisson brackets on $\mathfrak{g}^*$. The element $a \in \mathfrak{g}^*$ is fixed. $x \in \mathfrak{g}^*$ is a variable. 

\end{itemize}

The Lie-Poisson bracket $\mathcal{A}_x$ is compatible with the constant bracket $\mathcal{A}_a$ for any $a \in \mathfrak{g}^*$. The Poisson pencil \[\mathcal{\Pen}_a = \left\{ \mathcal{A}_x + \lambda \mathcal{A}_a \,\, \bigr| \,\, \lambda \in \mathbb{CP}^1 \right\} \] on $\mathfrak{g}^*$ is called a \textbf{Lie-Poisson pencil}.

\begin{remark} We can consider JK invariants of a Lie--Poisson pencil $\mathcal{P}_a$ at each point $x \in \mathfrak{g}^*$ (see to Section~\ref{S:RegPointPoiss}). 

\begin{itemize}

\item For any $a \in \mathfrak{g}^*$  and any $x \in \mathfrak{g}^*$ the Lie--Poisson pencil $\mathcal{P}_{a}$ at the point $x$ coincides with the pencil $\mathcal{P}_{x, a}$ from Section~\ref{S:JKLieDef}.

\item For generic $a \in \mathfrak{g}^*$  and generic $x \in \mathfrak{g}^*$ the algebraic type of $\mathcal{P}_{x, a}$  is given by the JK invariants of $\mathfrak{g}$. This is the algebraic type of the Lie--Poisson pencil $\mathcal{P}_a$  for generic $a \in \mathfrak{g}^*$.

\item It is possible to study the algebraic type of the Lie--Poisson pencil $\mathcal{P}_a$  for any $a \in \mathfrak{g}^*$. Simply speaking, we fix any $a \in \mathfrak{g}^*$  and take generic $x \in \mathfrak{g}^*$. A similar kind of study is performed in \cite{Gar1} and \cite{Gar2}. We do not study the Lie--Poisson pencils $\mathcal{P}_a$ for arbitrary $a \in \mathfrak{g}^*$   in this paper.

\end{itemize}
\
\end{remark}

\section{Kronecker case} \label{S:Kron}

Our working horse for realization of JK invariants is the following trivial statement.

\begin{theorem} \label{Th:Sum}  JK invariants of a sum of Lie algebras $\mathfrak{g}_1 \oplus \mathfrak{g}_2$ are the union of JK invariants of  $\mathfrak{g}_1$ and $\mathfrak{g}_2$. \end{theorem}

In other words, \begin{equation} \label{Eq:JKUnion} \operatorname{JK} \left( \mathfrak{g}_1  \oplus \mathfrak{g}_2\right)  = \operatorname{JK} \left( \mathfrak{g}_1\right)  \cup \operatorname{JK} \left( \mathfrak{g}_2\right) . \end{equation} Recall that the multiset $ \operatorname{JK} \left( \mathfrak{g}\right)$ is given by \eqref{Eq:JKSetLie}.

\begin{proof}[Proof of Theorem~\ref{Th:Sum}] For any $x, a \in\mathfrak{g}^*$ the pencil $ \mathcal{P}_{x, a} = \left\{ \mathcal{A}_{x + \lambda a} \right\}$ takes the form \[ \mathcal{A}_{x + \lambda a} = \left( \begin{matrix} \mathcal{A}_{x_1 + \lambda a_1} & 0 \\ 0 & \mathcal{A}_{x_2 + \lambda a_2}\end{matrix} \right), \]  where $x_i, a_i \in \mathfrak{g}_i$. Thus, the JK decomposition for $\mathcal{A}_{x + \lambda a}$ is the sum of the JK decompositions of $\mathcal{A}_{x_i + \lambda a_i}$.  It remains to prove that the eigenvalues of $\mathcal{A}_{x_i + \lambda a_i}$ are different for generic $x, a \in \mathfrak{g}^*$. Indeed, the eigenvalues $\lambda_i = \lambda_i(x_i, a_i)$ are functions on different Lie algebras $\mathfrak{g}_i$ and by Lemma~\ref{L:LieEigenNonConst} below they are not constants. Thus, the Jordan blocks for $\mathfrak{g}_i$ have different eigenvalues, and we get the union of multisets~\eqref{Eq:JKUnion}. Theorem~\ref{Th:Sum} is proved. \end{proof}

\begin{remark} \label{Rem:ProdJK} It can be shown similar to Theorem~\ref{Th:Sum}  that for a product of Poisson pencils $\mathcal{P} = \mathcal{P}' \times \mathcal{P}''$ the JK decomposition of $\mathcal{P}_{x', x''}$ is the sum of JK decompositions of $\mathcal{P}'_{x'} \times \mathcal{P}''_{x'}$. \end{remark}

\begin{example} Assume that

\begin{itemize}

\item JK invariants of $\mathfrak{g}_1$ are $J_{\lambda_1}(2n_1, 2n_2)$ and  $2k_1-1, 2k_2-1$,

\item JK invariants of $\mathfrak{g}_2$ are $J_{\lambda'_1}(2n'_1)$ and $2k'_1-1$.

\end{itemize}

Then the JK invariants of $\mathfrak{g}_1 \oplus \mathfrak{g}_2$ are \[J_{\lambda_1}(2n_1, 2n_2),\quad J_{\lambda'_1}(2n'_1), \quad 2k_1-1, \quad 2k_2-1, \quad 2k'_1-1.\] Note that Jordan tuples of $\mathfrak{g}_1$ and $\mathfrak{g}_2$ always correspond to different eigenvalues in $\mathfrak{g}_1 \oplus \mathfrak{g}_2$.  \end{example}

Thus, realization of JK invariants turns into a ``divide and rule'' problem. In the Kronecker case we can realize any Kronecker block, therefore all JK invariants are possible.

\begin{theorem} \label{T:KronInv}
One Kronecker $(2m+1) \times (2m+1)$ block, $m \geq 0$, can be realized by the Lie algebra with the basis $g_1, \dots, g_{m}, h_{0}, \dots, h_{m}$ and the bracket: \[ [g_i, h_0] = h_0, \qquad [g_i, h_i] = h_i, \qquad i=1, \dots, m. \]  All other commutators are equal to $0$. \end{theorem}

In particular, the trivial $1 \times 1$ Kronecker block is realized by the 1-dimensional abelian Lie algebra. In the corresponding coordinates $(x_1,\dots, x_{m}, y_0, \dots, y_{m})$ the matrix of the Lie-Poisson bracket is \[ \mathcal{A}_x = \left( \begin{matrix} 0 & Y \\ -Y^T & 0\end{matrix} \right), \qquad Y = \left( \begin{matrix} y_0 & y_1 &  &  \\ \vdots & & \ddots & \\ y_0 & & & y_{m} \end{matrix}\right).\] There is only one Kronecker block, since for a generic  pair $(x, a)$ we have \[\operatorname{dim} \operatorname{Ker} \left(\mathcal{A}_x + \lambda \mathcal{A}_a\right) = 1, \qquad \forall \lambda \in \mathbb{CP}^1.\] If in the local coordinates $a = (a_1, \dots, a_m, b_0, \dots, b_m)$, then \[ \Ker \left( \mathcal{A}_x + \lambda \mathcal{A}_a\right) = \operatorname{span}  \left\{  \left(\begin{matrix} 0_m & v \end{matrix} \right) \right\}, \quad v = \prod_{j=0}^m (y_j - \lambda b_j) \cdot \left( -\frac{1}{y_0 - \lambda b_0}, \frac{1}{y_1 - \lambda b_1}, \dots, \frac{1}{y_m - \lambda b_m}  \right). \] Note that elements of $v$ are polynomials in $y, b$, since the numerator $ \prod_{j=0}^m (y_j - \lambda b_j)$ contains all the denumerators $y_j - \lambda b_j$. The corresponding Casimir function is \[ f(x, y) = \frac{y_1 \dots y_m}{y_0}.\]

\begin{remark} The Lie algebra $\mathfrak{g}$ from Theorem~\ref{T:KronInv} is solvable. It is a semi-direct sum of abelian lie algebras $\mathbb{C}^m \ltimes \mathbb{C}^{m+1}$. Also note that $\operatorname{span} \left\{ h_i\right\}$ are 1-dimensional ideals. In other words, $y_i$ are linear semi-invariants: \[ \left\{ g(x,y), y_i \right\} = \chi_i (dg) y_i.\] We will use that last fact in Section~\ref{S:GenObst}. \end{remark} 

\begin{remark} \label{Rem:KronAlg} In Theorem~\ref{T:KronInv} we could take any semi-direct product $\mathfrak{g} = \mathbb{C}^m \ltimes \mathbb{C}^{m+1}$ with commutation relations \[ [g_i, h_j] = \chi_j(g_i)  h_j, \qquad i=1, \dots, m, \quad j=0, \dots, m\] where the covectors $\chi_i \in \left(\mathbb{C}^m\right)^*$ are in general position.  In other words, any $m$ of the $m+1$ covectors $\chi_i$ are linearly independant. Similar to Theorem~\ref{T:KronInv}, the JK invariants of $\mathfrak{g}$ consist of one Kronecker $(2m+1) \times (2m+1)$ block. \end{remark}

\begin{remark}  In \cite{BolsZhang} the Kronecker block is realized by another Lie algebra. Probably, there are many other ways to realize Kronecker blocks. We chose the example that we use below. Note that the JK decomposition of a generic linear pencil $\Pen = \left\{ A + \lambda B\right\}$ on an odd-dimensional vector space $V^{2k-1}$ consists of one Kronecker $\left(2k-1\right) \times \left( 2k-1\right)$ block. Thus, one may expect ``a generic'' odd-dimensional Lie algebra to be Kronecker. \end{remark}

\section{Jordan case} \label{S:Jord}

In the Jordan case the Lie--Poisson bracket $\mathcal{A}_x$ is nondegenerate for almost all $x$. That means that $\operatorname{ind} \mathfrak{g} =0$, i.e. in this section we consider Frobenius Lie algebras.

\begin{theorem} \label{Th:JordanCase} Jordan--Kronecker invariants without Kronecker indices can be realized by a Lie algebra if and only if each Jordan tuples $J_{\lambda_i}\left(2n_{i1}, \dots, 2n_{is_i}\right)$ has a unique maximum, i.e. $n_{i1} > n_{ij}$. \end{theorem}

\subsection{Realization of invariants} \label{SubS:JordanRealize}

All possible Lie algebras in the Jordan case were unintentionally described by P.Olver in \cite{Olver} (we obtained matrices of Lie-Poisson brackets \eqref{Eq:OneJordLie} and \eqref{Eq:Jord1} below as slight modifications of some linear matrices for symplectic structures in \cite{Olver}). By Theorem~\ref{Th:Sum} it suffices to realize JK invariants with only one eigenvalue.

\begin{itemize}
\item For one Jordan $2(n+1)\times 2(n+1)$ block, $n \geq 0$, the Lie-Poisson bracket is given in coordinates $(p_0, \dots, p_n, q_0, \dots, q_n)$ by \begin{equation} \label{Eq:OneJordLie} \mathcal{A}_x = \left(\begin{matrix} 0 & P_n(p) \\ -P_n^T(p) & 0 \end{matrix} \right), \qquad P_n(p) = \left( \begin{matrix} p_0 & & & \\ p_1  & p_0 & & \\ \vdots &\ddots & \ddots & \\ p_n & \cdots & p_1 & p_0 \end{matrix}\right), \end{equation} where $p = \left(p_0, \dots, p_n\right)$. For example, for one Jordan $4 \times 4$ block the Lie--Poisson bracket is \[\mathcal{A}_x = \left(\begin{matrix}  0 &  0 & p_0 & 0 \\ 0 & 0 & p_1 & p_0 \\ -p_0 & -p_1 &  0 & 0 \\ 0 & -p_0 & 0 & 0 \end{matrix} \right). \]

\item Now, how to realize one Jordan tuple with unique maximum. This tuple has the form $J_{\lambda} (2n_1+2, 2n_2, \dots, 2n_s)$, where $n_1 \geq n_2 \geq \dots \geq n_s >0$. The local coordinates are \[ \left(p_0, p^1, \dots, p^s, q_0, q^1, \dots, q^s\right),\] where $p^i$ and $q^i$ denote coordinates $p^i_1, \dots, p^i_{n_i}$ and $q^i_1, \dots, q^i_{n_i}$ respectively. The Lie-Poisson bracket is given by \begin{equation} \label{Eq:Jord1} \mathcal{A}_x = \left(\begin{matrix} 0 & \hat{P}(p) \\ -\hat{P}^T(p) & 0 \end{matrix} \right), \quad \hat{P}(p) = \left( \begin{matrix} p_0 &  &  & &\\ p^1 & P_{n_1-1}(\hat{p}^1) &  & & \\ p^2 & 0 & P_{n_2-1} (\hat{p}^2) & &    \\ \vdots & \vdots & \ddots   & \ddots &  \\  p^n  & 0 & \cdots & 0 & P_{n_s-1}(\hat{p}^s) \end{matrix}\right). \end{equation} Here $\hat{p}^i = \left(p_0, p^i_1, \dots, p^i_{n_i-1}\right)$ and the submatrices $P_{n_i-1}$ are as in \eqref{Eq:OneJordLie}. The corresponding Lie bracket also looks nice. 

\end{itemize}

\begin{theorem} \label{T:OneJordRealExample}
One Jordan tuple  $J_{\lambda} (2n_1+2, 2n_2, \dots, 2n_s)$, where $n_1 \geq n_2 \geq \dots \geq n_s >0$, can be realized by the Lie algebra with the basis $e_0, f_0, e^i_j, f^i_j$, where $i=1, \dots, s, \, \, j=1, \dots, n_i$, and the commutation relations are \begin{equation} \label{Eq:JordanRel} [e^i_j, f^i_k ] = e^i_{j-k}, \qquad j \geq k. \end{equation}
We formaly assume $e^i_0 =e_0, f^i_0 =f_0$ for all $i$. \end{theorem}

 For the Lie algebra from Theorem~\ref{T:OneJordRealExample}:
 
 \begin{itemize}
 
 \item the characteristic number \[\displaystyle \lambda = \frac{p_0(x)}{p_0(a)},\]
 
 \item  the singular set \[ \operatorname{Sing} = \left\{p_0 = 0 \right\}.\] 

\end{itemize}

\subsubsection{Proof of Theorem~\ref{T:OneJordRealExample}}

First, let us check the Jacobi idenitity for the Lie bracket given by \eqref{Eq:JordanRel}. The only nontrivial identity is \[ \left[ \left[ e^i_j, f^i_{k_1}\right], f^i_{k_2}\right] + \left[ \left[f^i_{k_1}, f^i_{k_2}\right],  e^i_j\right]  + \left[ \left[  f^i_{k_2}, e^i_j\right], f^i_{k_1}\right]  = e_{j-k_1 -k_2} + 0 - e_{j-k_1 - k_2} = 0. \]  Thus, the Lie algebra form Theorem~\ref{T:OneJordRealExample} is well-defined. It remains to find the number and types of blocks for $\mathcal{A}_x + \lambda \mathcal{A}_a$. This can be easily done using Lemma~\ref{L:JordNum}.  The next statement is trivial.

\begin{proposition} \label{P:JKTwoFormsBlockRedRecJCF} Consider a pair of $2$-forms on a vector space that have a form \[A =  \left(\begin{matrix} 0 & P_A \\ -P_A^T& 0 \end{matrix} \right), \qquad B =  \left(\begin{matrix} 0 & P_B \\ -P_B^T& 0 \end{matrix} \right).\]  Assume that $B$ is nondegenerate $\operatorname{Ker} B = 0$. The JK invariants of $A$ and $B$ are determined by the JCF of the matrix $P^{-1}_{B} P_{A}$. More precisely, for any $\lambda \in \mathbb{C}, n \in \mathbb{N}$ the following numbers of blocks coincide: 

\begin{itemize} 

\item the number of $2n \times 2n$ Jordan blocks with eigenvalue $\lambda$ for the pair $A, B$,

\item the number of $n \times n$ Jordan blocks with eigenvalue $\lambda$ for $P^{-1}_{B} P_{A}$.

\end{itemize}

\end{proposition}

Note that in Proposition~\ref{P:JKTwoFormsBlockRedRecJCF}  the recursion operator \[ B^{-1}A = \left(\begin{matrix}  \left(P_B^{-1}\right)^T P_A^T & 0 \\ 0 & P_B^{-1}P_A \end{matrix} \right) \] and the matrices $P_A\left(P_B^{-1}\right)$ and $P_B^{-1}P_A$ are similar.  Now, let us consider the case of one Jordan block, i.e. $s = 1$. We use the next simple statement.

\begin{proposition} \label{P:JCFTwoDiagMatr} Consider two matrices: \[P_A = \left( \begin{matrix} a_1 & & & \\ a_2 & a_1 & & \\ \vdots &\ddots & \ddots & \\ a_n & \cdots & \cdots & a_1 \end{matrix}\right), \qquad P_B = \left( \begin{matrix} b_1 & & & \\ b_2  & b_1 & & \\ \vdots &\ddots & \ddots & \\ b_n & \cdots & \cdots & b_1 \end{matrix}\right)\] If $b_1 \not = 0$ and $a_2 b_1 \not = a_1 b_2$, then the JCF of $P_{B}^{-1}P_A$ consists of one Jordan $n \times n$ block with eigenvalue $ \displaystyle \lambda = \frac{a_1}{b_1}$. \end{proposition}

By Propositions~\ref{P:JKTwoFormsBlockRedRecJCF} and \ref{P:JCFTwoDiagMatr} for the Lie-Poisson pencil \eqref{Eq:OneJordLie} and a generic pair $(x, a)$ the JK invariants of $\mathcal{A}_x + \lambda \mathcal{A}_a$ consist of one Jordan $2(n + 1) \times 2(n+1)$ block. Now consider the general case, where the Lie-Poisson bracket is given by \eqref{Eq:Jord1}. Denote \[ \mathcal{A}_a =  \left(\begin{matrix} 0 & \hat{P}(a) \\ -\hat{P}(a)^T& 0 \end{matrix} \right)\] By Proposition~~\ref{P:JKTwoFormsBlockRedRecJCF}  it remains to find the JCF of $\hat{P}(a)^{-1} \hat{P}(p)$. The matrix $\left(\hat{P}(a)^{-1} \hat{P}(p) \right)^T$ has the same JCF and the following matrices are similar: \[ \left(\hat{P}(a)^{-1} \hat{P}(p) \right)^T \sim \left(\hat{P}(a)^{T}\right)^{-1} \hat{P}(p)^T.\] Thus, we need to calculate JCF for the matrix $\left(\hat{P}(a)^{T}\right)^{-1} \hat{P}(p)^T$. It remains to prove the following statement.

\begin{proposition} \label{P:OneJordPairOperJCF} Let $n_1 \geq n_2 \geq \dots \geq n_s >0$ and let $V$ be a vector space with basis  $e_0,  e^i_j$, where $i=1, \dots, s, \, \, j=1, \dots, n_i$. For each point $x = (x_0, x^1, \dots, x^s)$, where $x^i$ denote coordinates $x^i_1, \dots, x^i_{n_i}$, consider the matrix  \[M_x = \left( \begin{matrix} x_0 &  x^1 &\dots  &\dots & x^s\\ & P_{n_1-1}^T(\hat{x}^1) &  0 & \cdots & 0 \\ &  & P_{n_2-1}^T (\hat{x}^2) & \ddots & \vdots   \\  & & & \ddots & 0 \\    & & &  & P_{n_s-1}^T(\hat{x}^s) \end{matrix}\right).\] Here $\hat{x}^i = \left(x_0, x^i_1, \dots, x^i_{n_i-1}\right)$, and the submatrices $P_{n_i-1}$ are as in \eqref{Eq:OneJordLie}. Then for a generic pair $(x, a) \in V \times V$ the JCF of the operator $M_x^{-1} M_a$ consist of Jordan blocks with eigenvalue $\lambda = \frac{x_0}{a_0}$ and sizes \[(n_1 + 1) \times (n_1 + 1), \quad n_2 \times n_2, \quad \dots, \quad n_s \times n_s.\] \end{proposition}

\begin{proof}[Proof of Proposition~\ref{P:OneJordPairOperJCF}] Denote $R = M_x^{-1} M_a$. Obviously, for a generic pair $(x, a)$ the operator $R$ is nondegenerate. The proof is in several steps.

\begin{enumerate}

\item The operator $R=M_x^{-1} M_a$ has an upper triangular matrix with  $\lambda = \frac{x_0}{a_0}$ along the main diagonal. Thus, all Jordan blocks of $M_x^{-1} M_a$ have the eigenvalue  $\lambda = \frac{x_0}{a_0}$.

\item The vector space $U_0 = \operatorname{span}(e_0)$ is $R$-invariant and $e_0 \in \operatorname{Ker} (R-\lambda I)$, since \[ M_x e_0 = x_0 e_0, \qquad M_a e_0 = a_0 e_0.\] Next, we use the following simple statement.

\begin{proposition} \label{P:NJordBasis} Let $N$ be a nilpotent linear operator on $V$ and let Jordan blocks of $N$ have sizes $k_1 \geq \dots \geq  k_s$.

\begin{enumerate}

\item Any vector $v \in \operatorname{Ker} N$ can be extended to a Jordan basis of $N$. 

\item Extend $v \in \operatorname{Ker} N$ to a Jordan basis of $N$. Then $v$ lies in a $k_j \times k_j$ Jordan block, where $k_j$ is the maximal $k \in \mathbb{N}$ such that $v \in \operatorname{Im} N^{k-1}$.

\item The Jordan blocks of $N\bigr|_{V/\operatorname{span}(v)}$ have sizes  $k_1, \dots, k_{j-1}, k_j -1, k_{j+1}, \dots, k_s$. 

\end{enumerate}

\end{proposition}

\item The JCF of $R\bigr|_{V/U_0}$ consists of Jordan blocks with eigenvalue $\lambda$ and sizes $n_1, \dots, n_s$. Indeed, the matrices of $M_x\bigr|_{V/U_0}$ are block-diagonal:  \[\left( \begin{matrix}P_{n_1-1}^T(\hat{x}^1) &     &\\      &\ddots &  \\     &  & P_{n_s-1}^T(\hat{x}^s) \end{matrix}\right). \] Hence, we get the sizes of blocks for $R\bigr|_{V/U_0}$ from Proposition~\ref{P:JCFTwoDiagMatr}.

\item It remains to prove that $e_0 \in \operatorname{Im}(R- \lambda I)^{k_1}$. The subspace $U_1 = \operatorname{span}(e_0, e^1_1, \dots, e^1_{n_1})$ is $R$-invariant and the restriction of the operator has the form \[ R\bigr|_{U_1} = \left( \begin{matrix} \lambda & r_1 & \dots & r_{n_1} \\ & \lambda & \ddots & \vdots \\ & & \ddots & r_1 \\ & & & \lambda \end{matrix}  \right), \]  where $r_1 \not = 0$ for a generic pair $(x, a)$. Hence, $e_0 \in \operatorname{Im}(R - \lambda I)^{k_1}$ and by Proposition~\ref{P:NJordBasis} the sizes of Jordan blocks for $R$ are $n_1+1, n_2, \dots, n_s$, as required.

\end{enumerate}

Proposition~\ref{P:OneJordPairOperJCF} is proved. \end{proof}

Theorem~\ref{T:OneJordRealExample} follows from Propositions~\ref{P:JKTwoFormsBlockRedRecJCF} and \ref{P:OneJordPairOperJCF}.

\subsection{Obstruction in the Jordan case} \label{S:ObstJord}

The obstruction is from differential geometry. Namely, it lies in Turiel's theorem about the canonical form for compatible symplectic structures, desribed in \cite{turiel}. In Sections~\ref{S:LocTur1Def}-\ref{S:LocTur5NonConst}. we retell main theorems from that article, slightly reformulating them in terms of Poisson pencils. In Section~\ref{S:CharLiePoisNotConst} we use these theorems to prove Theorem~\ref{Th:JordanCase}.

\subsubsection{Nondegenerate Poisson pencils} \label{S:LocTur1Def}

It is well-known that nondegenerate Poisson structures and symplectic structures are equivalent. Let us recall some definitions. 

\begin{itemize}

\item Two symplectic forms $(\omega_0, \omega_1)$ are compatible if and only if the corresponding Poisson brackets $\omega_0^{-1}$ and $\omega_1^{-1}$ are compatible. 

\item We call a Poisson pencil $\mathcal{P} = \left\{ \mathcal{A} + \lambda\mathcal{B} \right\}$ \textbf{nondegenerate} if there is a nondegenerate bracket $\mathcal{A} + \lambda_0\mathcal{B}, \lambda_0 \in \mathbb{CP}^1$. Then almost all brackets $\mathcal{A} + \lambda\mathcal{B}, \lambda \in \mathbb{CP}^1,$ are nondegenerate. 

\item Let $\mathcal{P} = \left\{ \mathcal{A} + \lambda\mathcal{B} \right\}$ be a nondegenerate Poisson pencil with nondegenerate bracket $\mathcal{B}$. The \textbf{recursion operator} is the field of endomorphisms $P =  \mathcal{A} \mathcal{B}^{-1}$.  For each point $x \in M$ the recusion operator defines an operator \[ P_x : T_x M \to T_x M.\]

\item We say that the \textbf{characteristic polynomial} of the pencil $\mathcal{P}$ is the characteristic polynomial of $
\mathcal{A} \mathcal{B}^{-1}$. 

\item Two pencils are \textbf{isomorpic} if there exists a diffeomorphism that identifies them. 

\item The \textbf{JK-regular points} for Poisson pencils were defined in Definition~\ref{D:JKReg}.

\end{itemize}

\subsubsection{Turiel's factorization theorem} \label{S:LocTur2Factor}

The next theorem reduces the local description to the case of one eigenvalue $\lambda_1$ in the complex case (or to the cases of one real eigenvalue $\lambda_1 \in \mathbb{R}$ or two complex conjugate eigenvalues $\lambda_{1,2} =\alpha\pm i \beta$ in the real case).

\begin{theorem}[F.\,J.~Turiel, \cite{turiel}] \label{Th:Turiel1} Let $\mathcal{P}$ be a nondegenerate Poisson pencil on $M$. For any JK-regular point $x \in \left(M, \mathcal{P} \right)$ there exists a neighbourhood $Ox$ isomorphic to a direct product \[ (Ox, \mathcal{P}) \approx \prod_{i=1}^n (O_i, \mathcal{P}_i),\] where the characteristic polynomials of $\mathcal{P}_i$ are irreducible.\end{theorem}

\begin{remark} Theorem~\ref{Th:Turiel1} immediately follows from the splitting theorem for Nijenhuis operators (see \cite{BolsinovN1}). Two symplectic forms $\omega_0, \omega_1$ are compatible iff the recursion operator $P = \omega_0^{-1} \omega_1$ is a Nijenhuis operator, i.e. its Nijenhuis tensor vanishes $N_P = 0$. For more information about Nijenhuis operators see e.g. \cite{BolsinovN1}.\end{remark}

\subsubsection{Flat nondegenerate Poisson pencils}  \label{S:LocTur3Flat}

Consider a JK-regular point $x_0$ of a Poisson pencil with one eigenvalue $\lambda(x)$. In \cite{turiel} two cases are discussed: \begin{itemize} \item either the eigenvalue $\lambda(x)$ is locally constant ($\lambda = \operatorname{const}$) \item or $x$ is a noncritical point ($d \lambda(x) \not = 0$). \end{itemize}  The first case $\lambda = \operatorname{const}$ is trivial.

\begin{definition} 
A Poisson pencil $\mathcal{P} = \left\{ \mathcal{A}_{\lambda}\right\}$ on $M$ is \textbf{flat} if for any point $x_0 \in M$ there exist local coordinates $x^1, \dots, x^n$ such that all Poisson structures $\mathcal{A}_{\lambda}$ have constant coefficients: \[ \mathcal{A}_{\lambda} = \sum_{i < j} \left( c_{ij} + \lambda d_{ij} \right) \frac{\partial}{\partial x^i} \wedge \frac{\partial}{\partial x^j}, \qquad c_{ij} = \operatorname{const}, \quad d_{ij} = \operatorname{const}.\]
\end{definition} 

\begin{theorem}[F.\,J.~Turiel, \cite{turiel}] \label{Th:Turiel2}  Let $\mathcal{P}$ be a nondegenerate pencil on $M$ such that all points $x \in (M, \mathcal{P})$ are JK-regular. The pencil $\mathcal{P}$ is flat if and only if all its eigenvalues $\lambda_i(x) = \operatorname{const}$. \end{theorem}

\subsubsection{Nondegenerate Poisson pencils with non-constant eigenvalues}  \label{S:LocTur5NonConst}

Now, we consider the case, when the eigenvalues are noncritical. The next theorem follows from Theorem 3 in \cite{turiel}.

\begin{theorem}[F.\,J.~Turiel, \cite{turiel}] \label{Th:Turiel4} Let $\mathcal{P}$ be a nondenerate Poisson pencil on a manifold $M$ with one eigenvalue $\lambda(x)$ such that $d \lambda (x) \not = 0$. Then for a generic JK-regular point the Jordan tuple $J_{\lambda}(2n_{1}, \dots, 2n_{s})$ has a unique maximum, i.e. $n_{1} > n_{j}$.\end{theorem}

\begin{remark} Note the term  ``generic'' in Theorem~\ref{Th:Turiel4}. It means that the statement holds on an open dense subset of $M$. Theorem 3 from \cite{turiel} is a more precise and general statement than Theorem~\ref{Th:Turiel4}. In particular, in that paper:

\begin{enumerate}

\item ``regular'' points, where Theorem~\ref{Th:Turiel4} holds, were defined,

\item   canonical coordinates in a neighborhood of a ``regular'' point were described.

\end{enumerate}

We discuss Theorem~\ref{Th:Turiel4} in more details in Appendix~\ref{S:TurielSympApp}).

\end{remark}

\subsubsection{Characteristic numbers for Lie-Poisson brackets are nonconstant} \label{S:CharLiePoisNotConst} 

In order to apply Theorem \ref{Th:Turiel4} to the JK invariants of Lie algebras we need the following simple fact.

\begin{lemma} \label{L:LieEigenNonConst} Any eigenvalue $\lambda_i(x)$ of a Lie-Poisson pencil $\mathcal{A}_{x+ \lambda a}$ cannot be constant in a neighborhood of any point $x$. \end{lemma}

\begin{proof}[Proof of Lemma~\ref{L:LieEigenNonConst}] If locally $\lambda_i(x) = \operatorname{const}$, then all $x + \lambda a = x + \operatorname{const}$ would belong to a singular set $\operatorname{Sing}$ for all $x$ in some open set. But the singular set $\operatorname{Sing}$ is an algebraic variety of the Lie algebra and cannot contain open regions. \end{proof}

Now we can prove Theorem~\ref{Th:JordanCase}.

\begin{proof}[Proof of Theorem~\ref{Th:JordanCase}] By Lemma~\ref{L:LieEigenNonConst} the eigenvalues for Lie-Poisson pencils are noncritical $d \lambda_i(x) \not = 0$ (for a generic point $x$). Thus, by Theorems~\ref{Th:Turiel1}  and \ref{Th:Turiel4}, if the JK invariants have no Kronecker indices, then each Jordan tuple has a unique maximum. Any collection of Jordan tuples such that each tuple has a unique maximum was realized in Section~\ref{SubS:JordanRealize}. Theorem~\ref{Th:JordanCase} is proved.  \end{proof}

\section{Core and mantle distributions} \label{S:CoreMantleSection}

Before discussing the mixed case for the JK invariants let us introduce several notions and prove several facts about compatible Poisson brackets $\mathcal{A}$ and $\mathcal{B}$ on a manifold $M$. Namely, in this section we 

\begin{itemize} 

\item introduce the core distribution $\mathcal{K}$ and the mantle distribution $\mathcal{M}$ (Definition~\ref{D:CoreMantleDist}),

\item prove their integrability (Propositions~\ref{P:CoreDistInt} and \ref{P:MantleDistInt}),

\item and introduce some useful local coordinates (Theorems~\ref{T:BiPoissRedCoreMantle} and \ref{T:TrivKronFact}).

\end{itemize}

\subsection{Core and mantle distributions} \label{S:CoreMantle}

In this section we discuss distributions on bi-Poisson manifolds $(M, \mathcal{P})$. For the terminology and more details about singular distributions and their integrability see \cite{DufourZung05}. For any distribution $\Delta \subset TM$ we can also consider its dual distribution $\Delta^0 \subset T^*M$. A distribution $\Delta^0$ in the cotangent bundle $T^*M$ will be called integrable (involutive, etc) if $\Delta$ is integrable (respectively, involutive, etc).

\begin{definition} \label{D:CoreMantleDist} Let $\mathcal{P} = \left\{ \mathcal{A}_{\lambda} = \mathcal{A} + \lambda \mathcal{B}\right\}$ be a Poisson pencil on $M$. There are two natural distributions in the cotangent bundle $T^*M$: 

\begin{itemize}

\item the \textbf{core distribution} \begin{equation} \label{Eq:CoreDist} \mathcal{K} =  \bigoplus_{\lambda - \text{regular}} \operatorname{Ker}\mathcal{A}_\lambda,  \end{equation} 

\item the \textbf{mantle distribution} \[ \mathcal{M} = \left( \bigoplus_{\lambda - \text{regular}} \operatorname{Ker} \mathcal{A}_\lambda\right)^{\displaystyle \perp}. \] 

\end{itemize}

 \end{definition}
 
One can also consider the dual distributions $\mathcal{K}^0$ and $\mathcal{M}^0$ in the tangent bundle $TM$. We call  $\mathcal{K}^0$ and $\mathcal{M}^0$ the \textbf{dual-core} and the \textbf{dual-mantle} distributions, respectively. In  \cite{Turiel11} $\mathcal{K}^0$ and $\mathcal{M}^0$ are called a \textbf{primary} and a \textbf{secondary axis}, respectively. Note that \[ \mathcal{K} \subset \mathcal{M} \subset T^*M \qquad \Rightarrow \qquad \mathcal{M}^0 \subset \mathcal{K}^0 \subset TM.\]

So far, $\mathcal{K}, \mathcal{M}$ and $\mathcal{K}^0, \mathcal{M}^0$ are singular distribution, i.e. to each point $x \in M$ we assign a vector subspace of $T^*M$ or $TM$. We denote that subspaces as $\mathcal{K}_x,  \mathcal{M}_x$, etc. Propositions~\ref{P:CoreDistInt}  and \ref{P:MantleDistInt} show that they are integrable regular distributions in a neighborhood of a generic point. First, we give another description for $\mathcal{K}^0$ and $\mathcal{M}^0$ that easily follows from the JK theorem.

\begin{proposition} \label{P:DualCoreMantleProp} Let $\Pen = \left\{ \mathcal{A}_{\lambda} = \mathcal{A} + \lambda \mathcal{B}\right\}$ be a Poisson pencil on $M$. Then for any point $x \in M$ we have the following.

\begin{enumerate}

\item The dual-core distribution is \begin{equation} \label{Eq:DualCoreDist} \mathcal{K}^0 =   \bigcap_{\lambda - \text{regular}} \Imm \mathcal{A}_{\lambda}. \end{equation}

\item Let $\mathcal{A}_{\lambda}$ be a regular bracket at $x$ and let $\omega_{\lambda}$ be the corresponding symplectic form on $\operatorname{Im} \mathcal{A}_{\lambda}$. Then the dual-mantle distribution is \[ \mathcal{M}^0 =  \operatorname{Ker} \left(\omega_{\lambda} \bigr|_{\mathcal{K}^0} \right). \]

\end{enumerate}

\end{proposition}

\subsubsection{Integrability of the core distribution}

\begin{proposition}  \label{P:CoreDistInt}  Let $\Pen = \left\{ \mathcal{A}_{\lambda} = \mathcal{A} + \lambda \mathcal{B}\right\}$ be a Poisson pencil on a manifold $M$ and $p_{\Pen}(x)$ be its characteristic polynomial at $x\in M$. 

\begin{enumerate}

\item If $\operatorname{rk} \mathcal{P}(x) = \operatorname{const}$ in a neighbourhood $Ox_0 \subset (M, \mathcal{P})$, then the dual-core distribution $\mathcal{K}^0$ is an integrable singular smooth distribution in $Ox_0$. 

\item For any point $x \in M$ we have \begin{equation} \label{Eq:CoreDistDim} \begin{gathered} \operatorname{dim} \mathcal{K}_x = \dim M - \frac{1}{2}\operatorname{rk} \mathcal{P}(x) - \deg p_{\Pen} (x), \\ \operatorname{dim} \mathcal{K}^0_x = \frac{1}{2}\operatorname{rk} \mathcal{P}(x) + \deg p_{\Pen} (x). \end{gathered} \end{equation} Thus, if $\operatorname{rk} \mathcal{P}(x) = \operatorname{const}$ and $ \deg p_{\Pen} (x) = \operatorname{const}$ in a  neighbourhood $Ox_0 \subset (M, \mathcal{P})$, then the dual-core distribution $\mathcal{K}^0$ is an integrable regular distribution in $Ox_0$.

\end{enumerate}

\end{proposition}

\begin{proof}[Proof of Proposition~\ref{P:CoreDistInt}]  For any Poisson pencil $\mathcal{A}$ the characteristic distribution $\operatorname{Im} \mathcal{A}$ is an integrable singular smooth distribution. If $\operatorname{rk} \mathcal{A} = \operatorname{const}$ in a  neighbourhood of $x$, then locally $\operatorname{Im} \mathcal{A}$ is an integrable regular distribution. By Proposition~\ref{P:LocRegBracketPenc} and \eqref{Eq:DualCoreDist} the dual-core distribution $\mathcal{K}^0$  is an intersection of a finite number of such distributions $\operatorname{Im} \mathcal{A}_{\lambda_i}$. Formula \eqref{Eq:CoreDistDim}  for the dimension of the core $\mathcal{K}$ follows from Proposition~\ref{P:DimKM}. Proposition~\ref{P:CoreDistInt} is proved. \end{proof}

\begin{remark} The intersection of all characteristic distributions $ \displaystyle \bigcap_{\lambda \in \operatorname{KP}^1} \Imm \left(\mathcal{A} + \lambda \mathcal{B}\right)$ is an integrable singular smooth distribution at any point $x \in M$. This is another natural distribution, we don't study it in this paper. We also do not discuss singularities of the core or the mantle distributions. \end{remark}

\subsubsection{Caratheodory--Jacobi--Lie theorem for Poisson manifolds} 

Next, we want to prove the integrability of the mantle distribution $\mathcal{M}$ (at a generic point). To do it we will need the following variant of Caratheodory--Jacobi--Lie theorem for Poisson manifolds. It is a slight modification of Theorem 2.1 from  \cite{Miranda08}. 

\begin{theorem} \label{T:CaraJacobLiePoisson} Let $(M, \mathcal{A})$ be a Poisson manifold, $\dim M = n$ and $\operatorname{rk} \mathcal{A} = 2k$ on $M$. Assume that 

\begin{itemize} 

\item  $z_1, \dots, z_{n-2k}$ are Casimir functions, i.e. $\left\{f, z_j\right\} = 0$,

\item  $p_1, \dots, p_r$, where $r \leq k$, are smooth functions in involution $\left\{p_i, p_j \right\} = 0$,

\item $dp_1, \dots, dp_r$ and $dz_1,\dots, dz_{n-2k}$ are linearly independent at $x \in M$, i.e. \[ \left(dp_1 \wedge \dots \wedge dp_r \wedge dz_1 \wedge \dots dz_{n-2k}\right)\bigr|_{x} \not = 0.\] 

\end{itemize} Then there exist functions $p_{r+1}, \dots, p_k, q_1, \dots, q_k$ such that $(p_i, q_i, z_j)$ are local coordinates  at $x$ and  \begin{equation} \label{Eq:CaraJracobLieBracket}\mathcal{A} = \sum_{i=1}^k \frac{\partial}{\partial p_i} \wedge \frac{\partial}{\partial q_i}. \end{equation} \end{theorem}

\begin{proof}[Proof of Theorem~\ref{T:CaraJacobLiePoisson}] Since $dp_1, \dots, dp_r, dz_1,\dots, dz_{n-2k}$ are linearly independant and \[\operatorname{Ker} \mathcal{A} = \operatorname{span}\left\{dz_1, \dots, dz_{n-2k} \right\}\] the Hamiltonian vector fields $X_{p_1}, \dots, X_{p_r}$ are linearly independant. By Theorem 2.1 from  \cite{Miranda08} there exists local coordinates $p_1, \dots, p_r, q_1, \dots, q_r, s_1, \dots, s_{n-2r}$ such that \[ \mathcal{A} =  \sum_{i=1}^r \frac{\partial}{\partial p_i} \wedge \frac{\partial}{\partial q_i} + \sum_{i, j = 1}^{n-2r} g_{ij}(s)  \frac{\partial}{\partial s_i} \wedge \frac{\partial}{\partial s_j}. \] It remains to note that $z_j$ are Casimir functions for the Poisson bivector \[\sum_{i, j = 1}^{n-2r} g_{ij}(s)  \frac{\partial}{\partial s_i} \wedge \frac{\partial}{\partial s_j}\] and apply the Darboux-Weinstein theorem~\ref{T:DarbouxWeinstein} for it. Theorem~\ref{T:CaraJacobLiePoisson} is proved. \end{proof}

We need Theorem~\ref{T:CaraJacobLiePoisson} for the following statement. Recall that a distribution in the cotangent bundle $\Delta \subset T^*M$ is integrable if and only if its dual distribution $\Delta^0 \subset TM$ is integrable. 

\begin{corollary} \label{Cor:PoisDualIntDistInt} Let  $(M, \mathcal{A})$ be a Poisson manifold and $\operatorname{rk} \mathcal{A} = 2k$ on $M$. Let $\Delta \subset T^*M$ be a regular isotropic distribution such that  $\operatorname{Ker} \mathcal{A} \subset \Delta$. If $\Delta$ is integrable, then $\Delta^{\perp}$ is also integrable.  \end{corollary}

\begin{proof}[Proof of Corollary~\ref{Cor:PoisDualIntDistInt} ] Take any point $x \in M$.

\begin{itemize}

\item  Let $z_1, \dots, z_{n-2k}$ be local Casimir functions at $x$, i.e. $\left\{f, z_j\right\} = 0$,. 

\item Since $\Delta$ is integrable and  $\operatorname{Ker} \mathcal{A} \subset \Delta$, there exists functions $p_1, \dots, p_r$ such that $dp_1, \dots, dp_r$ and $dz_1,\dots, dz_{n-2k}$ are linearly independent at $x \in M$, i.e. \[ \left(dp_1 \wedge \dots \wedge dp_r \wedge dz_1 \wedge \dots dz_{n-2k}\right)\bigr|_{x} \not = 0,\]and $\Delta$  is locally given by the level sets of the functions $p_i, z_j$, i.e. \[ \Delta = \operatorname{span} \left\{ dp_1, \dots, dp_r,  dz_1, \dots, dz_{n-2k}\right\}.\] 

\item Since $\Delta$ is isotropic, the functions  $p_1, \dots, p_r$ are in involution $\left\{p_i, p_j \right\} = 0$.

\end{itemize}

Thus, we can apply  Theorem~\ref{T:CaraJacobLiePoisson} and get local coordinates \[p_1, \dots, p_k, q_1, \dots, q_k, z_1, \dots, z_{n-2k}\] such that \eqref{Eq:CaraJracobLieBracket} holds. In this coordinates \[\Delta^{\perp} = \operatorname{span}\left\{ dp_1, \dots, dp_k, dq_{r+1}, \dots, dq_k, dz_1, \dots, dz_{n-2k}\right\}.\] Therefore, $\Delta^{\perp}$ is integrable. Corollary~\ref{Cor:PoisDualIntDistInt} is proved. \end{proof}

\begin{remark} In the holomorhpic case the proof remains the same, but one should use holomorphic analogues of some theorems. For instance, instead of the Frobenious theorem one can use the fact that involutive holomorphic subbundles are integrable in the holomorphic sense (see e.g. \cite{Voisin}). \end{remark}

\subsubsection{Integrability of the mantle distribution.} 

Applying Corollary~\ref{Cor:PoisDualIntDistInt} to the core distribution $\mathcal{K}$ we get the following.

\begin{proposition}  \label{P:MantleDistInt}  Let $\Pen = \left\{ \mathcal{A}_{\lambda} = \mathcal{A} + \lambda \mathcal{B}\right\}$ be a Poisson pencil on $M$ and $p_{\Pen}(x)$ be its characteristic polynomial at $x\in M$. 

\begin{enumerate}

\item  For any point $x \in M$ we have \begin{equation} \label{Eq:CoreDistMantle}  \begin{gathered} \operatorname{dim} \mathcal{M}_x = \dim M - \frac{1}{2}\operatorname{rk} \mathcal{P}(x) + \deg p_{\Pen} (x), \\ \operatorname{dim} \mathcal{M}^0_x = \frac{1}{2}\operatorname{rk} \mathcal{P}(x) - \deg p_{\Pen} (x). \end{gathered} \end{equation} 

\item Thus, if $\operatorname{rk} \mathcal{P}(x) = \operatorname{const}$ and $ \deg p_{\Pen} (x) = \operatorname{const}$ in a  neighbourhood $Ox_0 \subset (M, \mathcal{P})$, then the dual-mantle distribution $\mathcal{M}^0$ is an integrable regular distribution in $Ox_0$.

\end{enumerate}

\end{proposition}

Note that formula~\eqref{Eq:CoreDistMantle} for the dimension of the core $\mathcal{M}$ follows from Proposition~\ref{P:DimKM}.

\subsection{Local coordinates for core and mantle}

\begin{theorem}  \label{T:BiPoissRedCoreMantle}  Let $\Pen = \left\{ \mathcal{A}_{\lambda} = \mathcal{A} + \lambda \mathcal{B}\right\}$ be a Poisson pencil on $M$ and $p_{\Pen}(x)$ be its characteristic polynomial at $x\in M$. Assume that \[\deg p_{\Pen} (x) = \operatorname{const}, \qquad \operatorname{rk} \mathcal{P}(x) = \operatorname{const}\] on $M$.  Denote \[ n_J = \dim p_{\Pen}(x_0)  , \qquad m  = \operatorname{rk} \mathcal{P}(x_0) - 2n_J, \qquad r = \operatorname{corank} \mathcal{P}.\]  
Then for any point $x \in M$ there exist local coordinates $x_1,\dots, x_{m}, s_1, \dots, s_{2n_J}, y_1, \dots, y_{m+r}$ such that the core and mantle distribution are \begin{equation} \label{Eq:CoreMantleDistLoc} \mathcal{K} = \operatorname{span}\left\{dy_1, \dots, dy_{m+r} \right\}, \qquad \mathcal{M} = \operatorname{span}\left\{ds_1, \dots, ds_{2n_J}, dy_1, \dots, dy_{m+r} \right\}\end{equation} and the pencil has the form \[\mathcal{A}_{\lambda} = \sum_{i=1}^{m} \frac{\partial}{\partial x_i} \wedge v_{\lambda, i} + \sum_{1 \leq i < j \leq 2 n_J} c_{\lambda, ij}(s, y)  \frac{\partial}{\partial s_i} \wedge \frac{\partial}{\partial s_j} \] for some vector fields $v_{\lambda, i} = v_{\lambda, i}(x, s, y)$ and some functions $c_{\lambda, ij}(s, y)$.  \end{theorem}

Simply speaking, in the coordinatex $(x, s, y)$ from Theorem~\ref{T:BiPoissRedCoreMantle}  the matrices of the Poisson brackets take the form \begin{equation} \label{Eq:CoreMantleMatr} \mathcal{A}_{\lambda} = \left( \begin{matrix} * & * & * \\ * & C_{\lambda}(s, y) & 0 \\ * & 0 & 0\end{matrix} \right), \end{equation}  where $*$ are some matrices. Obviously, the vector fields $v_{\lambda, i} = v_{\lambda, i}(x, s, y)$ and the functions $c_{\lambda, ij}(s, y)$ depend linearly on $\lambda$: \[ v_{\lambda, i} = v_{0, i} + \lambda v_{\infty, i}, \qquad c_{\lambda, ij}(s, y) = c_{0, ij}(s, y) + \lambda c_{\infty, ij}(s, y).\]

\begin{remark} Slightly informaly, in the coordinates $(x, s, y)$ from  Theorem~\ref{T:BiPoissRedCoreMantle} the core and the mantle distributions are \[ \mathcal{K} = \operatorname{span}\left\{ dy\right\}, \qquad \mathcal{M} = \operatorname{span}\left\{ ds, dy\right\}.\] The dual-core and dual-mantle distributions have the form  \[ \mathcal{K}^0 = \operatorname{span}\left\{ \frac{\partial}{\partial x}, \frac{\partial}{\partial s}\right\}, \qquad \mathcal{M}^0 = \operatorname{span}\left\{ \frac{\partial}{\partial x}\right\}.\] 
\end{remark}

\begin{proof}[Proof of Theorem~\ref{T:BiPoissRedCoreMantle}]  The proof is in several steps.

\begin{enumerate} 

\item By Propositions~\ref{P:MantleDistInt} and \ref{P:CoreDistInt} the core $\mathcal{K}$ and mantle $\mathcal{M}$ are regular integrable distributions and $\mathcal{K} \subset \mathcal{M}$. Thus, there exist local coordinates such that $\mathcal{K}$ and $\mathcal{M}$ have the form \eqref{Eq:CoreMantleDistLoc}. 

\item Recall that the core $\mathcal{K}$ is bi-isotropic $\mathcal{A}_{\lambda}\bigr|_{\mathcal{K}} =0$ (see Corollary~\ref{Cor:CoreMantle}), thus \[ \left\{ y_i, y_j \right\}_{\lambda} = 0.\] 

\item By definition, $\mathcal{M} = \mathcal{K}^{\perp}$ and hence \[\left\{s_i, y_j \right\}_{\lambda} = 0. \]
After that step the matrices of Poisson bracktes have the form \[\mathcal{A}_{\lambda} = \left( \begin{matrix} * & * & * \\ * & C_{\lambda}(x, s, y) & 0 \\ * & 0 & 0\end{matrix} \right). \]

\item It remains to prove that $c_{\lambda, ij} = \left\{s_i, s_j\right\}_{\lambda}$ do not depend on $p_1, \dots, p_m$. It follows from the Jacobi identity: \[\left\{ y_k, \left\{ s_i, s_j \right\}_{\lambda} \right\}_{\lambda} = \left\{ \left\{ y_k, s_i, \right\}_{\lambda}  s_j \right\}_{\lambda} +\left\{ s_i, \left\{ y_k,  s_j \right\}_{\lambda}  \right\}_{\lambda} = 0.\]  Consider the Hamiltonian vector fields \[\mathcal{A}_{\lambda} dy_k  = \left\{ y_k, \cdot \right\}_{\lambda}.\]  Using the JK theorem,  it is easy to check that, since $dy_1, \dots, dy_{m+r}$ is a basis of the core $\mathcal{K}$, \[ \operatorname{span} \left\{ \mathcal{A}_{\lambda} dy_1, \dots, \mathcal{A}_{\lambda} dy_{m+r}\right\} = \left\{ \frac{\partial }{\partial  x_1}, \dots, \frac{\partial }{\partial  x_m}\right\}\]  for each $\lambda \in \mathbb{KP}^1$. We get that \[ \frac{\partial \left\{ s_i, s_j \right\}_{\lambda} }{\partial  x_k}  = 0, \qquad k =1, \dots, m, \qquad \forall  \lambda\in \mathbb{KP}^1\] and thus $c_{\lambda, ij} = c_{\lambda, ij}(s, y)$, as required.

\end{enumerate}

Theorem~\ref{T:BiPoissRedCoreMantle} is proved.  \end{proof}

\subsubsection{Eigenvalue decomposition}

We can slightly generalize Theorem~\ref{T:BiPoissRedCoreMantle} and reduce some problems about Poisson pencils $\mathcal{P}$ to the case when 

\begin{itemize}

\item there is only one eigenvalue $\lambda_0$,

\item and all Kronecker blocks of $\mathcal{P}$ are trivial $1\times 1$ blocks.

\end{itemize}

Let $(x, s, y)$ be the coordinates from Theorem~\ref{T:BiPoissRedCoreMantle}. Rouhly speaking, coordinates $s$ correspond to the sum of Jordan blocks $V_J$. Recall that the Jordan blocks can be grouped by eigenvalues: \[ V_J = \bigoplus_{i=1}^p  V_{J_{\lambda_i}}, \qquad   V_{J_{\lambda_i}} = \bigoplus_{j=1}^{s_i} V_{J_{\lambda_i, 2n_{ij}}}.\] Using Turiel's factorization theorem~\ref{Th:Turiel1} we can ``group'' the coordinates $s$ by eigenvalues. 

\begin{remark} We formulate the next theorem for Poisson pencils with trivial $1\times 1$ Kronecker blocks. For a general pencil $\mathcal{P}$, we can apply that theorem for the pencil $\mathcal{P}_{\mathrm{red}}$, given by  \eqref{Eq:KronRedMat}. Simpy speaking, we can use it for the subset $(s,y)$ of the coordinates from Theorem~\ref{T:BiPoissRedCoreMantle}. \end{remark}

\begin{theorem}  \label{T:TrivKronFact} Let $\mathcal{P} = \left\{\mathcal{A} + \lambda \mathcal{B} \right\}$ be a Poisson pencil on a manifold $M$ and let $x_0 \in (M, \mathcal{P})$ be a JK-regular point. Assume that locally the Kronecker blocks of $\mathcal{P}$ are $r$ trivial $1\times 1$  Kronecker blocks. Then in a neighbourhood of $x_0$  there exists local coordinates \begin{equation} \label{Eq:CoordOneEig} (s^1_1, \dots, s^1_{2n_1}, \dots, s^p_1, \dots, s^p_{2n_p}, z_1, \dots, z_r)\end{equation} such that the matrices of Poisson brackets have the form \begin{equation} \label{Eq:FormOneEigen} \mathcal{A}_{\lambda} = \left( \begin{matrix} C^1_{\lambda}(s^1, z) & & & \\ & \ddots & & \\ & & C^p_{\lambda}(s^p, z) & \\ & & & 0_r \end{matrix} \right) \end{equation} and the characterictic polynomials $p_t(\lambda)$ of the pencils $\left\{ C^t_{\lambda}(s^p, z) \right\}$ are irreducible. \end{theorem}

In other words, the pencil in Theorem~\ref{T:TrivKronFact} has the form \[\mathcal{A}_{\lambda} =\sum_{t=1}^p \left( \sum_{1\leq  i < j \leq 2n_t} c^t_{\lambda, ij}(s^t, z)  \frac{\partial}{\partial s^t_i} \wedge \frac{\partial}{\partial s^t_j} \right) \] for some functions $c^t_{\lambda, ij}(s^t_1,\dots, s^t_{2n_t}, z_1, \dots, z_r)$.

\begin{proof}[Proof of Theorem~\ref{T:TrivKronFact}] Since all Kronecker blocks are $1 \times 1$ all regular forms $\mathcal{A}_{\lambda}$ have common (local) Casimir functions $z_1, \dots, z_r$. They also have the same symplectic leaves $(S_z, \omega_{\lambda, z})$, i.e. level sets of Casimir functions: \[ S_z = \left\{ z_1 = \operatorname{const}, \dots, z_r = \operatorname{const}\right\}. \]  We can decompose each symplectic $(S_z, \omega_{\lambda, z})$ using Turiel's factorization theorem~\ref{Th:Turiel1}.  We get coordinates \eqref{Eq:CoordOneEig} such that the matrices of all forms $\omega_{\lambda, z}$ are block-diagonal: \[ \omega_{\lambda, z} =  \left( \begin{matrix} \Omega^1_{\lambda}(s^1, z) & &  \\ & \ddots &  \\ & & \Omega^p_{\lambda}(s^p, z) \end{matrix} \right). \] Since $z_i$ are Casimir function, the pencil $\mathcal{P}$ takes the form \eqref{Eq:FormOneEigen}. Theorem~\ref{T:TrivKronFact} is proved. \end{proof}

\begin{corollary} \label{Cor:OneEignSz} In the coordinates from Theorem~\ref{T:TrivKronFact} the eigenvalue $\lambda_t$, corresponding to the block $C^t_{\lambda}(s^t, z)$, is a function of $s^t$ and $z$, i.e. $\lambda_t = \lambda_t(s^t, z)$. \end{corollary} 

Note that the coordinates $s^t$ correspond to the sum of Jordan blocks with eigenvalue $\lambda_t$. Thus, for a generic Poisson pencil we can reformulate Corollary~\ref{Cor:OneEignSz} as follows.

\begin{corollary} Let $\mathcal{P}$ be a Poisson pencil on $M$ that in a neighbouhood of $x \in M$ satisfies the conditions of Theorem~\ref{T:BiPoissRedCoreMantle} and has characteristic numbers $\lambda_1, \dots, \lambda_p$. Denote by $\mathcal{J}_{\lambda_i}$ the sum of all Jordan blocks with eigenvlaue $\lambda_i$ (it is a distribution in $T^*M$). Then \[ d \lambda_i \in \mathcal{J}_{\lambda_i} \oplus \mathcal{K}.\] The distributions $ \mathcal{J}_{\lambda_i} \oplus \mathcal{K}$ are  pairwise orthogonal w.r.t. all brackets  $\mathcal{A}_{\lambda} \in \mathcal{P}$. Therefore, the characteristic numbers are in involutions w.r.t. all brackets $\mathcal{A}_{\lambda}$: \[ \left\{ \lambda_i, \lambda_j \right\}_{\lambda} = 0.\] \end{corollary}

Below, we prove a more general statement in Lemma~\ref{L:EigenDiff}.

\section{Mixed case, trivial Kronecker blocks} \label{S:TrivKron}

One of the simplest examples of Lie algebras with mixed JK invariants is the  Heisenberg algebra with the basis $e, f, h$ and relation $[e, f] = h$. The Lie-Poisson has the form \[ \mathcal{A}_x = \left( \begin{matrix} 0 & z & 0 \\ -z & 0 & 0 \\ 0 & 0 & 0 \end{matrix} \right), \] and the eigenvalue is $\displaystyle \lambda(x, a) = \frac{\langle x, h \rangle}{\langle a, h \rangle}$.

All possible JK invariants with trivial Kroncker blocks can be realized similar to the Heizenberg algebra. Roughly speaking, one Kronecker block can contain only one eigenvalue of a Jordan tuple with multiple maxima. There are no other limitations. 

\begin{theorem} \label{Th:TrivialKron} The following Jordan--Kronecker invariants with only trivial $1 \times 1$ Kronecker blocks can be realized by Lie algebras:

\begin{enumerate}

\item For $1$ trivial Kronecker block there is no more than $1$ Jordan tuple with multiple maxima $J_{\lambda_i}(2n_{i1}, \dots, 2n_{is_i})$, $n_{i1} = n_{i2} \geq n_{ij}$.

\item For several trivial Kronecker blocks any collection of Jordan tuples is possible. 

\end{enumerate}

\end{theorem}

In Section\ref{S:RealPencOneTriv} we realize several Jordan tuples with multiple maxima and one $1\times 1$ Kronecker block  by a pencil $\mathcal{A} + \lambda \mathcal{B}$, where the Poisson bracket $\mathcal{A}$ is affine (i.e. linear + constant) and $\mathcal{B}$ is constant.

\subsection{Realization}

\begin{enumerate}

\item First, let us realize one Jordan tuple $J_{\lambda}(2n_1, \dots, 2n_s)$ with one trivial Kronecker block. The formulas are similar to \eqref{Eq:Jord1}. The local coordinates are \[ \left(p^1, \dots, p^s, q^1, \dots, q^s, z\right),\] where $p^i$ and $q^i$ denote coordinates $p^i_1, \dots, p^i_{n_i}$ and $q^i_1, \dots, q^i_{n_i}$ respectively. The Lie-Poisson bracket is given by \begin{equation} \label{Eq:JordOneTrivKron} \mathcal{A}_x = \left(\begin{matrix} 0 & P_J(z) & 0 \\ -P_J^T(z) & 0 & 0 \\ 0 & 0 & 0 \end{matrix} \right),\end{equation} where \[ P_J(z) = \left( \begin{matrix} P_{n_1}(z) &  & \\ & \ddots & \\  &  & P_{n_s}(z) \end{matrix}\right),  \qquad P_{n_i}(z) = \left( \begin{matrix} z & & & \\ p_1  & z &  &  \\ \vdots & \ddots & \ddots &  \\ p_{n_i-1} & \cdots  & p_1 & z \end{matrix}\right). \] The commutation relations are similar to \eqref{Eq:JordanRel}. In a basis $e^i_j, f^i_j, h$, where $i=1, \dots, s, \, \, j=1, \dots, n_i$ they are \begin{equation} \label{Eq:BrackOneKronOneJord} [e^i_j, f^i_k ] = \begin{cases} e^i_{j-k}, \qquad j > k,\\ h, \qquad j = k.\end{cases} \end{equation}

 The eigenvalue $\displaystyle \lambda = \frac{z(x)}{z(a)}$. The singular set $\operatorname{Sing}$ is given by $z =0$.

\item By Theorem~\ref{Th:Sum} it remains to realize several Jordan tuples \[J_{\lambda_i}(2n_{i1}, \dots, 2n_{is_i}),\qquad  i = 1,\dots, p\] with two trivial Kronecker blocks. The realization is similar to \eqref{Eq:JordOneTrivKron}. We take a two-dimensional center with coordinates $z_1, z_2$ and add Jordan tuples as in \eqref{Eq:JordOneTrivKron}, replacing $z$ with different linear combinations of $z_1$ and $z_2$:\[\mathcal{A}_x = \left(\begin{matrix} 0 & P_{J_1}(z_1 + z_2) &  & & & &  \\ -P_{J_1}^T(z_1 + z_2) & 0 &  & & & &  \\ & & \ddots & & & \\ & & & 0 & P_{J_p} (z_1 + p z_2) &  &  \\ & & &  -P_{J_p}^T(z_1 + p z_2) & 0 &   \\ & & &  & & 0 & 0 \\  & & &  & & 0 & 0 \end{matrix} \right).\]

\end{enumerate}

\begin{theorem} \label{T:TwoTrivKronJordRealLie}
Two Kronecker $1\times 1$ blocks and $q$ Jordan tuples  $J_{\lambda_t} (2n_{t1}, \dots, 2n_{ts_i})$, $t=1, \dots, p$, can be realized by the Lie algebra $\mathfrak{g}$ with the basis $e^i_{tj}, f^i_{tj}, h_1, h_2$, where $t = 1, \dots, p, i=1, \dots, s_i, \, \, j=1, \dots, n_{ti}$, and the commutation relations are \begin{equation} \label{Eq:BrackOneKronSevJord} [e^i_{tj}, f^i_{tk} ] = \begin{cases} e^i_{t, j-k}, \qquad j > k,\\ h_1 + t h_2, \qquad j = k.\end{cases} \end{equation}\end{theorem}

The eigenvalues are \[\lambda_j = \frac{z_1(x) + j z_2(x)}{z_1(a) + j z_2(a)}, \qquad j = 1, \dots, p.\] The singular set $\operatorname{Sing}$ is given by \[ \prod_{j=1}^p (z_1 + j z_2) =0.\]

\begin{proof}[Proof of Theorem~\ref{T:TwoTrivKronJordRealLie}] It is easy to check the Jacobi identity, we prove a slightly more general statement in Proposition~\ref{P:BrackOneKronSevJordGen} below. Obviously, the elements $h_1$ and $h_2$ belong to the center $Z(\mathfrak{g})$ and correspond to the trivial $1\times 1$ Kronecker blocks. Denote $U = \operatorname{span}(e^i_{tj}, f^i_{tj})$. It follows from  Propositions~\ref{P:JKTwoFormsBlockRedRecJCF} and \ref{P:JCFTwoDiagMatr}  that the JK invariants for $\left(\mathcal{A}_x  + \lambda \mathcal{A}_a\right)\bigr|_{U}$ are $q$ Jordan tuples  $J_{\lambda_t} (2n_{t1}, \dots, 2n_{ts_i})$. Thus, the JK invariants for $\mathfrak{g}$ are these Jordan tuples and two trivial Kronecker blocks. Theorem~\ref{T:TwoTrivKronJordRealLie} is proved. \end{proof}

\subsubsection{Proof of Jacobi identity} \label{S:JacobOneTriv} Let us check the Jacobi identity for these Lie algebras. We consider a slightly more general Poisson bracket, with nonlinear characteristic numbers, and prove the Jacobi identity for it. 

\begin{proposition} \label{P:BrackOneKronSevJordGen} Let $(p^i_{tj}, q^i_{tj}, z_1, \dots, z_d)$, where $t = 1, \dots, p, i=1, \dots, s_t, \, \, j=1, \dots, n_{ti}$ be linear coordinates on a vector space. For any constants $c^i_{tj}$ and functions $f_t(z_1, \dots, z_n)$ the Poisson bracket given by  \begin{equation} \label{Eq:BrackOneKronSevJordGen} \left\{p^i_{tj}, q^i_{tk} \right\} = \begin{cases} p^i_{t, j-k} + c^i_{t, j-k}, \qquad& j > k,\\ f_t(z), \qquad& j = k.\end{cases} \end{equation} is well-defined. The brackets between other coordinate functions are $0$. \end{proposition}

In other words, the matrix of the bracket \eqref{Eq:BrackOneKronSevJordGen} has the form  \[\mathcal{A} = \left(\begin{matrix} 0 & P_{1,1} &  & & &   \\ -P_{1,1}^T & 0 &  & & &   \\ & & \ddots & & & \\ & & & 0 & P_{p,s_p} &   \\ & & &  -P_{p, s_p}^T & 0 &  \\ & & &  & & 0_d \end{matrix} \right),\]  where $0_d$ is the $d\times d$ zero matrix and \[ P_{t,i} = \left( \begin{matrix} f_t(z) & & & \\ p^i_{t1} + c^i_{t1}  & f_t(z) &  &  \\ \vdots & \ddots & \ddots &  \\ p^i_{t, n_{ti}-1} + c^i_{t, n_{ti}-1} & \cdots  & p^i_{t, n_1} + c^i_{t1} & f_t(z) \end{matrix}\right). \]

Obviously, the Lie-Poisson brackets given by  \eqref{Eq:BrackOneKronOneJord} and \eqref{Eq:BrackOneKronSevJord} have the form \eqref{Eq:BrackOneKronSevJordGen}.

\begin{proof}[Proof of Proposition~\ref{P:BrackOneKronSevJordGen} ] The Jacobi identity is proved by direct calculation. Note that any bracket $\left\{f, g\right\}$ is a function of $p$ and $z$. Also $\left\{ f, z_j \right\} = 0$ for any function $f$. Thus, the only nontrivial Jacobi identity is for the coordinate functions $p^i_{tj}, q^i_{tk_1}, q^i_{tk_2}$. Formaly denote \[p^i_{t0} = f_i(z),  \qquad p^i_{tk} = c^i_{tk} = c^i_{t0}=0, \quad k < 0. \]  Note that $\left\{f, p^i_{t0} \right\} =0$ for any function $f$ and that the formula \eqref{Eq:BrackOneKronSevJordGen} formally holds for $p^{i}_{tj}, j \leq 0$. We get the required Jacobi identity: \begin{gather*}\left\{ p^i_{tj} \left\{ q^i_{tk_1}, q^i_{tk_2}  \right\} \right\}  +\left\{ q^i_{tk_1} \left\{ q^i_{tk_2} , p^i_{tj} \right\} \right\}   + \left\{ q^i_{tk_2}  \left\{ p^i_{tj}, q^i_{tk_1}  \right\} \right\}  =  \\ = 0 + (p^i_{t, j- k_1 - k_2} + c^i_{t, j- k_1 - k_2} ) -  (p^i_{t, j- k_1 - k_2} + c^i_{t, j- k_1 - k_2} ) =0.  \end{gather*} Proposition~\ref{P:BrackOneKronSevJordGen}  is proved. \end{proof}

\subsubsection{Realization of JK invariants with one trivial Kronecker block by Poisson pencils} \label{S:RealPencOneTriv}

We cannot realize several Jordan tuples with multiple maxima and one trivial $1\times 1$ Kronecker block by a linear Lie-Poisson bracket $\mathcal{A}_x$ and a constant bracket $\mathcal{A}_a$. We can easily realize these JK invariants, if we loosen the conditions a little bit, and consider a pair of compatible brackets $\mathcal{A} + \lambda{B}$, where $\mathcal{A}$ is affine (i.e. linear + constant) and $\mathcal{B}$ is constant. The idea is simple, we take the consider the brackets from the Section~\ref{S:JacobOneTriv} of the form \[\mathcal{A} = \left(\begin{matrix} 0 & P_{J_1}(f_1(z)) &  & & &   \\ -P_{J_1}^T(f_1(z)) & 0 &  & & &   \\ & & \ddots & & & \\ & & & 0 & P_{J_p} (f_p(z)) &  &  \\ & & &  -P_{J_p}^T(f_p(z)) & 0 &   \\ & & &  & & 0_d  \end{matrix} \right).\]

\begin{theorem} \label{T:AffRealOneKron} Let $(p^i_{tj}, q^i_{tj}, z_1, \dots, z_d)$, where $t = 1, \dots, p, i=1, \dots, s_i, \, \, j=1, \dots, n_{ti}$ be linear coordinates on a vector space $V$. Let  $d_t \not = 0$  and $c^i_{tj}$ be constants, $f_t(z)$ be smooth functions. Consider a pair of compatible Poisson bracket $\left\{ \mathcal{A} + \lambda \mathcal{B} \right\}$ given by \begin{equation} \label{Eq:AffRealOneKronEq} \left\{p^i_{tj}, q^i_{tk} \right\}_{\lambda} = \begin{cases} p^i_{t, j-k} + \lambda c^i_{j-k}, \qquad& j > k,\\ f_t(z)  + \lambda d_t, \qquad& j = k.\end{cases} \end{equation} For a generic point $x \in V$ the JK invariants of $\mathcal{A}\bigr|_{x} + \lambda \mathcal{B}\bigr|_{x}$ are $d$ trivial $1 \times 1$ Kronecker blocks and  $q$ Jordan tuples  $J_{\lambda_t} (2n_{t1}, \dots, 2n_{ts_i})$ with eigenvalue $\displaystyle \lambda_t = \frac{f_t(z)}{d_t}$.
\end{theorem}

\begin{remark} If some eigenvalues $\lambda_t = \frac{f_t(z)}{d_t}$ in Theorem~\ref{T:AffRealOneKron} are equal, then Jordan tuples with the same eigenvalues should be unified. \end{remark} 

\begin{remark} Since we can choise various functions $f_t(z)$ in Theorem~\ref{T:AffRealOneKron}, we can take $d=1$ and $f_t(z) = z + t$. We get a realization of one $1\times 1$ Kronecker block and any collection of Jordan tuples by compatible Poisson brackets. The coefficients of the bracket $\mathcal{A}$ are affine functions, and $\mathcal{B}$ is a constant bracket. 
\end{remark}

\begin{proof} [Proof of Theorem~\ref{T:AffRealOneKron}] By Proposition~\ref{P:BrackOneKronSevJordGen} for each $\lambda \in \mathbb{C} \cup \left\{ \infty\right\}$ the Poisson bracket $\mathcal{A}_{\lambda} = \mathcal{A} + \lambda \mathcal{B}$ given by \eqref{Eq:AffRealOneKronEq} is well-defined. That means that the Poisson brackets $\mathcal{A}$ and $\mathcal{B}$ are compatible. The calculation of JK invariant is similar to the proof of Theorem~\ref{T:TwoTrivKronJordRealLie}. Theorem~\ref{T:AffRealOneKron} is proved. \end{proof} 

\subsection{Obstruction} \label{S:ObstOneKron} 

We need to prove the following.

\begin{theorem} \label{Th:ObstOneKron} There are no Lie algebras with only one Kronecker $1 \times 1$ block and more than one Jordan tuple with multiple maxima. \end{theorem}  

The proof of Theorem~\ref{Th:ObstOneKron} is in 2 steps:

\begin{enumerate}

\item For any Jordan tuple $J_{\lambda}(2n_1, \dots, 2n_s)$ with multiple maxima the differential of its eigenvalue $d\lambda$ lies in the the core $\mathcal{K}$ (Lemma~\ref{L:EigenCore}). Here $\mathcal{K}$ coincides with the $1 \times 1$ Kronecker block.

\item One $1\times 1$ Kronecker block is ``not big enough'' for two eigenvalues of a Lie-Poisson pencil (Lemma~\ref{L:TwoEigen}).

\end{enumerate}

Once again, as for Theorem~\ref{Th:JordanCase}, the obstruction is a combination of a fact from differential geometry and a fact about Lie algebras.

\subsubsection{Eigenvalues of Jordan tuples with multiple maxima}

We want to prove the following.

\begin{lemma} \label{L:EigenCore} Let $\mathcal{P} = \left\{\mathcal{A} + \lambda \mathcal{B} \right\}$ be a Poisson pencil on a manifold $M$ and $x_0 \in (M, \mathcal{P})$ be a JK-regular point. If the Jordan tuple $J_{\lambda_0}(2n_1, \dots, 2n_s)$ of an eigenvalue $\lambda_0$  has multiple maxima (i.e. $n_1 = n_2 \geq n_j$), then \begin{equation} \label{Eq:DLCore} d \lambda_0 \in  \bigoplus_{\lambda - \text{regular}} \operatorname{Ker} \left(\mathcal{A} + \lambda \mathcal{B} \right) = \mathcal{K}.\end{equation} \end{lemma}

We use the next statement, which is a special case of Turiel's decomposition theorem. Namely, it is a particular case of  \cite[Theorem~7.1]{Turiel11} for trivial Kronecker blocks and one eigenvalue.

\begin{theorem}[F.\,J.~Turiel, \cite{Turiel11}]  \label{T:TurielDecompTrivKronOneJord}  Let $\mathcal{P} = \left\{\mathcal{A} + \lambda \mathcal{B} \right\}$ be a Poisson pencil on a manifold $M$ such that all points $x \in M$ are JK-regular and the JK invariants are \begin{itemize}

\item $r$ trivial $1\times 1$  Kronecker blocks,

\item one Jordan tuple $J_{\lambda_0} (2n_1, \dots, 2n_s)$.

\end{itemize}

Also assume that either \[ \lambda_0(x) \equiv \operatorname{const}, \qquad \text{ or } \qquad d \lambda_0(x_0) \not \in \mathcal{K}.\] Here the core $\mathcal{K}$ is the sum of $1\times 1$ Kronecker blocks. Then in a neighbourhood $Ox_0$ of a generic point $x_0\in M$ the pencil $\Pen$ decomposes into a product of a Jordan pencil $\Pen_{Jord}$ and a trivial Kronecker pencil: \[ (Ox_0, \mathcal{P} ) \approx \left(M_{\textrm{Jord}}, \Pen_{\textrm{Jord}}\right) \times \left(M_{\textrm{Kron}}, 0\right).\]  \end{theorem}

\begin{remark} Note that Theorem~\ref{T:TurielDecompTrivKronOneJord}  holds for a generic point $x_0 \in M$, i.e. it holds on an open dense subset of $M$. We discuss Turiel's decomposition theorem  in Appendix~\ref{S:TurielDecompApp}. \end{remark}

 Theorem~\ref{T:TurielDecompTrivKronOneJord} states that for a generic point $x_0 \in M$ there exist local coordinates $s_1, \dots, s_{2n}, z_1, \dots, z_r$ such that the matrices of $\mathcal{A}_{\lambda}$ take the form \[\mathcal{A}_{\lambda} = \left( \begin{matrix}   C_{\lambda} (s) & 0 \\  0 & 0\end{matrix} \right). \] Note that $z_i$ are the common Casimir functions of $\mathcal{A}_{\lambda}$, and that $C_{\lambda} (s) $ does not depend on $z_i$. Compare with Theorem~\ref{T:TrivKronFact}. The next example clarifies Theorem~\ref{T:TurielDecompTrivKronOneJord}.

\begin{example} Consider a Poisson pencil on $\mathbb{C}^3(x,y,z)$ with matrices: \[ \mathcal{A} =  \left( \begin{matrix} 0 & \lambda_0(x,y,z) & 0 \\ -\lambda_0(x,y,z) & 0 & 0 \\ 0& 0 & 0 \end{matrix} \right), \qquad \mathcal{B} =  \left( \begin{matrix} 0 & 1 & 0 \\ -1 & 0 & 0 \\ 0& 0 & 0 \end{matrix} \right) .\] Let us consider 3 possible cases:

\begin{enumerate}

\item $\lambda_0$ is locally const. Then the pencil is flat. (And it obviously decomposes.)

\item $\lambda_0 = \lambda_0(z)$. In other words, $\lambda_0$ is not constant but it is constant on symplectic leaves $ z = z_0$. In this case the pencil does not decompose, it cannot be a product of its Jordan and Kronecker parts.

\item $\lambda_0$ depends not only on $z$ but also on $x$ and $y$ (i.e., it is not constant on symplectic leaves  $z = z_0$). Then by Theorem~\ref{T:TurielDecompTrivKronOneJord} the pencil also locally decomposes into a product of its Jordan and Kronecker parts.

\end{enumerate}

\end{example}

\begin{proof}[Proof of Lemma~\ref{L:EigenCore}] Using Theorems~\ref{T:BiPoissRedCoreMantle}  and \ref{T:TrivKronFact} we reduce the general case to the the case when 

\begin{itemize}

\item there is only one eigenvalue $\lambda_0$,

\item and all Kronecker blocks of $\mathcal{P}$ are trivial $1\times 1$ blocks.

\end{itemize}

If $d \lambda_0 \not\in \mathcal{K}$, then we have the decomposition from Theorem~\ref{L:EigenCore}. We get  a Jordan (i.e. nondegenerate) Poisson pencil $\Pen_{\textrm{Jord}}$ with one Jordan tuple with multiple maxima. We get a contradiction with Theorem~\ref{Th:Turiel4}. Lemma~\ref{L:EigenCore} is proved. \end{proof}

\subsubsection{Independence of eigenvalues of Lie-Poisson pencil}

Let $a \in \mathfrak{g}^*$ and let  $\lambda_1(x)$ and $\lambda_2(x)$ be eigenvalues of a Lie-Poisson pencil $\mathcal{A}_{x+ \lambda a}$. Roughly speaking, we want to prove that for generic $a$ and $x$ the differentials of eigenvalues are linearly independant: \[ d \lambda_2(x) \not = c  \cdot d \lambda_1(x), \qquad c \in \mathbb{C}. \] In this section it is convenient to consider eigenvalues $\lambda_i$ as functions of $x$ and $a$. Recall that in a neighbouhood of a generic pair $(x_0, a_0)$ the eigenvalues $\lambda_i = \lambda_i(x, a)$ are analytic in both $x$ and $a$.  Formally speaking, we want to prove the following.

\begin{lemma} \label{L:TwoEigen} Let $\lambda_1(x,a)$ and $\lambda_2(x,a)$ be two different eigenvalues of a Lie-Poisson pencil $\mathcal{A}_{x+ \lambda a}$. Consider an open subset $U \subset \mathfrak{g}^* \times \mathfrak{g}^*$ such that $\lambda_1(x,a)$ and $\lambda_2(x,a)$ are locally analytic on $U$. Then for a generic pair $(x_0, a_0) \in U$ we have \[ \frac{\partial \lambda_1}{\partial x}(x_0, a_0) \wedge \frac{\partial \lambda_2}{\partial x}(x_0, a_0) \not = 0. \] \end{lemma}

Lemma~\ref{L:TwoEigen} is the combination of the following Corollary~\ref{Cor:EigenPart1} and  Proposition~ \ref{P:EigenNotEqConst}.

\paragraph{Linearly dependant differentials of two eigenvalues coincide.}

\begin{proposition} \label{P:EigenEqUpToConst} Consider a Lie algebra $\mathfrak{g}$, take any $a \in \mathfrak{g}^*$.  Let $\lambda_1(x)$ and $\lambda_2(x)$ be two eiganvalues of the Lie-Poisson pencil $\mathcal{A}_x + \lambda \mathcal{A}_a$ that are locally analytic on $U_0 \subset \mathfrak{g}^*$. If for any point $x_0 \in U_0$ the differentials are linearly dependant \[d \lambda_2(x_0) = c \cdot d \lambda_1(x_0), \qquad c \in \mathbb{C},\] then they coincide: \[d \lambda_2(x_0) =  d \lambda_1(x_0).\] \end{proposition}

Note that Proposition~\ref{P:EigenEqUpToConst} holds for any $a \in \mathfrak{g}^*$, not necessary generic. If for a fixed $a \in \mathfrak{g}^*$ we have $d \lambda_2(x) =  d \lambda_1(x)$, then $\lambda_2 = \lambda_1 + \operatorname{const}$. For various $a \in \mathfrak{g}^*$ that constant is function on $a$.

\begin{corollary} \label{Cor:EigenPart1}  Under the conditions of Lemma~\ref{L:TwoEigen}, if \[ \frac{\partial \lambda_1}{\partial x}(x,a) \wedge \frac{\partial \lambda_2}{\partial x}(x,a) = 0\] on $U$, then \[ \frac{\partial \lambda_1}{\partial x}(x, a) = \frac{\partial \lambda_2}{\partial x}(x, a) \qquad \Leftrightarrow \qquad  \lambda_2(x, a) - \lambda_1(x,a) = g(a).\]  \end{corollary} 

Let us study the differential of a characteristic number $d\lambda(x)$ for a Lie-Poisson pencil  $\mathcal{A}_x + \lambda \mathcal{A}_a$. Consider the reduced polynomial \eqref{Eq:FRed} for $\operatorname{Sing}_0$ (i.e. union of codimension 1 components of singular set). The eigenvalue $\lambda=\lambda(x)$ is a solution of \[f_j(x -\lambda(x)a) = 0\] for some $j$. We can easily find $d \lambda(x)$ by differentiating this expression with respect to $x$. Namely, we get the following. 

\begin{proposition} \label{P:DiffRoot} Let $a \in \mathbb{C}^n$, $f(x)$ be a non-zero homogeneous polynomial on $\mathbb{C}^n$ and $\lambda = \lambda(x)$ be a local analytic function on an open subset $U \subset \mathbb{C}^n$ that satisfies $ f(x - \lambda (x) a) = 0$. Then the following holds:

\begin{enumerate}

\item For any $x \in U$ such that $df( x  - \lambda(x) a) \not = 0$ we have \begin{equation} \label{Eq:DLambda} d \lambda =\frac{df(x - \lambda a) }{\langle df(x - \lambda a) , a\rangle }. \end{equation} 

\item For any $x \in U$ we have \begin{equation} \label{Eq:ProdA}  \langle d\lambda(x), a \rangle = 1.\end{equation} 

\end{enumerate}

\end{proposition}

Note that \eqref{Eq:DLambda} holds for a generic $x \in U$, since $f(x)$ is a polynomial. Therefore, \eqref{Eq:ProdA} holds for all $x \in U$. 

\begin{remark} \eqref{Eq:ProdA} can also be written as \[ \lambda(x + c a) = \lambda(x) + c\] for small $c \in \mathbb{C}$. Another way to prove this is to note that if $x - \lambda a \in \operatorname{Sing}$, then \[ (x + ca) - (\lambda + c) a = x - \lambda a \in \operatorname{Sing}.\] \end{remark}

\begin{proof}[Proof of Proposition~\ref{P:EigenEqUpToConst} ]  By Proposition~\ref{P:DiffRoot},  if $d\lambda_2(x_0) =c \cdot d \lambda_1(x_0)$, then \[c = c \cdot \langle d\lambda_2(x_0),a \rangle =\langle d\lambda_1(x_0),a \rangle = 1.\] Proposition~\ref{P:EigenEqUpToConst}  is proved.  \end{proof}

\paragraph{Two eigenvalues cannot be equal up to a constant.}

\begin{proposition} \label{P:EigenNotEqConst}  Under the conditions of Lemma~\ref{L:TwoEigen}, if $ \lambda_2(x, a) - \lambda_1(x,a) = g(a)$, then $\lambda_2(x, a) = \lambda_1(x, a)$. \end{proposition}

We give two proofs of Propositon~\ref{P:EigenNotEqConst}. 

\begin{proof}[First proof of Proposition~\ref{P:EigenNotEqConst} ] Assume that $g(a) \not \equiv 0$. 
We want to prove that the singular set $\operatorname{Sing}$ contains an open region, that would be a contradiciton. Recall that $x - \lambda_i (x,a) a \in \operatorname{Sing}$. Thus, by taking sufficiently small neighbourhoods, we can assume that there are a smooth hypersurface $S_1 \subset \operatorname{Sing}$ and an open region $W \subset \mathfrak{g}^*$ such that \[y + g(a) a \in \operatorname{Sing}, \qquad \forall y \in S_1, a \in W\] and $g(a) \not =0$ for $a \in W$. Consider the surface \[ S_2  = \left\{ g(a) a \quad \bigr| \quad a \in W\right\}. \] It is easy to see that $\operatorname{dim} S_2 \geq n-1$, since the Jacobian for the map $F(a) = g(a) a$ is \[ J_F = g(a) \cdot I + a \otimes dg.\]  On one hand, the singular set contains \[ S_1 + S_2 = \left\{ x_1 + x_2 \quad \bigr| \quad x_i \in S_i \right\} \subset \operatorname{Sing},\] where $\dim S_i \geq n-1$. On the other hand,  $\dim \operatorname{Sing} \leq n-1$. It follows that $T_{y_1} S_1 = T_{y_2} S_2$ for all $y_i \in S_i$, otherwise $S_1 + S_2$ would contain an open region. We get that $S_1$ and $S_2$ are regions of parallel affine hyperplanes: \[ S_i \subset \left\{ x \quad \bigr| \quad  l(x)  + b_i = 0, \quad l(x) =  l_1 x_1 + \dots l_n x_n.\right\}. \] But the singular set $\operatorname{Sing}$ is a cone, and any region of an affine hyperplane in it must be a region of a linear hyperplane:\[ b_1 = b_2 = 0.\] All elements $g(a)a \in S_2$ must satisfy $l(g(a)a ) = 0$. That means that $g(a) = 0$, if $a$ is not parallel to $S_i$: \[ l(a) \not = 0 \qquad \Rightarrow \qquad g(a) = 0.\]  But $g(a) \not =0$ and we get a contradiction. Proposition~\ref{P:EigenNotEqConst}  is proved. \end{proof}

\begin{proof}[Second proof of Proposition~\ref{P:EigenNotEqConst} ] It suffices to prove the statement for a generic pair $(x, a)$. Without loss of generality, $\lambda_i(x, a)$ are local analytic functions in a neighbourhood of $(x, a)$. 

\begin{itemize}

\item Each eigenvalue $\lambda = \lambda(x, a)$ is a solutions of \begin{equation} \label{Eq:DepLxa} f_j(x -\lambda(x, a)a) = 0\end{equation} for some factor $f_j(x)$ of the reduced  polynomial  \eqref{Eq:FRed}. 

\item Since the polynomial $f_j(x)$ is homogeneous,   we have  \[ \lambda(cx, ca) = \lambda(x, a) \] for any $c \not =0$.  Thus, \[\lambda_2(x, a) - \lambda_1(x, a) = g(a) \qquad \Rightarrow \qquad g(ca) = 0.\] 

\item Differentiating \eqref{Eq:DepLxa} with respect to $a$ we get \[ d_a \lambda = - \lambda \frac{df(x - \lambda a) }{\langle df(x - \lambda a) , a\rangle } \]  if $df(x -\lambda(x,a) a ) \not =0$. In particular, for a generic pair $(x, a)$ we get \begin{equation} \label{Eq:MultA}  \langle d_a\lambda, a \rangle = - \lambda. \end{equation} 

\item We have \[ d_a \lambda_2 - d_a \lambda_1 = d_a g(a).\] Using \eqref{Eq:MultA} we get \[ - g(a)  = \langle d_a g(a), a\rangle. \] Note that the right-hand side is a derivative "in the radial direction" \[\displaystyle \langle d_a g(a), a\rangle = \sum_i a_i \frac{\partial g}{\partial a_i}.\] It vanishes $\langle d_a g(a), a\rangle =  0$, since $g(ca) = g(a)$. We have proved that $g(a)  = 0$.

\end{itemize}

Proposition~\ref{P:EigenNotEqConst}  is proved. \end{proof}

\section{Mixed case. Realization of JK invariants} \label{S:RealGenMix}

In this section we realize some JK invariants that contain both Jordan and Kronecker blocks. Below we consider the JK invariants with one Kronecker block. The JK invariants with several Kronecker blocks can then be realized by Theorem~\ref{Th:Sum}. The case of one trivial $1\times 1$ Kronecker block was discussed in Section~\ref{S:TrivKron}.  The idea, how to contruct Lie--Poisson pencils with one Kronecker block, is simple:

\begin{itemize}

\item The matrix of the Lie--Poisson bracket will have the form \begin{equation} \label{Eq:LieBrackEx} \mathcal{A}_{x} = \left( \begin{matrix} 0 & A_{xs} & A_{xy} \\ -A_{xs}^T & A_{ss} & 0 \\ -A_{xy}^T & 0 & 0\end{matrix} \right).\end{equation}

\item  The submatrix $\left( \begin{matrix} 0 & A_{xy} \\ -A_{xy}^T & 0\end{matrix} \right)$ corresponds to the Kronecker block and it will be the Lie--Poisson bracket from Theorem~\ref{T:KronInv}.

\item  The submatrix $\left( \begin{matrix} A_{ss} & 0 \\  0 & 0\end{matrix} \right)$ corresponds to the mantle (i.e. the sum of Jordan blocks and the core) and it will be the Lie--Poisson bracket from Theorem~\ref{T:TwoTrivKronJordRealLie}.

\item Finally, we take the matrix $A_{xs}$ such that the Jacobi identity is satisfied and $\mathcal{A}_x$ is a Lie--Poisson bracket.

\end{itemize}

In Sections~\ref{S:Kron2} and \ref{S:GenMixLie} we consider the cases of the Kronecker $3 \times 3$ and $(2m+1) \times (2m+1)$ blocks, where $m >1$, respectively.

\subsection{JK invariants with the Kronecker $3 \times 3$ block} \label{S:Kron2} 

If there is a Kronecker $3\times 3$ block, then all JK invariants are possible. With Theorems~\ref{Th:Sum} and \ref{T:KronInv} it remains to prove the following statement.

\begin{theorem} \label{Th:Kron3}
The Jordan-Kronecker invariants with one Kronecker $3\times 3$ block and any collection of Jordan tuples $J_{\lambda_s}(2n_{s1}, \dots, 2n_{sp_s})$ ($s=1, \dots N$) can be realized by a Lie algebra with basis $g, h_1, h_2, e^i_{tj}, f^i_{tj}$, where $t=1,\dots, N, i=1,\dots p_s, j=1,\dots, n_{s i}$ and commutation relations \begin{equation} \label{Eq:Kron3} \begin{gathered}{}
[g, h_1] = h_1, \qquad [g, h_2] = h_2, \qquad [g, e^i_{tk}] = e_{tk}^i,  \\  [e_{tj}^i, f_{tk}^i] = \begin{cases} e^i_{t, j-k}, \quad j>k, \\ \alpha_t h_1 + \beta_t h_2, \quad j = k,\end{cases} \end{gathered} \end{equation} where $(\alpha_t, \beta_t) \in  \mathbb{C}^2 - 0$ are constants, such that $(\alpha_t : \beta_t) \in \mathbb{CP}^1$ are pairwise different. \end{theorem}

\begin{remark}  The eigenvalues of the pencil from Theorem~\ref{Th:Kron3} are given by \eqref{Eq:EigenKron2}. It is easy to see that for proportional pair $(\alpha_i, \beta_i)  =  c (\alpha_j, \beta_j)$ the eigenvalues are equal $\lambda_i = \lambda_j$. Therefore, in Theorem~\ref{Th:Kron3} we demand $(\alpha_t: \beta_t)$ to be different elements  of $\mathbb{CP}^1$. 

\end{remark} 

Let us illustrate Theorem~\ref{Th:Kron3} for one $2n \times 2n$ Jordan block and the Kronecker $3\times 3$ block. The Lie-Poisson bracket in the coordinates \[(x, p_1, \dots, p_n, q_1,\dots, q_n, y_1, y_2)\] has the form \begin{equation} \label{Eq:3t3KronAnyJordPois} \mathcal{A}_x = \left(\begin{array}{c|cc|cc} 0 &  \hat{p} & 0  & y_1 & y_2 \\ \hline
 -\hat{p}^T & 0 & P_n(\alpha y_1 + \beta y_2) & 0 & 0 \\ 
0 & -P_n^T(\alpha y_1 + \beta y_2) & 0 & 0 & 0 \\ \hline
-y_1  & 0 & 0 & 0 & 0 \\ 
-y_2    & 0 & 0 & 0 & 0 \\ 
\end{array} \right),\end{equation} where \[P_n(z) =  \left( \begin{matrix} z & & & \\ p_1  & z &  &  \\ \vdots & \ddots & \ddots &  \\ p_{n-1} & \cdots  & p_1 & z \end{matrix}\right), \qquad \hat{p} = \left(p_1, \dots, p_n\right). \]

\begin{remark} Note that Poisson brackets \eqref{Eq:3t3KronAnyJordPois} have the form \eqref{Eq:CoreMantleMatr} as in Theorem~\ref{T:BiPoissRedCoreMantle}. Below we show that the core $K=\operatorname{span} \left( h_1, h_2\right)$ and the mantle $M = \operatorname{span} \left( h_1, h_2, e^i_{sj}, f^i_{sj}\right)$.
\end{remark} 

We add the commutation relation $[g, e^i_{sk}] = e_{sk}^i$ to get a Lie algebra. Other Jordan blocks are ``attached'' to the Kronecker block in a similar fashion.

\begin{remark} Denote $\mathfrak{g}$ the Lie algebra from Theorem~\ref{Th:Kron3} and $\mathfrak{h} = \operatorname{span}(h_1, h_2, e^i_{sj}, f^i_{sj})$. Then $\mathfrak{h}$ is an ideal of $\mathfrak{g}$ and it is the Lie algebra from Theorem~\ref{T:TwoTrivKronJordRealLie}. In other words, $\mathfrak{g}$ is a one-dimensional extension of $\mathfrak{h}$: \[ 0 \to \mathfrak{h} \to \mathfrak{g} \to \mathbb{C} \to 0.\] \end{remark}

\begin{proof}[Proof of Theorem~\ref{Th:Kron3}] It is not hard to check the Jacobi identity for \eqref{Eq:Kron3}. It suffices to check the identity only for the basis vectors. 

\begin{itemize}

\item We know that $\mathfrak{h} = \operatorname{span}(h_1, h_2, e^i_{sj}, f^i_{sj})$ is the Lie algebra from Theorem~\ref{T:TwoTrivKronJordRealLie}, hence we don't have to check the Jacobi identity without $g$. 

\item It is clear that  $\mathfrak{h}_1 = \operatorname{span}(g, h_1, h_2, f^i_{sj})$ is a Lie algebra, thus we don't have to check the Jacobi identity without $e^i_{sj}$.

\item It remains to check the Jacobi identity for $g, e^i_{sj}, f^i_{sk}$, for other remaining triplets all terms in the identity are zero:  \[ [g, [ e^i_{sj}, f^i_{sk}]]  + [ e^i_{sj}, [f^i_{sk}, g]] + [ f^i_{sk}, [g, e^i_{sj}]] = 
[ e^i_{sj}, f^i_{sk}] + 0 - [ e^i_{sj}, f^i_{sk}]  = 0. \]

\end{itemize}

Thus, \eqref{Eq:Kron3} indeed defines a Lie algebra. It remains to calculate the Jordan--Kronecker invariants.

\begin{enumerate}

\item We have $\dim \operatorname{Ker}(\mathcal{A}_{x+\lambda a} )= 1$ for almost all $\lambda$, thus there is only 1 Kronecker block.

\item It is easy to see that \[ \bigoplus_{\lambda - \textrm{ regular }} \operatorname{Ker}(\mathcal{A}_{x+\lambda a}) = \operatorname{span} \left( h_1, h_2\right),\] thus the  Kronecker block is $3 \times 3$ and core subspace \[K= \operatorname{span} \left( h_1, h_2\right).\]

\item  The mantle subpsace is \[ M = \operatorname{span} \left( h_1, h_2, e^i_{sj}, f^i_{sj}\right).\] The Jordan blocks for  $\mathcal{A}_x + \lambda \mathcal{A}_a$  and for their restrictions on $M$ coincide. For the restriction $\left( \mathcal{A}_x + \lambda \mathcal{A}_a\right)\bigr|_{M}$  the JK invariants are known from Section~\ref{S:TrivKron} (see also Propositions~\ref{P:JKTwoFormsBlockRedRecJCF}  and \ref{P:JCFTwoDiagMatr}). Namely, each subspace $\operatorname{span} \left( e^i_{tj}, f^i_{sj}\right)$, where $i, t$ are fixed and $j=1,\dots, n_{t p_t}$, correspond to a Jordan $2n_{t p_t} \times 2n_{t p_t}$ Jordan block with eigenvalue \begin{equation} \label{Eq:EigenKron2} \lambda_t =  \frac{\alpha_t \langle h_1, x\rangle  + \beta_t \langle h_2, x\rangle }{\alpha_t \langle h_1, a\rangle  + \beta_t \langle h_2, a\rangle}.\end{equation}

\end{enumerate}

We get the required JK invariants. \end{proof}

\subsection{General mixed case} \label{S:GenMixLie}

Let Kronecker indices be $k_1,\dots, k_q$ (i.e. the sizes of Kronecker blocks are $(2k_j-1) \times (2k_j-1)$). Then we can realize $N$ Jordan tuples, where $ N \leq k_1 + \dots + k_q$, similar to Section~\ref{S:Kron2}. By Theorem~\ref{Th:Sum} it suffices to realize one Kronecker $(2m+1) \times (2m+1)$ block and $d$ Jordan tuples, where $d \leq m+1$. We do it in the next Theorem~\ref{Th:RealOneKronSeveralEigen}. 

First, let us illustrate this theorem for one $2n \times 2n$ Jordan block  and Kronecker $(2m+1) \times (2m+1)$ block. The Lie-Poisson bracket in the coordinates \[(x_1, \dots, x_m, p_1, \dots, p_n, q_1,\dots, q_n, y_0, \dots, y_m)\] has the form  \[\mathcal{A}_x = \left(\begin{array}{c|cc|c} 0 &  P_s & 0  & Y \\ \hline -P_s^T & 0 & P_n(y_s) & 0  \\ 0 & -P_n^T(y_s) & 0 & 0 \\ \hline-Y^T  & 0 & 0 & 0  \\ \end{array} \right),\]  where  \[ Y = \left( \begin{matrix} y_0 & y_1 &  &  \\ \vdots & & \ddots & \\ y_0 & & & y_{m} \end{matrix}\right), \qquad P_n(y_s) =  \left( \begin{matrix} y_s & & & \\ p_1  & y_s &  &  \\ \vdots & \ddots & \ddots &  \\ p_n & \cdots  & p_1 & y_s \end{matrix}\right). \]

The eigenvalue $\lambda_s$ corresponds to the vector $h_s$. In order to the Jacoby identity to hold, we take the following matrix $P_s$: 

\begin{itemize} 

\item If $s\not =0$, then all rows of $P_s$ except for i-th are zeros: \[ P_s = \left(\begin{matrix} 0 \\ \vdots \\ p \\ \vdots \\ 0\end{matrix}  \right), \qquad p=(p_1, \dots, p_n).\]  

\item For $s=0$ all rows of $P_0$ are the same: \[ P_0 = \left(\begin{matrix} p \\ \vdots \\ p \end{matrix}  \right), \qquad p=(p_1, \dots, p_n).\]

\end{itemize}

\begin{remark} \label{Rem:BlockXS1} Note that the functions $y_i$ are semi-invariants. There exist vectors $\chi_0, \dots, \chi_m \in \mathbb{C}^m$ such that \[ \left( \begin{matrix} \left\{ x_1, y_i\right\} \\ \dots \\ \left\{ x_m, y_i\right\} \end{matrix} \right)  = y_i \cdot \chi_i. \] The matrices $P_i$ are then given by \[P_i = \chi_i^T p.\] Here we regard $\chi_i$ and $p$ as a $1 \times m$ and $1 \times n$ vectors (the matrix $P_i$ is $m \times n$). The Lie--Poisson pencils in this section is a special case of the Poisson pencils constructed in Section~\ref{S:RealPenc}  (compare with Remark~\ref{Rem:GMatGenPois}).  \end{remark}

For $d \leq m+1$ Jordan tuples $J_{\lambda_t}(2n_{t1}, \dots, 2n_{ts_t})$ we can simultaneously add Jordan blocks in a similar way for $h_0, \dots, h_{d-1}$. In other words, we take the Lie--Poisson bracket of the form 
\begin{equation} \label{Eq:LiePoissonSevJornOneMKron} \mathcal{A}_x = \left(\begin{array}{c|ccccc|c} 0 & P_{0,1} & 0 & \cdots & P_{ d-1,s_d} & 0 & Y \\ \hline -P_{0, 1}^T & 0 & P_{n_{1,1}}(y_0) &  & & &   \\ 0 & -P_{n_{1,1}}^T(y_0) & 0 &  & & &   \\ \vdots & & & \ddots & & & \\ - P_{d-1,s_d}^T & & & & 0 & P_{n_{d, s_d}}(y_{d-1}) &   \\ 0 & & & &  -P_{n_{d, s_d}}(y_{d-1})^T & 0 &  \\ \hline -Y^T & & & &  & & 0 \end{array} \right).\end{equation}  The next theorem can be proved similarly to Theorem~\ref{Th:Kron3}.

\begin{theorem} \label{Th:RealOneKronSeveralEigen} The Jordan-Kronecker invariants with one Kronecker $(2m+1) \times (2m+1)$ block and $d$ Jordan tuples $J_{\lambda_t}(2n_{t1}, \dots, 2n_{ts_t})$ ($t=1, \dots d$), where $d \leq m +1$, can be realized by a Lie algebra with basis $g_1, \dots, g_{m}, h_0, \dots, h_{m}, e^i_{tj}, f^i_{tj}$, where $t=0,\dots, d-1, i=1,\dots, s_t, j=1,\dots, n_{ti}$ and commutation relations \begin{gather*} [g_i, h_0] = h_0, \qquad [g_i, h_i] = h_i,  \\
 [g_i, e^i_{0k}] = e_{0k}^i, \qquad  [g_t, e^i_{tk}] = e_{tk}^i, 
   \\  [e_{tj}^i, f_{tk}^i] = \begin{cases} e^i_{t, j-k}, \quad j>k, \\ h_t, \quad j = k. \end{cases} \end{gather*} \end{theorem}

Theorem~\ref{Th:RealOneKronSeveralEigen} also follows from the results of Section~\ref{S:RealPenc} (see Theorem~\ref{T:MaxBracket}, Lemma~\ref{L:SemiInvSatGood} and Corollary~\ref{Cor:LinSemiGood}).

\section{Realization of JK invariants by Poisson pencils} \label{S:RealPenc}

In this section we prove Theorem~\ref{T:RealPenc} that all JK invariants with at least one Kronecker block can be realized by Poisson pencils.  

\begin{itemize}

\item In Section~\ref{S:FamPois} we construct a family of Poisson pencils $\hat{\mathcal{P}} = \left\{ \hat{\mathcal{A}}_{\lambda} \right\}$ (that is not linear on $\lambda$ and is not defined for some $\lambda$).

\item In Section~\ref{S:ConstPois} we turn the family  $\hat{\mathcal{P}}$ into a Poisson pencil $\mathcal{P}$, if some PDEs \eqref{Eq:CondEigen} are satisfied. 

\item In Section~\ref{S:PenConsLiePois} we discuss when the pencil $\mathcal{P}$ is a Lie--Poisson pencil. In particular, we get the Lie--Poisson pencils from Theorems~\ref{Th:Kron3} and \ref{Th:RealOneKronSeveralEigen}. 

\item In Section~\ref{S:ConsCharNum} we study the solutions of PDEs \eqref{Eq:CondEigen}.

\item In Section~\ref{S:ProofThrealPoiss} we combine the results of the previous sections and realize JK invariants by Poisson pencils. 

\end{itemize}

\subsection{Family of Poisson brackets} \label{S:FamPois}

We start with a construction of a family of Poisson pencils $\hat{\mathcal{P}} = \left\{ \hat{\mathcal{A}}_{\lambda} \right\}$ that is defined for almost all $\lambda \in \mathbb{C}$. This family is not necessary linear in $\lambda$. The matrix of $\mathcal{A}_{\lambda}$ will have the form similar to \eqref{Eq:LieBrackEx}: \[ \mathcal{A}_{\lambda} = \left( \begin{matrix} 0 & A_{\lambda, xs} & A_{\lambda, xy} \\ -A_{\lambda, xs}^T & A_{\lambda, ss} & 0 \\ -A_{\lambda,  xy}^T & 0 & 0\end{matrix} \right). \] 

\begin{itemize}

\item The idea is to take the matrices $A_{\lambda, ss}$ as in Section~\ref{S:RealGenMix}, but replace $y_t$ with some functions $f_t(y)$. 

\item The submatrix $\left( \begin{matrix} 0 & A_{xy} \\ -A_{xy}^T & 0\end{matrix} \right)$ corresponds to some Lie--Poisson bracket from Remark~\ref{Rem:KronAlg} (below, in Sections~\ref{S:ConsCharNum} and \ref{S:ProofThrealPoiss}, we will take this submatrix as in Theorem~\ref{T:KronInv}, similar to Section~\ref{S:RealGenMix}).

\item The  blocks $A_{\lambda, xs}$ will be modified in such a way that $\mathcal{A}_{\lambda}$ becomes a Poisson bracket. 

\end{itemize}

For instance, $\mathcal{A}$ will have the form similar to \eqref{Eq:LiePoissonSevJornOneMKron}, i.e. \begin{equation} \label{Eq:LieView1} \mathcal{A} = \left(\begin{array}{c|ccccc|c} 0 & G_{0,1} & 0 & \cdots & G_{d-1, s_d} & 0 & Y \\ \hline -G_{0,1}^T & 0 & P_{n_{1,1}}(f_1(y)) &  & & &   \\ 0 & -P_{n_{1,1}}^T(f_1(t)) & 0 &  & & &   \\ \vdots & & & \ddots & & & \\ - G_{d-1, s_d}^T & & & & 0 & P_{n_{d, s_d}}(f_d(y)) &   \\ 0 & & & &  -P_{n_{d, s_d}}(f_d(y))^T & 0 &  \\ \hline -Y^T & & & &  & & 0 \end{array} \right),\end{equation} for some matrices $G_t$ and $Y$.

\begin{theorem} \label{T:MaxBracket} Let $(x_1, \dots, x_m, p^i_{tj}, q^i_{tj}, y_1, \dots, y_r)$, where $t = 1, \dots, d, i=1, \dots, s_t, \, \, j=1, \dots, n_{ti}$ be linear coordinates on a vector space $V$. Let  

\begin{itemize}

\item $b_{w}$ and $\alpha_{uw}$ be constants for $u=1, \dots, m$ and $w=1, \dots, r$,

\item $c^i_{tj}$ and $\delta_t$  be constants for $t = 1, \dots, d, i=1, \dots, s_t, \, \, j=1, \dots, n_{ti}$,

\item  $f_t(y)$ be smooth functions for $t = 1, \dots, d$.

\end{itemize}

Then for any $\lambda \in \mathbb{C}$ such that $f_t(y) + \lambda \delta_t \not = 0$ the Poisson bracket  $\left\{\cdot, \cdot\right\}_{\lambda}$ given by the following commutation relations is well defined: \begin{equation} \label{Eq:SuperBr1} \left\{x_u ,  y_w\right\}_{\lambda} = \alpha_{uw} \left(y_w + \lambda b_w \right), \qquad \left\{p^i_{tj},  q^i_{tk} \right\}_{\lambda} = \begin{cases} p^i_{t,j-k} + \lambda c^i_{t, j-k}, \qquad & j \geq k, \\ f_t(y) + \lambda \delta_t, \qquad & j = k,  \end{cases} 
\end{equation}  and \begin{equation} \label{Eq:SuperBr2} \left\{x_u ,  p^i_{tj}\right\}_{\lambda} = \left(p^i_{t,j} + \lambda c^i_{t, j} \right) G_{\lambda, u t}(y), 
\end{equation}  where \begin{equation} \label{Eq:GFunc}  G_{\lambda, u t}(y) = \frac{\left\{ x_u, f_t(y) + \lambda \delta_t \right\}_{\lambda} }{ f_t(y) + \lambda \delta_t}.\end{equation} \end{theorem}

Note that the Lie--Poisson brackets from Theorems~\ref{Th:Kron3} and \ref{Th:RealOneKronSeveralEigen} are a special case of the brackets from Theorem~\ref{T:MaxBracket} (see also Corollary~\ref{Cor:LinSemiGood} below).

\begin{remark} Note that the right-hand side of \eqref{Eq:GFunc} is a function of $y$ and it can be found using \eqref{Eq:SuperBr1}. Therefore, we can calculate \eqref{Eq:SuperBr2} after we defined the bracket $\left\{x_u ,  y_w\right\}_{\lambda}$ by \eqref{Eq:SuperBr1}.  \end{remark}

\begin{proof}[Proof of Theorem~\ref{T:MaxBracket}] We need to check the Jacobi indentity. 

\begin{enumerate}

\item For the coordinates $(x_u, y_w)$ we get a Lie--Poisson pencil (see Remark~\ref{Rem:KronAlg}) and the functions $q^i_{tk}$ commute with them. Hence, we don't have to check the Jacobi identity without $p^i_{tk}$. 

\item For the coordinates $(p^i_{tj}, q^i_{tk}, y_w)$ we get the Poisson bracket from Proposition~\ref{P:BrackOneKronSevJordGen}. Thus, we don't have to check the Jacobi identity without $x_{u}$.

\item Note that all functions in the right-hand side of \eqref{Eq:SuperBr1}  and \eqref{Eq:SuperBr2} are functions of $y$ and $p^i_{tk}$.  Thus, for the triples $(x_{\alpha}, p^i_{tk}, y_w)$ and $(x_{\alpha}, p^i_{tk}, p^{i'}_{t'k'})$ all terms in the Jacobi identity vanish. It remains to check the Jacobi identity for following two triples: \[ (x_{u}, p^i_{tk}, q^{i'}_{j'k'}), \qquad (x_{u}, x_{u'}, p^i_{tk}).\] 

\item Consider the Jacobi identity for $(x_{\alpha}, p^i_{tk}, q^{i'}_{j'k'})$. Note that \[ \left\{ x_{u}, \left\{ p^i_{tk}, q^{i'}_{j'k'}\right\}_{\lambda} \right\}_{\lambda}  = G_{\lambda, u t}(y)\left\{ p^i_{tk}, q^{i'}_{j'k'}\right\}_{\lambda}. \]
Therefore, \begin{gather*}  \left\{ x_{u}, \left\{ p^i_{tk}, q^{i'}_{j'k'}\right\}_{\lambda} \right\}_{\lambda} +  \left\{p^i_{tk}, \left\{ q^{i'}_{j'k'},  x_{u}\right\}_{\lambda} \right\}_{\lambda} +  \left\{ q^{i'}_{j'k'}\left\{  x_{u}, p^i_{tk}\right\}_{\lambda} \right\}_{\lambda} = \\ = G_{\lambda, u t}(y)\left\{ p^i_{tk}, q^{i'}_{j'k'}\right\}_{\lambda} + 0 + G_{\lambda, u t}(y)\left\{q^{i'}_{j'k'}, p^i_{tk} \right\}_{\lambda} = 0.\end{gather*}

\item It remains to prove the Jacobi identity for $(x_{u}, x_{u'}, p^i_{tk})$. We get \begin{gather*} \left\{ x_{u},  \left\{ x_{u'}, p^i_{tk}\right\}_{\lambda} \right\}_{\lambda} +  \left\{x_{u'}, \left\{ p^i_{tk},  x_{u}\right\}_{\lambda} \right\}_{\lambda} +  \left\{ p^i_{tk}\left\{  x_{u}, x_{u'}\right\}_{\lambda} \right\}_{\lambda} = \\  = \left\{ x_{u},  \left( p^i_{tk} + \lambda c^i_{tk} \right)  G_{\lambda, u't}(y) \right\}_{\lambda}-  \left\{x_{u'}, \left( p^i_{tk} + \lambda c^i_{tk} \right)  G_{\lambda, ut}(y) \right\}_{\lambda} +  0  \end{gather*}
By the Leibniz rule the expression takes the form \begin{gather*} \left( p^i_{tk} + \lambda c^i_{tk} \right) \left( \left\{ x_{u},  G_{\lambda, u't}(y) \right\}_{\lambda} +   G_{\lambda, ut}(y) G_{\lambda, u't}(y) -  G_{\lambda, ut}(y) G_{\lambda, u't}(y) -  \left\{x_{u'}, G_{\lambda, ut}(y) \right\}_{\lambda} \right) =  \\ =\left( p^i_{tk} + \lambda c^i_{tk} \right) \left( \left\{ x_{u},  G_{\lambda, u't}(y) \right\}_{\lambda} -  \left\{x_{u'}, G_{\lambda, ut}(y) \right\}_{\lambda} \right) \end{gather*}  It remains to prove that \[ \left\{ x_{u},  G_{\lambda, u't}(y) \right\}_{\lambda} -  \left\{x_{u'}, G_{\lambda, ut}(y) \right\}_{\lambda} = 0.\]  Denote \[ H_{\lambda, t}(y) = f_t(y) + \lambda \delta_t.\] Then \[G_{\lambda, ut}(y) = \frac{\left\{ x_u, H_{\lambda, t}(y)  \right\}_{\lambda} }{H_{\lambda, t}(y) }. \]  The expression takes the form \begin{gather*} \left\{ x_{u}, \frac{\left\{ x_{u'}, H_{\lambda, t} \right\}_{\lambda} }{H_{\lambda, t} } \right\}_{\lambda} -  \left\{x_{u'}, \frac{\left\{ x_u, H_{\lambda, t} \right\}_{\lambda} }{H_{\lambda, t} } \right\}_{\lambda} = \\ = \frac{1}{H_{\lambda, t}} \left( \left\{ x_{u},\left\{ x_{u'}, H_{\lambda, t}  \right\}_{\lambda} \right\}_{\lambda} -  \left\{x_{u'}, \left\{ x_u, H_{\lambda, t}  \right\}_{\lambda}\right\}_{\lambda} \right) - \\ - \frac{1}{H_{\lambda, t}^2} \left( \left\{ x_{u},H_{\lambda, t}  \right\}_{\lambda}  \left\{ x_{u'}, H_{\lambda, t}  \right\}_{\lambda}  -  \left\{x_{u'}, H_{\lambda, t}  \right\}_{\lambda} \left\{ x_u, H_{\lambda, t}(y)  \right\}_{\lambda} \right). \end{gather*}  The second term is zero. We have already proved the Jacobi indentity in the coordinates $(x_u, y_w)$. Therefore, the first term also vanishes:  \[   \left\{ x_{u},\left\{ x_{u'}, H_{\lambda, t}(y)  \right\}_{\lambda} \right\}_{\lambda} -  \left\{x_{u'}, \left\{ x_u, H_{\lambda, t}(y)  \right\}_{\lambda}\right\}_{\lambda}  =  \left\{H_{\lambda, t}(y), \left\{ x_u, x_{u'}   \right\}_{\lambda}\right\}_{\lambda} = 0.\] 
\end{enumerate} 

Theorem~\ref{T:MaxBracket} is proved. \end{proof}

\subsection{Construction of Poisson pencils} \label{S:ConstPois}

Consider the family $\hat{\mathcal{P}}$ of Poisson brackets from Theorem~\ref{T:MaxBracket}. We want to turn it into a Poisson pencil $\mathcal{P}$. A priori, we have the following problems: 

\begin{enumerate}

\item The family $\hat{\mathcal{P}}$ is not linear on $\lambda$ because of the right-hand side of \eqref{Eq:SuperBr2}.

\item The bracket $\left\{\cdot, \cdot \right\}_{\infty}$, which should be the constant bracket $\mathcal{A}_a$, is not defined. 

\end{enumerate}

In Lemma~\ref{L:FamilyFix} we show that we can solve these problems if the functions $f_t(y)$ satisfies some system of PDE \eqref{Eq:CondEigen}. Let us intoduce several notations.

\begin{itemize}

\item Denote \begin{equation} \label{Eq:FutEigen} \lambda_t(y) = \frac{f_t(y)}{\delta_t}.\end{equation} This functions $\lambda_t$ would be eigenlvalues of the Poisson pencil $\mathcal{P}$. 
\item Consider a linear space with coordinates $(x_1, \dots, x_m, y_1, \dots, y_r)$. Denote by \[\mathcal{Q} = \left\{ \mathcal{A}_y + \lambda \mathcal{B}_b\right\}\] the Poisson pencil given by \begin{equation} \label{Eq:SubAlgPencYB} \left\{x_u ,  y_w\right\}_{\lambda} = \alpha_{uw} \left(y_w + \lambda b_w \right). \end{equation} Note that this a Lie--Poisson pencil of some Lie algebra $\mathfrak{h}$ (see Remark~\ref{Rem:KronAlg}).

\end{itemize}

\begin{lemma} \label{L:FamilyFix} Consider the family $\hat{\mathcal{P}} = \left\{ \mathcal{P}_{\lambda} \right\}$ of Poisson brackets from Theorem~\ref{T:MaxBracket}. Assume that the functions $\lambda_t(y)$ given by \eqref{Eq:FutEigen} satisfy \begin{equation} \label{Eq:CondEigen} \left( \mathcal{A}_y - \lambda_t (y) \mathcal{B}_b \right) d \lambda_t(y) = 0,\end{equation}  where the pencil $\mathcal{Q} = \left\{ \mathcal{A}_y + \lambda \mathcal{B}_b\right\}$ is given by \eqref{Eq:SubAlgPencYB}. Then the brackets from $\hat{\mathcal{P}}$ have the form \[ \mathcal{P}_{\lambda} = \mathcal{A} + \lambda \mathcal{B}, \] and the family $\hat{\mathcal{P}}$ extend to a pencil of compatible Poisson brackets \begin{equation} \label{Eq:FixedFam1} \mathcal{P} = \left\{ \mathcal{A} + \lambda \mathcal{B} \, \, \bigr| \, \, \lambda \in \mathbb{C} \cup \left\{ \infty\right\} \right\}.\end{equation}  \end{lemma}

\begin{proof}[Proof of Lemma~\ref{L:FamilyFix}] Condition~\eqref{Eq:CondEigen} can be written as \begin{equation} \label{Eq:FixFamEq1} \left\{ g(x, y), \lambda_t(y) \right\}_{\mathcal{A}_y}   = \lambda_t(y) \left\{ g(x, y), \lambda_t(y)\right\}_{\mathcal{B}_b}\end{equation}  for any smooth functions $g(x, y)$. The brackets $ \mathcal{A}_y + \lambda \mathcal{B}_b$ are the restrictions of the brackets $\mathcal{P}_{\lambda}$ to the coordinates $(x, y)$. Thus, \begin{equation} \label{Eq:FixFamEq2} \left\{g, \lambda_t \right\}_{\mathcal{P}_{\lambda}} =  \left\{g, \lambda_t\right\}_{\mathcal{A}_{y}} + \lambda \left\{g, \lambda_t\right\}_{\mathcal{B}_{b}}. \end{equation} Consider~\eqref{Eq:GFunc}, it takes the form \[  G_{\lambda, u t}(y) = \frac{\left\{ x_u, f_t(y) + \lambda \delta_t \right\}_{\mathcal{P}_{\lambda} }}{ f_t(y) + \lambda \delta_t} =\frac{\left\{ x_u, \lambda_t(y) \right\}_{\mathcal{P}_{\lambda} }}{ \lambda_t(y) + \lambda}.  \] In the last equality we excluded $\lambda \delta_t$ from the numerator, since it is a constant. By \eqref{Eq:FixFamEq1} and \eqref{Eq:FixFamEq2} we get \begin{equation} \label{Eq:GLambd1} G_{\lambda, u t}(y)  = \frac{(\lambda_t (y) + \lambda)\cdot \left\{ x_u, \lambda_t(y) \right\}_{\mathcal{B}_b}}{\lambda_t (y)+ \lambda}  = \left\{ x_u, \lambda_t(y) \right\}_{\mathcal{B}_b}.\end{equation} Therefore,  the right-hand side of \eqref{Eq:SuperBr2} and the pencil $\hat{\mathcal{P}}$ are linear on $\lambda$. Thus, we can extend $\hat{\mathcal{P}}$ to a Poisson pencil \eqref{Eq:FixedFam1}. Lemma~\ref{L:FamilyFix} is proved. \end{proof}

\subsection{Construction of Lie--Poisson pencils} \label{S:PenConsLiePois}

Lemma~\ref{L:FamilyFix} allows us to construct Poisson pencils if we have functions $\lambda_t(y)$ that satisfy \eqref{Eq:CondEigen}. The next obvious question is how to find such functions $\lambda_t(y)$. We want to realize JK invariants with $N$ Jordan tuples and a Kronecker block by Poisson pencils. The idea is to find \eqref{Eq:CondEigen} with at least $N$ solutions $\lambda_t(y)$. We do it in the next section. In this section we construct Lie--Poisson pencils using linear solutions of \eqref{Eq:CondEigen}. The next statement generalize Lie--Poisson pencils from Theorems~\ref{Th:Kron3} and \ref{Th:RealOneKronSeveralEigen}. Recall that in these theorems the eigenvalues $\lambda_t(y)$ and functions $f_t(\lambda)$ are linear semi-invariants. 

\begin{lemma} \label{L:SemiInvSatGood} Consider the family $\hat{\mathcal{P}} = \left\{ \mathcal{P}_{\lambda} \right\}$ of Poisson brackets from Theorem~\ref{T:MaxBracket}. Let $\mathfrak{h}$ be the Lie algebra with coordinates $(x_1, \dots, x_m, y_1, \dots, y_r)$ and the Lie-Poisson bracket \[ \left\{ x_u, y_w \right\} = \alpha_{uw} y_w.\]  If the functions $f_t(y)$ are semi-invariants of $\mathfrak{h}$ and $\delta_t = f_t(b)$, then the functions $\displaystyle \lambda_t(y) = \frac{f_t(y)}{\delta_t}$ satisfy \eqref{Eq:CondEigen}. \end{lemma}

\begin{proof}[Proof of Lemma~\ref{L:SemiInvSatGood}]  We show that $\displaystyle \lambda_t(y) = \frac{f_t(y)}{\delta_t}$ satisfies \eqref{Eq:CondEigen} in Lemma~\ref{L:EigenSemiInv} below (for linear semi-invariants it also follows from Lemma~\ref{L:LinSemiInvPoissEquality} below).  \end{proof}

\begin{corollary} \label{Cor:LinSemiGood} Under the conditions of Lemma~\ref{L:SemiInvSatGood}, we can extend $\hat{\mathcal{P}}$ to a Poisson pencil  $\mathcal{P} = \left\{ \mathcal{A} + \lambda \mathcal{B}  \right\}$. If 
all $f_t(y)$ are linear semi-invariants, then $\mathcal{P}$ is a Lie--Poisson pencil for some Lie algebra $\mathfrak{g}$. 
\end{corollary} 

\begin{proof}[Proof of Corollary~\ref{Cor:LinSemiGood}] We can extend $\hat{\mathcal{P}}$ to $\mathcal{P}$ by  Lemma~\ref{L:FamilyFix}. It is easy to see that $\mathcal{A}$ is linear in $x$ and $\mathcal{B}_b$ is a constant Poisson bracket. Note that all the functions $G_{\lambda, u t}(y)$ are constants, since by \eqref{Eq:GLambd1} \begin{equation} \label{Eq:GLinSemi}  G_{\lambda, u t}(y)  = \left\{ x_u, \lambda_t(y) \right\}_{\mathcal{B}_b},\end{equation} where $\lambda_t(y)$ are linear functions and $\mathcal{B}_b$ is a constant bracket. Since $\mathcal{A}$ is linear, it is a Lie--Poisson bracket of some Lie algebra $\mathfrak{g}$. It is also easy to see that $\mathcal{B}$ is $\mathcal{A}$ at some point of $\mathfrak{g}$ (the constants $c^i_{tj}$ and $b_w$ are the values of the coordinates $p^i_{tj}$ and $y_w$). Thus, $\mathcal{P}$ is a Lie--Poisson pencil of $\mathfrak{g}$. Corollary~\ref{Cor:LinSemiGood} is proved. 
\end{proof} 

\begin{remark} \label{Rem:GMatGenPois} In Corollary~\ref{Cor:LinSemiGood}, since $f_t(y)$ are linear semi-invariants, there are vectors $\chi_t \in \mathbb{C}^m$ such that \[ \left( \begin{matrix} \left\{ x_1, y_i\right\}_{\mathcal{A}_y} \\ \dots \\ \left\{ x_m, y_i\right\}_{\mathcal{A}_y} \end{matrix} \right)  = y_i \cdot \chi_i. \] It is easy to check that \eqref{Eq:GLinSemi} takes the form \[   G_{\lambda, u t}(y)  = \left\{ x_u, \frac{f_t(y)}{f_t(b)} \right\}_{\mathcal{B}_b} = (\chi_t)_u. \] Thus, if the Lie--Poisson bracket has the form \eqref{Eq:LieView1}, then the block $G_{t-1,i}$, given by  \eqref{Eq:SuperBr2}, has the form \[ G_{t-1,i} = \chi_t^T p^i_t.  \]  Here we regard $\chi_t$  as a $1 \times m$ vector, and \[p^i_t = \left(p^i_{t1}, \dots, p^i_{t n_{ti}} \right)\] is a $1 \times n_{ti}$ vector (the block $G_{t-1,i}$ is $m \times n_{ti}$). Note that for  Lie--Poisson pencils from Theorems~\ref{Th:Kron3} and \ref{Th:RealOneKronSeveralEigen} we get the matrices that were described in Section~\ref{S:GenMixLie} (Compare with Remark~\ref{Rem:BlockXS1}.) \end{remark}

\subsection{Construction of characteristic numbers} \label{S:ConsCharNum}

In order to realize JK invariants with $N$ Jordan tuples and a Kronecker block by Poisson pencils we need to find  the pencil $\mathcal{Q} = \left\{ \mathcal{A}_y + \lambda \mathcal{B}_b\right\}$ such that the equation~\eqref{Eq:CondEigen} has at least $N$ solutions $\lambda_t(y)$. Let the pencil $\mathcal{Q} = \left\{ \mathcal{A}_y + \lambda \mathcal{B}_b\right\}$ be the Lie--Poisson pencil from Theorem~\ref{T:KronInv}. In other words, \begin{equation} \label{Eq:LiePoissMegaPDE} \left\{ \mathcal{A}_y + \lambda \mathcal{B}_b\right\} = \left( \begin{matrix} 0 & Y_{\lambda} \\ -Y_{\lambda}^T & 0 \end{matrix} \right), \qquad Y_{\lambda} =  \left( \begin{matrix} y_0 + \lambda b_0 & y_1 + \lambda b_1 &  &  \\ \vdots & & \ddots & \\ y_0 + \lambda b_0& & & y_{m} + \lambda b_m\end{matrix}\right). \end{equation} The equation  \eqref{Eq:CondEigen} becomes the following system of  first order quasi-linear PDE on $u = \lambda_t(y)$: \begin{equation} \label{Eq:MegaPDE1} \left(y_0 -  u b_0 \right) \frac{\partial u}{\partial y_0} + \left(y_j -  u b_j \right) \frac{\partial u}{\partial y_j} =0, \qquad j=1, \dots, m.\end{equation} Next statement shows that at any point $y$ there exist any number $N$ of solutions of \eqref{Eq:MegaPDE1} with distinct values $u(y)$.

\begin{theorem} \label{T:MegaPDE1} Let $b_0, \dots, b_m \in \mathbb{C}^*$. For any $u_0 \in \mathbb{C}, (a_0, \dots, a_m) \in \mathbb{C}^{m+1}$ there exists a holomorphic solution $u = u(y_0,\dots, y_m)$ of \eqref{Eq:MegaPDE1} such that $u(a_0, \dots, a_m) = u_0$.  \end{theorem}

\begin{proof}[Proof of Theorem~\ref{T:MegaPDE1}]  First of all, if $\displaystyle u_0 = \frac{a_j}{b_j}$ for some $j =0, \dots, m$, then we can take the solution \[ u(y) = \frac{y_j}{b_j}.\] Below we assume that $\displaystyle u_0 \not = \frac{a_j}{b_j}$.

The idea of the proof is simple:  since \eqref{Eq:MegaPDE1} is a system of first order PDE, we want to use the method of characteristics. For a systems of PDE we find a solution using the (complex) Frobenious theorem. Let us explain the details.

\begin{enumerate}

\item \textit{We use the method of characteristics.} Consider the space $\mathbb{C}^{n+2}(y_0,\dots, y_m, u)$ and the following $m$ holomorphic vector field on it: \[ X_j =  \left(y_0 -  u b_0 \right) \frac{\partial }{\partial y_0} + \left(y_j -  u b_j \right) \frac{\partial }{\partial y_j}, \qquad j=1, \dots, m. \] Consider a holomorphic function \[ F(y_0, \dots, y_m, u) = u - f(y_0, \dots, y_m).\] Then \[ dF = du - \sum_{j=0}^m \frac{\partial f}{\partial y_j} d y_j. \] It is easy to see that $u = f(y)$ is a solution of \eqref{Eq:MegaPDE1} if and only if \begin{equation} \label{Eq:CharDual} \langle dF, X_j \rangle = 0, \qquad j=1,\dots, m.\end{equation} Therefore, \textit{we seach for holomorhpic functions $F(y, u)$ such that \eqref{Eq:CharDual} holds and} \begin{equation} \label{Eq:Cond1Char}  F(a, u_0) = 0, \qquad \frac{\partial F}{\partial u}(a, u_0) \not = 0.  \end{equation} The first condition in \eqref{Eq:Cond1Char} comes from $u_0 = u(y_0)$. We need the second condition in \eqref{Eq:Cond1Char}  in order to use the Holomorphic Implicit Function Theorem (see e.g. \cite{Fritzsche}) and get the solution $u = f(y)$. 

\item We can easily find two holomorphic 1-forms in $\mathbb{C}^{n+1}$ that are solutions of \begin{equation} \label{Eq:1FormMega1} \langle \varphi, X_j \rangle = 0, \qquad j =1,\dots, m.\end{equation} They are \begin{equation} \label{Eq:FormsSol1} \varphi^1 = du, \qquad \varphi^2 = - \frac{d y_0}{y_0 - u b_0} + \sum_{j=1}^m \frac{d y_j}{y_j - u b_j}.\end{equation} This forms $\varphi^1, \varphi^2$ are linearly independant everywhere on a neighborhoood $V$ of $(a_0, \dots, a_m, u_0) \in \mathbb{C}^{m+1}$.  Next, \textit{let us use the well-known complex Frobenius theorem} (see, e.g. Appendix 8 in \cite{KobayashiNomidzu}). 

\begin{theorem}\label{T:CompFrob} Let $\varphi^1, \dots, \varphi^r$ be holomorphic $1$-forms defined in a neighborhoood $V$ of the origin in $\mathbb{C}^m$. If $\varphi^1, \dots, \varphi^r$ are linearly independant everywhere on $V$ and if \[ d\varphi^i 	= \sum_{k=1}^r \psi^j_k \wedge \varphi^k, \qquad j=1, \dots, r,\] where each $\psi^j_k$ is a holomorphic $1$-form on $U$, then there exist a smaller neighborhoood $W$ of the origin and holomorphic functions $g^1, \dots, g^r$ on $W$ such that \[ \varphi^j = \sum_{k=1}^r p^j_k d g^k, \qquad j=1,\dots, r\] where each $p^j_k$ is a holomorphic function on $W$.  \end{theorem}

It is easy to see that for the forms $\varphi^i$, given by \eqref{Eq:FormsSol1}, we have \[ d \varphi^1 = d^2u = 0, \qquad d\varphi^2 = \psi \wedge du.\] Thus, we can apply the complex Frobenious theorem. We get that \[ du = p^1_1 dg^1 + p^1_2 dg^2, \qquad \varphi^2 = p^2_1 dg^1 + p^2_2 dg^2.\] Without loss of generality $p^1_1(a_0, \dots, a_m, u_0) \not  =0$ and we can replace $g^1, g^2$ with the holomorphic functions $u$ and $g=g^2$.  We get that \[ \varphi^2 = p^1(y, u) du + p^2 (y, u) d g(y, u).\] Note that $\varphi^1 = du, \varphi^2$ are linear combinations of $du, dg$. Since $\varphi^i$ are solutions of \eqref{Eq:1FormMega1} and they are linearly independant (everywhere on $W$),  the $1$-forms $du, dg$ are also linearly independant (everywhere on $W$) and satisfy \eqref{Eq:1FormMega1}. Therefore, \textit{any holomorphic function \[ F(y, u) = H(u, g(y,u))\] is a solution of \eqref{Eq:Cond1Char}}. 

\item Consider the following holomorhic functions \[ F(y,u) = C(u-u_0) + \left(g(y, u) - g(a, u_0)\right),\] for some $C \in \mathbb{C}$. Any such function $F(y,u)$ is a solution of  \eqref{Eq:Cond1Char} and it satisfies \[ F(a, u_0) = 0.\] Note that \[ \frac{\partial F}{\partial u} (a, u_0) = C + \frac{\partial g}{\partial u} (a, u_0).\] Thus for some $C \in \mathbb{C}$ we have \[ \frac{\partial F}{\partial u} (a, u_0) \not =0. \] Using the Holomorphic Implicit Function Theorem we get the required solution $u = u(y)$. 

\end{enumerate}

Theorem~\ref{T:MegaPDE1} is proved. \end{proof}

\subsection{Proof of Theorem~\ref{T:RealPenc}} \label{S:ProofThrealPoiss}

It suffices to realize JK invariants with one Kronecker block, then we can take the product with the Kronecker Lie--Poisson pencil from Theorem~\ref{T:KronInv} (see Remark~\ref{Rem:ProdJK}). Let the Jordan tuples be $J_{\lambda_t}(2n_{t1}, \dots, 2n_{ts_t})$ ($t=1, \dots d$), and the Kronecker size be $(2m+1) \times (2m+1)$. 

\begin{itemize}

\item Consider the family of Poisson brackets $\hat{\mathcal{P}}$ from Theorem~\ref{T:MaxBracket}, where the pencil $\mathcal{Q} = \left\{ \mathcal{A}_y + \lambda \mathcal{B}_b\right\}$, given by \eqref{Eq:SubAlgPencYB}, is the Lie--Poisson pencil from Theorem~\ref{T:KronInv}. In other words $\mathcal{Q}$ has the form \eqref{Eq:LiePoissMegaPDE}. 

\item By Theorem~\ref{T:MegaPDE1} for any point $a \in \mathbb{C}^{m}$ there exist $d$ solutions $\lambda_t(y)$ of  \eqref{Eq:CondEigen} with distinct values $\lambda_t(a)$. Consider a neighborhood $M$ of $a$ such that the functions $\lambda_t(y)$ are pairwise distinct.

\item Take constants $d_t =1$ and functions $f_t(y) = \lambda_t(y)$. By Lemma~\ref{L:FamilyFix} we can extend the family $\hat{\mathcal{P}}$ to a Poisson pencil $\mathcal{P}$ on $M$.

\item It is proved similar to Theorem~\ref{Th:Kron3} that the JK invariants of $\mathcal{P}$ are the Jordan tuples  $J_{\lambda_t}(2n_{t1}, \dots, 2n_{ts_t})$ and the Kronecker $(2m+1) \times (2m+1)$ block.  Note that the eigenvalues $\lambda_t(y)$ are different by construction. 

\end{itemize}

Theorem~\ref{T:RealPenc} is proved.

\section{General mixed case. Obstruction} \label{S:GenObst}

Recall that $\operatorname{JK}(\mathfrak{g})$, given by \eqref{Eq:JKSetLie}, is the multiset that contains the following elements:

\begin{itemize}

\item tuples $(2n_1, \dots, 2n_s)$ for each Jordan tuple $J_{\lambda} (2n_1, \dots, 2n_s)$, 

\item Kronecker sizes $2k_i-1$. 

\end{itemize}

\begin{definition} The \textbf{multiplicity} of a Jordan tuple $J_{\lambda} (2n_1, \dots, 2n_s)$ is the multiplicity of the tuple $(2n_1, \dots, 2n_s)$ in $\operatorname{JK}(\mathfrak{g})$. A Jordan tuple $J_{\lambda} (2n_1, \dots, 2n_s)$ is \textbf{unique} if it has multiplicity $1$. \end{definition}

\begin{remark} In \cite{BolsZhang}  the term ``multiplicity of a characteristic number'' was used. This is another number, let us explain the difference between these ``multiplicities''. Consider a Jordan tuple $J_{\lambda} (2n_1, \dots, 2n_s)$.

\begin{itemize}

\item In  \cite{BolsZhang} the multiplicity of the characteristic number $\lambda$ as the root of the characteristic polynomial $p_{\mathcal{P}}$ was used. That is the number of occurrences of $\lambda$ in the complete factorization of $p_{\mathcal{P}}$, and it is equal to $n_1 + \dots + n_s$. 

\item The multiplicity of the Jordan tuple  $J_{\lambda} (2n_1, \dots, 2n_s)$ is the  number of instances the tuple  $(2n_1, \dots, 2n_s)$ appears in $\operatorname{JK}(\mathfrak{g})$.

\end{itemize}

Luckily, below we are mostly interested whether a Jordan tuple is unique or not. So, there should not be much confusion with ``multiplicities''. \end{remark} 

\begin{example}
Consider JK invariants: \[ J_{\lambda_1}(4, 2), \qquad J_{\lambda_2}(4, 2), \qquad J_{\lambda_3}(4, 2), \qquad J_{\lambda_4}(2).\] The Jordan tuple $J_{\lambda_1}(4, 2)$ has multiplicity $3$. The Jordan tuple $J_{\lambda_4}(2)$ has multiplicity $1$, i.e. it is unique. \end{example}

\begin{theorem} \label{Th:FinalObs} Consider JK invariants of a Lie algebra.

\begin{itemize}

\item Let  the Kronecker sizes be $2k_1-1, \dots, 2k_q - 1$. 

\item Let $J_{um}$ be the number\footnote{Here the letter $u$ in $J_{um}$ stands for ``unique'' and the letter $m$ means ``multiple maxima''.} of unique Jordan tuples with multiple maxima $J_{\lambda_i}(2n_{i1}, \dots, 2n_{is_i}), n_{i1} = n_{i2} \geq n_{ij}$. 

\end{itemize}

If there are no Kronecker  $3 \times 3$ blocks and no more than one Kronecker $1 \times 1$ block, then \[  J_{um} \leq k_1 + \dots + k_q.\]

\end{theorem}

Theorem~\ref{Th:FinalObs}  is proved in Section~\ref{S:ProofThGenObs}. First, we need to prove some facts about characteristic numbers (=eigenvalues) $\lambda_i(x)$ and semi-invariants of Lie algebras. The idea of the proof of Theorem~\ref{Th:FinalObs} is quite simple:

\begin{itemize}

\item In Section~\ref{S:EigenLineSemi} (see Lemma~\ref{L:EigenSemiInv}) we prove that the eigenvalues $\lambda_i$ of unique Jordan tuples are linear semi-invariants.

\item In Section~\ref{S:IndepLinSemiInv}  (see Lemma~\ref{L:LinSemiInvPoissEquality}) we show that under the conditions of Theorem~\ref{Th:FinalObs}  linear semi-invariants are linearly independant. 

\item This lead to a contradiction with Lemma~\ref{L:EigenCore}, which requires $d \lambda_i$ for Jordan blocks with multple maxima to belong to the $(k_1 + \dots +k_q)$-dimensional core $\mathcal{K}$.

\end{itemize}

In order to prove Lemma~\ref{L:EigenSemiInv}, we need a well-known fact that \[d \lambda_i \in \operatorname{Ker} \left(\mathcal{A} - \lambda_i \mathcal{B} \right)\]  for any eigenvalue $\lambda_i(x)$ of a Poisson pencil $\left\{\mathcal{A}+ \lambda \mathcal{B}\right\}$. That fact is proved in Section~\ref{S:LinSemiMainWow} (see Lemma~\ref{L:EigenDiff}).

\subsection{Eigenvalues of unique Jordan tuples are linear semi-invariants} \label{S:EigenLineSemi}

\begin{definition} Let $\mathfrak{g}$ be a Lie algebra and $S(\mathfrak{g})$ be its symmetric algebra, that we identify with the algebra of polynomials on $\mathfrak{g}^*$. A non-zero element $f\in S(\mathfrak{g})$ is a \textbf{semi-invariant with weight} $\chi_f \in \mathfrak{g}^*$ if \begin{equation} \label{Eq:SemiDef} \left\{ g, f\right\}_x = \chi_f(dg) f(x)\end{equation} for any other $g \in S(\mathfrak{g})$. 
\end{definition}

We identify elements $e \in \mathfrak{g}$ with the corresponding linear functions on $\mathfrak{g}^*$: \[ l_e : \mathfrak{g} \to \mathbb{C}, \qquad l_e(x) = \langle e, x \rangle.\] Note that linear semi-invariants $e \in \mathfrak{g}$ correspond to one-dimensional ideals in $\mathfrak{g}$: \[ [u,e] = \chi_{e} (u) \cdot e, \qquad \forall u \in \mathfrak{g}.\]

\begin{lemma} \label{L:LinEigen} Consider a Lie-Poisson pencil $\mathcal{P} = \left\{ \mathcal{A}_x + \lambda \mathcal{A}_a\right\}$ for a Lie algebra $\mathfrak{g}$. If $J_{\lambda}(2n_1, \dots, 2n_s)$ is a unique Jordan tuple, then $\lambda(x)$ is a linear semi-invariant: \[ \lambda(x) = \langle e,  x \rangle, \qquad \left\{ g, \lambda\right\}_x = \chi_{\lambda}(dg) \cdot \lambda(x),\] for some $e \in \mathfrak{g}$, $\chi_{\lambda} \in \mathfrak{g}^*$ and any $g \in S(\mathfrak{g})$. 
\end{lemma}

\begin{proof}[Proof of Lemma~\ref{L:LinEigen}]  Eigenvalues are roots of the reduced polynomial \[ p_{\mathfrak{g}, \operatorname{red}} = f_1(x) \cdot \dots \cdot f_s(x)\]  (see Section~\ref{S:CharNumSing}). Assume that $f_j(x - \lambda(x) a) = 0$. By Proposition~\ref{Prop:ReducEigen}, since $J_{\lambda}(2n_1, \dots, 2n_s)$ is a unique tuple, $f_j(x)$ is a linear function: \[ f_j(x) = \langle u, x \rangle, \qquad u \in\mathfrak{g}.\] Recall that the functions $f_i(x)$ are semi-invariants. Hence, $\lambda(x)$ is also a linear function   \[ \lambda(x) = \frac{\langle u, x \rangle}{\langle u, a \rangle}\] and a semi-invariant. Lemma~\ref{L:LinEigen}  is proved. \end{proof}

\subsection{Linear semi-invariants and Lie-Poisson pencils}  \label{S:LinSemiMainWow}

In the proof of Theorem~\ref{Th:FinalObs} we use the following Lemma~\ref{L:LinSemiInvPoissEquality}. In Sections~\ref{S:DifEigenDifGeom}  and \ref{S:RootSemiInv} we prove more general facts for characteristic numbers of Poisson pencils, and for arbitrary semi-invariants respectively. The facts from Sections~\ref{S:DifEigenDifGeom}  and \ref{S:RootSemiInv}  are useful, but we don't need them for Theorem~\ref{Th:FinalObs}.

\begin{lemma} \label{L:LinSemiInvPoissEquality} Let $\mathfrak{g}$ be a Lie algebra and $l_e(x) = \langle e, x \rangle, e \in \mathfrak{g}$ be a linear semi-invariant.  Then for any $x, a\in \mathfrak{g}^*$ and any function $g \in C^{\infty}\left(\mathfrak{g}^*\right)$ we have  \begin{equation} \label{Eq:LinSemiNiceForm}  l_e(a) \cdot \left\{ g, l_e\right\}_x = l_e (x) \cdot \left\{ g, l_e\right\}_a.\end{equation} \end{lemma}

Recall that we identify the linear function $l_e$ with the element $e \in \mathfrak{g} =\left(\mathfrak{g}^*\right)^*$. If $l_e(a) \not = 0$, then in terms of the Lie-Poisson pencil $\mathcal{A}_{x} + \lambda \mathcal{A}_a$ we can write \eqref{Eq:LinSemiNiceForm} as \[ e \in \operatorname{Ker} \left( \mathcal{A}_x - \frac{\langle e, x\rangle  }{\langle e, a\rangle} \mathcal{A}_a\right).\]

\begin{proof}[Proof of Lemma~\ref{L:LinSemiInvPoissEquality}] If $(x_1, \dots, x_n)$ are linear coordinates on $\mathfrak{g}^*$, then by the Leibniz rule \[ \left\{ g, l_e \right\}_x = \sum_i \frac{\partial g}{\partial x_i} \left\{ x_i, l_e \right\}_x, \qquad \left\{ g, l_e \right\}_a = \sum_i \frac{\partial g}{\partial x_i} \left\{ x_i,  l_e \right\}_a.\] Therefore, it suffices to prove \eqref{Eq:LinSemiNiceForm} for linear functions $g(x) = g_u(x) = \langle u, x \rangle, u \in \mathfrak{g}$. Since $l_e(x)$ is a semi-invariant, \begin{equation} \label{Eq:LinSemiXBrack} \left\{ g_u, l_e\right\}_x = \chi_e(u) l_e(x). \end{equation} Since $g_u$ and $l_e$ are linear functions on $\mathfrak{g}^*$,  \[ \left\{ g_u, l_e\right\}_x  = \langle x, [u, e] \rangle, \qquad  \left\{ g_u, l_e\right\}_a  = \langle a, [u, e] \rangle. \] We get that \begin{equation} \label{Eq:LinSemiABrack} \left\{ g_u, l_e\right\}_a = \chi_e(u) l_e(a). \end{equation} \eqref{Eq:LinSemiNiceForm} follows from \eqref{Eq:LinSemiXBrack}  and \eqref{Eq:LinSemiABrack}. Lemma~\ref{L:LinSemiInvPoissEquality} is proved. \end{proof}

\subsubsection{Differential of a characteristic number} \label{S:DifEigenDifGeom}

Let us prove an important fact about eigenvalues of Poisson pencils on a manifold $M$.  

\begin{lemma} \label{L:EigenDiff} Let $\mathcal{P} = \left\{ \mathcal{A} + \lambda \mathcal{B} \right\}$ be a Poisson pencil on a manifold $M$. For any JK-regular point $x \in (M, \mathcal{P})$ and any finite eigenvalue $\lambda_i(x) < \infty$ we have \begin{equation} \label{Eq:Eigen1} (\mathcal{A} - \lambda_i (x) \mathcal{B}) d\lambda_i (x) =0.\end{equation} \end{lemma}

In terms of the Schouten--Nijenhuis bracket the formula~\eqref{Eq:Eigen1} can also be written as \[ \left[ \mathcal{A} - \lambda_i (x) \mathcal{B}, \lambda_i (x) \right] = 0. \] 

\begin{remark} Although we were not able to find the statement of Lemma~\ref{L:EigenDiff} in the literature, 
it is well-known to the experts in the field. For nondegenerate pencils Lemma~\ref{L:EigenDiff} follows from a similar statement about eigenvalues of Nijenhuis operators (see Proposition 2.3 from \cite{BolsinovN1}). For the proof of that fact for Lie-Poisson pencils see, e.g., Proposition 5.2 of \cite{Izosimov14}. We also prove a similar statement for semi-invariants of Lie algebras in Section~\ref{S:RootSemiInv}. \end{remark}

\begin{proof}[Proof of Lemma~\ref{L:EigenDiff}] Fix a JK-regular point $x_0$ and denote $\lambda_i(x_0) = \lambda_0$. Since $x_0$ is regular, the rank of the pencil $\mathcal{P} =  \left\{\mathcal{A} + \lambda \mathcal{B}\right\}$ is constant in a neighborhood $Ox_0$. Notice that $\lambda_0$ is an eigenvalue at $x\in Ox_0$ iff \[ \operatorname{rk}\left(\mathcal{A} -  \lambda_0 \mathcal{B}\right) \bigr|_{x} < \operatorname{rk} \mathcal{P}.  \]  Let $S$ be the symplectic leaf of $\mathcal{A} -  \lambda_0 \mathcal{B}$ through $x_0$. Then \[ \operatorname{dim} S = \operatorname{dim}  \operatorname{Im} \left(\mathcal{A} - \lambda_0  \mathcal{B}\right) \bigr|_{x} = \operatorname{rk} \left(\mathcal{A} - \lambda_0 \mathcal{B}\right) \bigr|_{x} \] for any $x \in S$. Thus $\lambda_0$ is an eigenvalue on $S$. Since locally eigenvalues are distinct, $\lambda_i(x) = \lambda_0$ on $S$. We get that $d \lambda_i(x) =0$ on $T_{x} S = \operatorname{Im}\left(\mathcal{A} - \lambda_0 \mathcal{B}\right)\bigr|_{x}$ for $x \in S$. Thus, \[ d \lambda_i(x_0) \in \left( \operatorname{Im}\left(\mathcal{A} - \lambda_0 \mathcal{B}\right) \bigr|_{x_0} \right)^0  =  \operatorname{Ker}\left(\mathcal{A} - \lambda_0 \mathcal{B}\right) \bigr|_{x_0}, \] which proves \eqref{Eq:Eigen1} and Lemma~\ref{L:EigenDiff}. \end{proof}

As a consequence of Lemma~\ref{L:EigenDiff} we also get the well-known fact that the eigenvalues $\lambda_i(x)$ are in involution w.r..t all the brackets $\mathcal{A}_{\lambda}$. It suffices to use the JK theorem and note that the subspaces $\operatorname{Ker} \left(\mathcal{A} - \lambda_i (x) \mathcal{B}\right)$ are pairwise othogonal (w.r.t. all forms $\mathcal{A}_{\lambda}$).   

\begin{corollary} \label{Cor:EigenCommute} In a neighbourhood of a JK-regular point $x$ the eigenvalues $\lambda_i(x)$ of a Poisson pencil $\mathcal{P} = \left\{ \mathcal{A} + \lambda \mathcal{B} \right\}$ are in involution w.r.t. all forms $\mathcal{A}_{\lambda} \in \mathcal{P}$: \[ \left\{ \lambda_i(x), \lambda_j(x) \right\}_{\lambda} = 0, \qquad \forall \lambda \in \mathbb{CP}^1. \]  \end{corollary}

\subsubsection{Roots of semi-invariants} \label{S:RootSemiInv}

Let $f(x)$ be a semi-invariant with weight $\chi_f$ on $\mathfrak{g}$ and $a \in \mathfrak{g}^*$. In this section we consider the shift of the semi-invariant $f(x - \lambda a)$ and study the roots of the equation \[f(x - \lambda(x) a) = 0.\] 
First, let us show that a formular similar to \eqref{Eq:SemiDef} holds for $f(x-\lambda a)$ and the bracket $\left\{\cdot,\cdot \right\}_{x - \lambda a}$. We can replace  in \eqref{Eq:SemiDef} the point $x$ with $x -\lambda a$ and get \[\left\{g(x - \lambda a),  f(x - \lambda a)  \right\}_{x - \lambda a} = \chi_f(dg(x - \lambda a)) f(x - \lambda a). \] If we put $h(x) = g(x - \lambda a)$, we get \begin{equation} \label{Eq:SemiInvShiftSemiInv} \left\{ h(x), f(x - \lambda a)  \right\}_{x - \lambda a} = \chi_f(dh(x)) f(x - \lambda a). \end{equation} Formula \eqref{Eq:SemiInvShiftSemiInv} can also be written as \begin{equation} \label{Eq:SemiInvShiftSemiInvBrack}\langle x - \lambda a, [ dh(x), df(x-\lambda a) ]\rangle  = \chi_f (dh(x)) f(x - \lambda a).\end{equation}  The next statement generalizes Lemma~\ref{L:LinSemiInvPoissEquality}.

\begin{lemma} \label{L:EigenSemiInv} Let $f(x)$ be a semi-invariant on $\mathfrak{g}$ and $a \in \mathfrak{g}^*$. Assume that $\lambda(x)$ is a solution of $f(x - \lambda(x) a) = 0$, which is locally analytic on  $U \subset \mathfrak{g}^*$. Then for any function $g(x)$ and any point $x \in U$ we have 
\begin{equation} \label{Eq:EigenSemiInv} \left\{\lambda, g \right\}_x  = \lambda \left\{\lambda, g \right\}_a.\end{equation}
\end{lemma}

Note that in terms of Poisson brackets \eqref{Eq:EigenSemiInv}  is  the formula \eqref{Eq:Eigen1}.

\begin{proof}[Proof of Lemma~\ref{L:EigenSemiInv}] Let us use the explicit formulas for the Lie-Poisson and frozen argument brackets: \[  \left\{\lambda, g \right\}_x  = \langle x, [ d\lambda, dg]\rangle, \qquad  \left\{\lambda, g \right\}_a  = \langle a, [ d\lambda, dg]\rangle.\]   Then \[ \left\{\lambda, g \right\}_x  - \lambda \left\{\lambda, g \right\}_a  = \langle x  - \lambda a, [ d\lambda, dg]\rangle. \] For a generic $x \in U$ we have $df(x - \lambda a) \not = 0$, thus we can use formular \eqref{Eq:DLambda} for $d \lambda$. We get \[  \left\{\lambda, g \right\}_x  - \lambda \left\{\lambda, g \right\}_a = \frac{1}{\langle df(x- \lambda a), a\rangle} \langle x  - \lambda a, [ df(x-\lambda a), dg(x)]\rangle.\] Using \eqref{Eq:SemiInvShiftSemiInvBrack} we get \[  \left\{\lambda, g \right\}_x  - \lambda \left\{\lambda, g \right\}_a = -\frac{1}{\langle df(x- \lambda a), a\rangle} \chi_f(dg(x)) f(x - \lambda a) = 0.\] The last equality holds since $f(x - \lambda a) = 0$, by the definition of $\lambda$. Lemma~\ref{L:EigenSemiInv} is proved. \end{proof} 

\subsection{Independance of linear semi-invariants} \label{S:IndepLinSemiInv} 

\begin{lemma}  \label{L:LinIndSemiInv} Let $e_1, \dots, e_N$ be linear semi-invariants of a Lie algebra $\mathfrak{g}$ that are pairwise not proportional $e_j \not = c e_i, c \in \mathbb{C}$. If in the JK invariants of $\mathfrak{g}$ there are no Kronecker $3 \times 3$ blocks and no more than one Kronecker $1\times 1$ block, then $e_1, \dots, e_N$ are linearly independant. \end{lemma}

\begin{proof}[Proof of Lemma~\ref{L:LinIndSemiInv}]  The proof is in several steps.

\begin{enumerate}

\item Without loss of generality  $e_1, \dots, e_k$ are linearly independent and \begin{equation} \label{Eq:EigenLinDep} e_{k+1} = c_1 e_1 + \dots + c_k e_k,\end{equation} where all $c_i \not = 0$. Note that $k\geq 2$, since $e_j \not = c e_i$. 

\item Each $e_i$ is a semi-invariant: \[ [u, e_i] = \chi_{i} (u) \cdot e_i.\] From \eqref{Eq:EigenLinDep}  we get  \begin{equation} \label{Eq:EqualL} \chi_1 = \chi_2 = \dots = \chi_{k+1}. \end{equation} 

\item Extend $e_1, \dots, e_k$ to a basis $g_1, \dots, g_{n-k}, e_1, \dots, e_k,$ of $\mathfrak{g}$. Let $(y_1,\dots, y_{n-k}, x_1, \dots, x_k)$ be the corresponding coordinates. Then the Lie-Poisson bracket has the form \begin{equation} \label{Eq:LiePoissObstRank1}  \mathcal{A}_x = \left( \begin{matrix} * & M_x \\ -M_x^T & 0 \end{matrix} \right), \quad M_x = \left( \begin{matrix} \alpha_1 x_1 & \dots & \alpha_1 x_k \\ \alpha_2 x_1 & \dots & \alpha_2 x_k \\ \vdots & \vdots & \vdots \\ \alpha_{n-k} x_1 & \dots & \alpha_{n-k} x_k \end{matrix} \right) = \left(\begin{matrix} \alpha_1 \\ \vdots \\ \alpha_{n-k}  \end{matrix} \right) \left(\begin{matrix} x_1 & \cdots & x_{k}  \end{matrix} \right).  \end{equation}  The form of $M_x$ is due to \eqref{Eq:EqualL}.

\item  Note that for \eqref{Eq:LiePoissObstRank1} we have $\operatorname{rk} M_x =1$. It is easy to see that \[\operatorname{dim} \left(\operatorname{Ker} \mathcal{A}_{x + \lambda a}  \bigcap \operatorname{Ker} \mathcal{A}_{x + \mu a} \right) \geq k-2\] for any $\lambda, \mu \in \mathbb{C}, \lambda \not = \mu$. By Proposition~\ref{P:KronckerKernels} there are at least $k-2$ Kronecker $1 \times 1$ blocks. By condition, there are no more than one Kronecker $1 \times 1$ block. Thus, $k =2$ or $k=3$. 

\item Consider the case $k = 3$. Note that $\chi_i \not = 0$, i.e. $ \left(\begin{matrix} \alpha_1, \dots, \alpha_{n-k}  \end{matrix} \right) \not = 0$. Otherwise, there would be at least $3$ trivial Kronecker $1\times 1$ blocks, which is forbidden. We claim that there are two Kronecker blocks with sizes not bigger than $3 \times 3$. Consider solutions $v(\lambda)$ of the equation \begin{equation} \label{Eq:PolSolProofF} \left( \mathcal{A}_x + \lambda \mathcal{A}_a\right) v(\lambda) = 0\end{equation} that are polynomial in $\lambda$. There are solutions of the form \[v(\lambda) = \left(\begin{matrix} 0_{n-k} \\ w(\lambda) \end{matrix} \right),\] where $w(\lambda)$ is the solution of \[ \left(M_x + \lambda M_a\right) w(\lambda) = 0.\] This equation  is equivalent to \begin{equation} \label{Eq:Mxa} (x + \lambda a) w(\lambda) = 0, \qquad x, a \in \mathbb{C}^3.\end{equation} 

\begin{proposition} \label{P:SolXAC3} For a generic pair $(x, a) \in \mathbb{C}^3 \times \mathbb{C}^3$ there are solutions $w_1 (\lambda)$ and $w_2(\lambda)$ of \eqref{Eq:Mxa} such that \[ \deg w_1(\lambda) = 0, \qquad \deg w_2(\lambda) = 1\] and $w_1(0)$ and $w_2(0)$ are linearly independant. 
\end{proposition}

\begin{proof}[Proof of Proposition~\ref{P:SolXAC3}] For any matrix $D \in \operatorname{GL}(3)$ we have \[ (x+\lambda a) w(\lambda) = (x+\lambda a) D D^{-1} w(\lambda). \] Since the pair $(x, a)$ is generic, without loss of generality, we may assume that \[ x = (1, 0, 0), \qquad a = (0, 1, 0). \] For these $x$ and $a$ we have solutions \[w_1(\lambda)  = (0, 0, 1), \qquad w_2(\lambda) = (\lambda, -1, 0).\]
Proposition~\ref{P:SolXAC3} is proved.  \end{proof} 

We got two solutions of \eqref{Eq:PolSolProofF} with degrees $0$ and $1$ such that their initial values are linearly independant. By Corollary~\ref{C:DegPol} there are two Kronecker blocks with sizes not bigger than  $3 \times 3$. We get a contradiction with the statement of the theorem.

\item Only the case $k=2$ left. By Lemma~\ref{L:LinSemiInvPoissEquality} for a generic pair $(x, a)$ we have \[e_i \in \operatorname{Ker} \left(\mathcal{A}_x -  \frac{\langle e_i, x \rangle}{\langle e_i, a\rangle } \mathcal{A}_a\right).\] Note that for a generic pair $(x, a)$ we have $\langle e_i, a\rangle \not = 0$, and \[\frac{\langle e_i, x \rangle}{\langle e_i, a \rangle} \not =  \frac{\langle e_j, x \rangle}{\langle e_j, a \rangle},\] since $e_i \not = c e_j$. By \eqref{Eq:EigenLinDep} the vectors $e_1, e_2$ and $e_3$ are linearly dependent. Thus, by Proposition~\ref{P:KronckerKernels} the vectors $e_i$ belong to the sum $U$ of Kronecker $1\times1$ and $3\times3$ blocks. By the statement of the theorem, $\dim U \leq 1$. But the vectors $e_1,e_2 \in U$ are linearly independant. We get a contradiction.

\end{enumerate}

Lemma~\ref{L:LinIndSemiInv}  is proved.

\end{proof}

\subsection{Proof of Theorem~\ref{Th:FinalObs}} \label{S:ProofThGenObs}

\begin{proof}[Proof of Theorem~\ref{Th:FinalObs}]
Let $J_{\lambda_i}(2n_{i1}, \dots, 2n_{is_i})$, $i = 1,\dots, J_{um}$ be unique Jordan tuples with multiple maxima. Assume that $J_{um} > k_1 + \dots + k_q$. On one hand, we have the following fact from differential geometry.

\begin{itemize}

\item For any Poisson pencil $\mathcal{A}+\lambda\mathcal{B}$ and any Jordan tuple without multiple maxima $J_{\lambda}(2n_{1}, \dots, 2n_{s_i})$ the differential of the eigenvalue $d \lambda$ belongs both to the core $\mathcal{K}$ (Lemma~\ref{L:EigenCore}).  

\end{itemize}

Thus, $d \lambda_1, \dots, d\lambda_{J_{um}}$ are linearly dependant: \[ c_1 d \lambda_1 + \dots + c_{J_{um}} d\lambda_{J_{um}} = 0.\] On the other hand, JK invariants put restrictions on semi-invariants of the Lie algebra:

\begin{itemize}

\item If $J_{\lambda}(2n_{1}, \dots, 2n_{s})$ is a unique Jordan tuple, then $\lambda$ is a linear semi-invariant (Lemma~\ref{L:LinEigen}).

\item If there are no Kronecker $3 \times 3$ blocks and no more than one Kronecker $1 \times 1$ block, then the linear semi-invariants $\lambda_1, \dots, \lambda_{J_{um}}$ (which are pairwise not proportional $\lambda_j \not = c \lambda_i$ by Lemma~\ref{L:TwoEigen}) are linearly independant (Lemma~\ref{L:LinIndSemiInv}).

\end{itemize}

Since $\lambda_i$ are linear functions, the differentials $d \lambda_1, \dots, d\lambda_{J_{um}}$ are also linearly independant \[ c_1  d \lambda_1 + \dots + c_{J_{um}} d \lambda_{J_{um}} \not = 0.\]. We get a contradiction. Theorem~\ref{Th:FinalObs} is proved.  \end{proof}

\section{Generalized Jordan-Kronecker invariants}

Note that in Theorem~\ref{Th:FinalObs} we take Jordan tuples with multiple maxima only to apply Lemma~\ref{L:EigenCore}. We can slighly generalize Theorem~\ref{Th:FinalObs} if we specify for each eigenvalue whether \begin{equation} \label{Eq:InCore} d \lambda_i(x)\in \mathcal{K}\end{equation} holds or not. We also want to empasize that the property $d \lambda_i \in \mathcal{K}$ plays the crucial role in Theorem~\ref{Th:FinalObs} and the decomposition Theorem~\ref{T:TurielDecompTrivKronOneJord}.

\begin{itemize}

\item In Section~\ref{S:CoreJord} and  Section~\ref{S:CoreJordLIe} we formalize condition \eqref{Eq:InCore} for Poisson pencils and Lie algebras respectively. We introduce the notion of core Jordan tuples\footnote{The idea may be more important here than the technical realization.}.

\item In Section~\ref{S:InvStab} we discuss two simple ways to check \eqref{Eq:InCore} for Lie algebras. Let $\lambda_i$ be a root of an irreducible polynomial $f_j(x - \lambda_i a ) = 0$, let $S_j$ be the corresponding irreducible component of $\operatorname{Sing}_0$. Then  \eqref{Eq:InCore} holds in the following cases:

\begin{enumerate}

\item $f_j(x )$ is an invariant (see Section\ref{S:InvLieAlgCore}). 

\item a generic singular point $y \in S_j$ has some specific stabilizers (see Section~\ref{S:StabLieCore}).

\end{enumerate}

\item In Section~\ref{S:RestCoreJord} and \ref{S:LieNoInvRest} we generalize Theorem~\ref{Th:FinalObs} for core Jordan tuples and for Lie algebras without (proper) semi-invariants respectively. 

\item Finally, in Section~\ref{S:UnimInd1} we consider JK invariants for Lie algebras without (proper) semi-invariants and with one Kronecker block.

\end{itemize}

\subsection{Core Jordan tuples} \label{S:CoreJord}

In this section we consider a Poisson pencil $\mathcal{P} = \left\{ \mathcal{A} + \lambda \mathcal{B}\right\}$ on a manifold $M$. We introduce the following notion.

\begin{definition} \label{Def:CoreJord} A Jordan tuple $J_{\lambda(x)} (2n_1, \dots, 2n_s)$ is a \textbf{core Jordan tuple} at a JK-regular point $x \in M$ if \begin{equation} \label{Eq:CharCore}  d\lambda (x) \in \mathcal{K} =  \bigoplus_{\lambda - \text{reg.}} \operatorname{Ker} \left(\mathcal{A} + \lambda \mathcal{B}\right).  \end{equation} A characteristic number $\lambda(x)$ that satisfies \eqref{Eq:CharCore} is a \textbf{core characteristic number} at the point $x \in M$. We denote core Jordan tuples as $J_{\lambda(x), \mathcal{K}}(2n_1, \dots, )$. The \textbf{generalized Jordan-Kronecker invariants} are the JK invariants with an additional information: what Jordan tuple are core Jordan tuples\footnote{In practice, some partial information can be specified. We may know that some Jordan tuples are core tuples, but for a majority of the tuples it may be unknown. }. 
\end{definition}

The next example shows a possible notations for generalized JK invariants.

\begin{example} Assume that the JK invariants are \[J_{\lambda_1} (4), \qquad J_{\lambda_2}(2, 2), \qquad k_1 = 1, \qquad k_2 = 1, \qquad k_3 = 3, \] where $\lambda_2$ is a core characteristic number and $\lambda_1$ is not (at all points of $M$). Then the corresponding generalized JK invariants are \[J_{\lambda_1} (4), \qquad J_{\lambda_2, \mathcal{K}}(2, 2), \qquad k_1 = 1, \qquad k_2 = 1, \qquad k_3 = 3. \] \end{example}

\subsection{Core Jordan tuples for Lie algebras} \label{S:CoreJordLIe}

Now, we want an analogue of Definition~\ref{Def:CoreJord} for a Lie algebra $\mathfrak{g}$. It is more convenient to work with the irreducible factors $f_j(x)$ of the fundamental semi-invariant $p_{\mathfrak{g}}$, then with the eigenvalues $\lambda_i(x,a)$. The factors $f_j(x)$ are polynomials, whereas $\lambda_i(x,a)$ are locally analytic functions that are defined up to permutation. Hence, we reformulate \eqref{Eq:CharCore} in terms of the factors $f_j(x)$.

Denote the set of regular elements by $\mathfrak{g}_{\operatorname{reg}}^* = \mathfrak{g}^* - \operatorname{Sing}$. Take any regular $a \in \mathfrak{g}_{\operatorname{reg}}^*$. Consider the \textit{algebra of (polynomial) shifts} $\mathcal{F}_a$ from \cite{BolsZhang}. In other words, we take local analytic invariants $g_i(x)$, then take their Taylor
expansions \[g_i(a + \lambda x) = g_i^{(0)}(x)  + \lambda g_i^{(1)} (x) + \dots. \] The algebra $\mathcal{F}_a$ is generated by the homogeneous polynomials $g_i^{(k)}(x)$. By Remark 1 and Theorem 3 from \cite{BolsZhang} we have the following. 

\begin{proposition} Fix any regular $a \in\mathfrak{g}_{\operatorname{reg}}^*$. For any $x \in \mathfrak{g}^*$ the core distribution $\mathcal{K}_a$ of the Lie--Poisson pencil $\mathcal{P} = \left\{ \mathcal{A}_{x + \lambda a}\right\}$ is  \[\mathcal{K}_a(x) = d \mathcal{F}_a(x) = \operatorname{span} \left\{df(x) \,\, \bigr|\,\, f \in \mathcal{F}_a \right\}.\] \end{proposition}

The condition \eqref{Eq:CharCore} takes the form \begin{equation} \label{Eq:EignCoreLie1} \frac{\partial \lambda}{\partial x}(x,a) \in d \mathcal{F}_a(x).\end{equation} Now, let $f_j(x)$ be an irreducible factor of $p_{\mathfrak{g}}$ and $S_j$ be the corresponding irreducible component of $\operatorname{Sing}_0$. Assume that $\lambda(x,a)$ is a root of \[ f_j(x - \lambda(x,a) a) = 0.\] By Proposition~\ref{P:DiffRoot}, for a generic $x$ the derivative $\displaystyle \frac{\partial \lambda}{\partial x}(x,a)$ is proportional to $df(x -\lambda(x,a)a)$. Since $\lambda(x,a)$ is a root of $f_j(x)$, the point $y =x -\lambda(x,a)a \in S_j$. Thus, \eqref{Eq:EignCoreLie1} can be rewritten as \begin{equation} \label{Eq:CompCond2}  df_j(y) \in  d \mathcal{F}_a(y), \qquad \forall y \in S_j = \left\{x \,\bigr|\,  f_j(x) =0\right\}. \end{equation} 
We got a condition \eqref{Eq:CompCond2}  on the factors $f_j(x)$. Let us show that there are $2$ types of factors $f_j(x)$.

\begin{proposition} \label{P:TwoTypesIrr} Let $f_j(x)$ be an irreducible factor of the fundamental semi-invariant $p_{\mathfrak{g}}$ and $S_j$ be the corresponding irreducible component of $\operatorname{Sing}_0$. There exists an open dense subset $U \subset S_j \times \mathfrak{g}^*_{\mathrm{reg}}$ such that

\begin{itemize}

\item either $ df_j(y) \not \in  d \mathcal{F}_a(x)$ for all pairs $(y, a) \in U$, 

\item or $ df_j(y) \in  d \mathcal{F}_a(x)$ for all pairs $(y, a) \in U$.

\end{itemize}

\end{proposition}

\begin{proof}[Proof of Proposition~\ref{P:TwoTypesIrr} ] Let $\operatorname{Sing}_{\operatorname{sm}} \subset \operatorname{Sing}_0$  be the set of points where $\operatorname{Sing}_0$ is smooth. This set is open and dense in $\operatorname{Sing}_0$. Denote by $S_{j, \mathrm{sm}} = \operatorname{Sing}_{\operatorname{sm}}  \cap S_j$ the set of smooth points of the irreducible component $S_j$. It is well-known (see e.g. Appendix in \cite{Cooper14}) that  $S_{j, \mathrm{sm}}$ is a path-connected, dense, open subset of $S_j$. We get a connected manifold $M = S_{j, \mathrm{sm}} \times \mathfrak{g}^*_{\mathrm{reg}}$ that is open and dense in $ S_j \times \mathfrak{g}^*_{\mathrm{reg}}$. 

Consider the condition $ df_j(y) \in  d \mathcal{F}_a(x)$ on $M$. Note that, roughly speaking, ``everything is analytic''. Obviously, $df_j(y)$ is analytic on $y$. Since $d \mathcal{F}_a(x)$ is a distribution, $\dim  d \mathcal{F}_a(x) = m = \operatorname{const}$. Thus, $d \mathcal{F}_a(x)$ defines an an analytic curve $\gamma(x,a)$ in a Grassmanian $\operatorname{Gr}(m, n)$. It remains to note that the manifold $M$ is connected. Therefore,  either $ df_j(y) \in  d \mathcal{F}_a(x)$ holds for all points $(x,a) \in M$, or $ df_j(y) \not \in  d \mathcal{F}_a(x) $ on  an open dense subset $M$. Proposition~\ref{P:TwoTypesIrr} is proved.  \end{proof}

We want to distinguish Jordan tuples which correspond to the factors $f_j(x)$ of $p_{\mathfrak{g}}$ that satisfy \eqref{Eq:CompCond2}. Formally, this cannot be done for the JK invariants of Lie algebras, since (by our definition) they have no eigenvalues. Instead, we should consider JK invariants of JK-generic pencils $\mathcal{P}_{x,a}$ (see Remark~\ref{Rem:JKGenPenc}).

\begin{definition}  Consider JK invariants of a JK-generic pencil $\mathcal{P}_{x,a} = \left\{\mathcal{A}_x + \lambda \mathcal{A}_a\right\}$. We say that a Jordan tuple $J_{\lambda} (2n_1, \dots, 2n_s)$ is a \textbf{core Jordan tuple} if the corresponding irreducible factor $f_j(x)$ of the fundamental semi-invariant $p_{\mathfrak{g}}$ satisfies \eqref{Eq:CompCond2} for a generic pair $(y,a) \in S_j \times \mathfrak{g}^*_{\mathrm{reg}}$.  \end{definition}

Since we started with the property~\ref{Eq:EignCoreLie1}, let us make sure that it is satisfied. 
\begin{proposition} \label{Prob:DifEigencore} Let $J_{\lambda(x,a)} (2n_1, \dots, 2n_s)$ be a core Jordan tuple for a JK-generic pencil $\mathcal{P}_{x,a} = \left\{\mathcal{A}_x + \lambda \mathcal{A}_a\right\}$. Assume that the characteristic number $\lambda(x,a)$ is locally analytic on an open subset $U \subset \mathfrak{g}^* \times \mathfrak{g}^*$. Then on $U$ we have \begin{equation} \label{Eq:CondEigenProp} \frac{\partial \lambda}{\partial x}(x,a) \in d \mathcal{F}_a(x).\end{equation}
\end{proposition}

\begin{proof}[Proof of Proposition~\ref{Prob:DifEigencore}] Indeed, for a generic pair $(x,a)$ we have the following:

\begin{itemize}

\item  By Proposition~\ref{P:DiffRoot}, $\displaystyle \frac{\partial \lambda}{\partial x}(x,a)$ is proportional to $df(x -\lambda(x,a)a)$.

\item By Proposition~\ref{P:TwoTypesIrr}, $df(x -\lambda(x,a)a) \in  d \mathcal{F}_a(x)$.

\end{itemize}

Since \eqref{Eq:CondEigenProp} holds on an open dense subset $W \subset U$, it holds on the whole $U$. Proposition~\ref{Prob:DifEigencore} is proved. \end{proof}

\subsection{Invariants and stabilizers} \label{S:InvStab}

So far, we got the condition~\eqref{Eq:CompCond2} on irreducible factors $f_j(x)$ of $p_{\mathfrak{g}}$. It means that \textit{the differential $df_j(x)$ and the differentials of shifts of invariants $d g (x)$, $g(x) \in \mathcal{F}_a(x)$ are linearly dependant if $f_j(x) = 0$.} The next natural question is: \begin{itemize}

\item How to effectively check \eqref{Eq:CompCond2}? 

\end{itemize} In this section we provide a couple of simple statements that allow us to check \eqref{Eq:CompCond2} in some special cases. 

\subsubsection{Invariants of Lie algebras} \label{S:InvLieAlgCore}

The first idea is really simple. If $f$ is a shift of an invariant, i.e. $f \in \mathcal{F}_a$, then $df \in d \mathcal{F}_a$. What functions belong to all the shifts $\mathcal{F}_a$? At least, all the invariants $f(x) \in \mathbb{C}[\mathfrak{g}]^{\mathfrak{g}}$.

\begin{proposition} \label{P:InvDifCore} Any $\operatorname{Ad}^*$-invariant polynomial $f(x)$ satisfies \eqref{Eq:CompCond2} for any $a \in \mathfrak{g}^*$. \end{proposition}

\begin{corollary} \label{Cor:InvCoreTuple} If an irreducible factor $f_j(x)$ of $p_{\mathfrak{g}}$ is an invariant, then for all the roots of $f_j(x - \lambda_i a) = 0$ the corresponding Jordan tuples  $J_{\lambda_i} (2n_{i1}, \dots, 2n_{is_i})$ are core Jordan tuples. \end{corollary}

\begin{remark} Not all functions $f(x)$ that satisfy \eqref{Eq:CompCond2} are invariants.  Recall the Lie algebra $\mathfrak{g}$ from Theorem~\ref{Th:RealOneKronSeveralEigen}. The Casimir function for that Lie algebra was \[ g(x, y) = \frac{y_1 \dots y_m}{y_0}\] (see also Section~\ref{S:Kron}). The irreducible factors of the fundamental semi-invariant $p_{\mathfrak{g}}$ were the linear coordinates \[ y_{0}, \qquad \dots, \qquad y_{d-1}.\] It is easy to see that $dg$ and $dy_j$ are linearly dependant if $y_j = 0$. But $y_j$ are not invariants (they are semi-invariants). \end{remark}

\subsubsection{Stabilizer of a generic singular point}\label{S:StabLieCore}

In this section we study the stabilizers of a generic point of an irreducible component $S_j \subset \operatorname{Sing}_0$. Sometimes, that gives us information about the corresponding JK invariants. For a point $x \in \mathfrak{g}^*$   we denote the stationary subalgebra of $x$ in the sense of the coadjoint representation we denote as $\mathfrak{g}_x$.  The stabilizers $\mathfrak{g}_x$ for generic $x \in S_j$ were described in \cite{Izosimov14Derived}: \begin{itemize}

\item By \cite[Corollary 2.2]{Izosimov14Derived}, for generic $x \in S_j$, \begin{equation} \label{Eq:DimDer1} \dim [\mathfrak{g}_x, \mathfrak{g}_x] =1.\end{equation} 

\item The full list of stabilizers $\mathfrak{g}_x$ that satisfy \eqref{Eq:DimDer1} is given in \cite[Example 2.3]{Izosimov14Derived}.

\end{itemize}

It was proved the following:

\begin{lemma}[A.\,M.~Izosimov, \cite{Izosimov14Derived}] \label{L:StabSingGen} For each irreducible component $S_i$ of the variety $\operatorname{Sing}_0$, there exists an open subset $\hat{S}_i$ such that the stabilizers $\mathfrak{g}_x$, $x \in \hat{S}_i$, are all isomorphic and are one of the following Lie algebras:

\begin{enumerate}

\item $\mathbb{C}^{d_1}$, i.e. the abelian Lie algebra,

\item $\mathfrak{aff}(1) \oplus \mathbb{C}^{d_2}$, where $\mathfrak{aff}(1)$ is the two-dimensional non-abelian Lie algebra,

\item $\mathfrak{h}_{2m+1} \oplus \mathbb{C}^{d_3}$, where  $\mathfrak{h}_{2m+1}$ is the $(2m + 1)$-dimensional Heisenberg algebra. 

\end{enumerate}

\end{lemma}

In the case $\mathfrak{g}_x = \mathfrak{h}_{2m+1} \oplus \mathbb{C}^{d_3}$, we assume $m \geq 1$. Let us restate Proposition 12 from \cite{BolsZhang} about the number of Jordan blocks.

\begin{proposition}[A.\,V.~Bolsinov, P.~Zhang, \cite{BolsZhang}] \label{P:NumJordBlocks} Consider Jordan tuples $J_{\lambda_i}(2n_{i1}, \dots, 2n_{is_i})$ of a JK-generic pencil  $\mathcal{P}_{x,a} =  \left\{\mathcal{A}_x + \lambda \mathcal{A}_a\right\}$.

\begin{enumerate}

\item  The number $s_i$ of Jordan $\lambda_i$-blocks is equal to \[ \frac{1}{2} \left( \dim \mathfrak{g}_{x - \lambda_i a} - \operatorname{ind} \mathfrak{g}\right).\]

\item The number of non-trivial Jordan $\lambda_i$-blocks (i.e. of size greater than $2 \times 2$) is equal
to \[\frac{1}{2} \left( \operatorname{ind} \mathfrak{g}_{x - \lambda_i a} - \operatorname{ind} \mathfrak{g}\right).\]

\end{enumerate}

\end{proposition}

Consider Jordan tuples $J_{\lambda_i}(2n_{i1}, \dots, 2n_{is_i})$ corresponding to the roots of $f_j(x - \lambda_i a) = 0$. We combine the result of Lemma~\ref{L:StabSingGen} and Proposition~\ref{P:NumJordBlocks} about the number of Jordan blocks in Table~\ref{Tab:JordanBlocks}. 

\begin{table}[h!]
\centering 
\begin{tabular}{ c|c|c }
 $\mathfrak{g}_x$ & number of Jordan blocks & number of $2\times 2$ blocks \\  [15pt] \hline
 $\mathbb{C}^{d_1}$ &  $\displaystyle \dfrac{\mathstrut 1}{\mathstrut 2} \left( d_1 - \operatorname{ind} \mathfrak{g}\right)$ & 0 \\ [15pt] 
 \hline
 $\mathfrak{aff}(1) \oplus \mathbb{C}^{d_2}$ &  $\displaystyle 1 + \dfrac{\mathstrut 1}{\mathstrut 2} \left( d_2 - \operatorname{ind} \mathfrak{g}\right)$ & 1 \\ [15pt]  
 \hline
 $ \mathfrak{h}_{2m+1} \oplus \mathbb{C}^{d_3}$ & $\displaystyle  m + \dfrac{\mathstrut 1}{\mathstrut 2} \left( d_3 + 1 - \operatorname{ind} \mathfrak{g}\right)$ & m \\ [15pt] \hline
\end{tabular}
\caption{Jordan blocks and stabilizers of singular points}
\label{Tab:JordanBlocks}
\end{table}

Note that we have two cases for $\mathfrak{g}_x$, when there is  only one  Jordan block and it has size $2 \times 2$:
\[ \mathfrak{aff}(1) \oplus \mathbb{C}^{\operatorname{ind} \mathfrak{g}} \qquad \text{ or } \qquad \mathfrak{h}_{3} \oplus \mathbb{C}^{\operatorname{ind} \mathfrak{g} -1}. \] As it turns out, these two cases are distinguished by the condition~\eqref{Eq:CompCond2}. For $\mathfrak{h}_{3} \oplus \mathbb{C}^{\operatorname{ind} \mathfrak{g} -1}$ the Jordan tuples are core tuples, and for $\mathfrak{aff}(1) \oplus \mathbb{C}^{\operatorname{ind} \mathfrak{g}}$ they are not. We prove a slightly more general statement.

\begin{lemma} \label{L:CoreTupStab} Let $S_j $ be an irreducible component of $\operatorname{Sing}_0$, given by the irreducible factor $f_j(x)$ of $p_{\mathfrak{g}}$. Consider Jordan tuples $J_{\lambda_i}(2n_{i1}, \dots, 2n_{is_i})$ corresponding to the roots of $f_j(x - \lambda(x,a) a) = 0$.

\begin{enumerate}

\item If there exists $x \in S_j$ such that $\mathfrak{g}_x \simeq \mathfrak{aff}(1) \oplus \mathbb{C}^{d_2}$, then \[ df_j(y) \not \in d \mathcal{F}_a(y)\] for a generic pair $(y,a) \in S_j \times \mathfrak{g}^*_{\operatorname{red}}$. The Jordan tuples   $J_{\lambda_i}(2n_{i1}, \dots, 2n_{is_i})$  cannot be core tuples.

\item Assume that \[\mathfrak{g}_x \simeq\mathfrak{h}_{2m+1} \oplus \mathbb{C}^{d_3}, \qquad   d_3  = \operatorname{ind} \mathfrak{g} - 1,\] for a generic $x \in S_i$. Then all Jordan $\lambda_i$-blocks have size $2 \times 2$. Moreover, \eqref{Eq:CompCond2} holds for a generic pair $(y,a) \in S_j \times \mathfrak{g}^*_{\operatorname{red}}$. In other words, the Jordan tuples have the form \[J_{\lambda_i}(\underbrace{2, \dots, 2}_{m}),\] and they are core Jordan tuples.

\end{enumerate}

\end{lemma}

\begin{proof}[Proof of Lemma~\ref{L:CoreTupStab}]  By Proposition~\ref{P:DiffRoot}, it suffices to check whether $d \lambda_i(x) \in d\mathcal{F}_a$ or not.  We use the following two facts \eqref{Eq:CapDF} and\eqref{Eq:DLamDF}  from \cite[Section 7]{Izosimov14}: 

\begin{itemize}

\item  Denote $\mathfrak{g}_i(x) = \mathfrak{g}_{x -\lambda_i a}$, then  \begin{equation} \label{Eq:CapDF} \mathfrak{g}_{i} \cap d\mathcal{F}_a \subset \mathcal{Z}\left(\mathfrak{g}_i\right).\end{equation}

\item If $\mathfrak{g}_i$ is not Abelian, then \begin{equation} \label{Eq:DLamDF} d \lambda_i (x) \in [ \mathfrak{g}_i, \mathfrak{g}_i], \qquad d \lambda_i(x) \not = 0.\end{equation}

\end{itemize}

First, consider the case $  \mathfrak{g}_x \simeq \mathfrak{aff}(1) \oplus \mathbb{C}^{d_2}$. Using \eqref{Eq:CapDF} we get \[[ \mathfrak{g}_i, \mathfrak{g}_i] \cap \mathcal{Z}(\mathfrak{g}_i) = 0 \qquad \Rightarrow \qquad d \lambda_i(x)  \not \in d\mathcal{F}_a.\] Finally, consider the case $ \mathfrak{g}_x \simeq\mathfrak{h}_{2m+1} \oplus \mathbb{C}^{\operatorname{ind} \mathfrak{g} - 1}$.  We already know  (see Table~\ref{Tab:JordanBlocks}) that the Jordan tuples are $J_{\lambda_i}(\underbrace{2, \dots, 2}_{m})$. Note that \begin{equation} \label{Eq:CenterZheiz} \mathcal{Z}(\mathfrak{g}_i)  = [ \mathfrak{g}_i, \mathfrak{g}_i] \oplus \mathbb{C}^{\operatorname{ind} \mathfrak{g} - 1}. \end{equation} Since $\dim d \mathcal{F}_a = \operatorname{ind} \mathfrak{g}$, by \eqref{Eq:CapDF} we get \[ d \mathcal{F}_a = \mathcal{Z}(\mathfrak{g}_i).\] Finally, by  \eqref{Eq:DLamDF} and \eqref{Eq:CenterZheiz}: \[  d \lambda_i \in \mathcal{Z}(\mathfrak{g}_i)  = d \mathcal{F}_a.\]  Lemma~\ref{L:CoreTupStab} is proved.  \end{proof}

 \begin{remark} Statement for $ \mathfrak{g}_x \simeq\mathfrak{h}_{2m+1} \oplus \mathbb{C}^{\operatorname{ind} \mathfrak{g} - 1}$ and $m >1$ also follows from Lemma~\ref{L:EigenCore}, since the Jordan tuples $J_{\lambda_i}(\underbrace{2, \dots, 2}_{m})$  have multiple maxima.  \end{remark} 
 
 \begin{remark}  In short,  \eqref{Eq:CapDF} and \eqref{Eq:DLamDF} in \cite{Izosimov14} were proved as follows: \begin{itemize}

\item  \cite[Section 5]{Izosimov14}, Zariski dense subset $\mathcal{N} \subset \mathfrak{g}^*$ of ``nice'' elements was defined.

\item \cite[Proposition 5.3]{Izosimov14}, for any $x_0 \in \mathcal{N}$ and any $\xi, \eta \in \mathfrak{g}_i(x_0)$ it was proved that \[ [ \xi, \eta] = \langle a, [\xi, \eta] \rangle d \lambda_i(x_0).\] We get \eqref{Eq:DLamDF}, and also \begin{equation} \label{Eq:Part81a} \operatorname{Ker} \left( \mathcal{A}_a\bigr|_{\mathfrak{g}_i} \right) \subset \mathcal{Z}\left(\mathfrak{g}_i\right).\end{equation}

\item \cite[Proposition 6.2]{Izosimov14}: by the JK theorem, for a linear Poisson pencil $\left\{ A_{x+\lambda a}\right\}$  with the core subspace $K$, an eigenvalue $\lambda_i$ and a regular $a$, we have \[\operatorname{Ker} \mathcal{A}_{x-\lambda_i a} \cap K \subset \operatorname{Ker} \mathcal{A}_{a} \bigr|_{\operatorname{Ker} \mathcal{A}_{x-\lambda_i a}}.\]  Since  $\mathfrak{g}_i  = \operatorname{Ker}\left( \mathcal{A}_{x} - \lambda_i \mathcal{A}_a\right)$ and $d\mathcal{F}_a =K$, we get that \begin{equation} \label{Eq:Part81b} \mathfrak{g}_i \cap d\mathcal{F}_a \subset \operatorname{Ker} \left( \mathcal{A}_a\bigr|_{\mathfrak{g}_i} \right).  \end{equation} Combining \eqref{Eq:Part81a} and \eqref{Eq:Part81b}, we get \eqref{Eq:CapDF}. \end{itemize}
 \end{remark}

\begin{remark} Consider that case \begin{equation} \label{Eq:CondNoGuar} \mathfrak{g}_x= \mathfrak{h}_{3} \oplus \mathbb{C}^{d_3}, \qquad d_3 > \operatorname{ind} \mathfrak{g} -1.\end{equation} Then $d \mathcal{F}_a$ is a proper subspace of the center: \[ d \mathcal{F}_a \subset [ \mathfrak{g}_i, \mathfrak{g}_i] \oplus \mathbb{C}^{d_3}, \qquad \dim d \mathcal{F}_a  = \operatorname{ind} \mathfrak{g}.\] Condition~\eqref{Eq:CompCond2} holds if and only if $[ \mathfrak{g}_i, \mathfrak{g}_i]  \subset d \mathcal{F}_a$. \end{remark}

\begin{remark} Note that the following two numbers: 

\begin{itemize}

\item  the number $N$ of Jordan blocks, 

\item the number $M$ of Jordan $2\times 2$  blocks,

\end{itemize}

determine $\mathfrak{g}_x$ from Table~\ref{Tab:JordanBlocks} if $M \not =1$. Thus, in general, (without additional information) $\mathfrak{g}_x$ does not guarantee whether condition~\eqref{Eq:CompCond2} holds or not. \end{remark}

\subsection{Restrictions on core Jordan tuples} \label{S:RestCoreJord}

The next theorem is proved similarly to Theorem~\ref{Th:FinalObs}.  

\begin{theorem} \label{Th:FinalObs2} The generalized Jordan--Kronecker invariants that simultaneously satisfy the following 2 conditions cannot be realized by Lie algebras:

\begin{enumerate}

\item There are no Krocker $3 \times 3$ blocks and no more than one  Kronecker $1 \times 1$ block. 

\item If Kronecker indices are $k_1, \dots, k_q$, then there are more than $k_1 + \dots + k_q$ unique core Jordan tuples $J_{\lambda_i,\mathcal{K}}(2n_{i1}, \dots, 2n_{is_i})$.

\end{enumerate}

\end{theorem}

Theorem~\ref{Th:FinalObs} is a special case of Theorem~\ref{Th:FinalObs2}, since  by Lemma~\ref{L:EigenCore} a Jordan tuple with multiple maxima $J_{\lambda_i}(2n_{i1}, \dots, 2n_{is_i}), n_{i1} = n_{i2} \geq n_{ij}$ is always a core Jordan tuple.

\subsection{Lie algebras without semi-invariants} \label{S:LieNoInvRest}

If a Lie algebra $\mathfrak{g}$ has no (proper) semi-invariants, then all Jordan tuples are core Jordan tuples (see Corollary~\ref{Cor:InvCoreTuple}). Theorem~\ref{Th:FinalObs2} takes the following form.

\begin{theorem} \label{Th:FinalObs2} Let $\mathfrak{g}$ be a Lie algebra without (proper) semi-invariants. Consider its JK invariants. Assume that:

\begin{itemize}

\item there are no Krocker $3 \times 3$ blocks and no more than one  Kronecker $1 \times 1$ block,

\item the Kronecker indices are $k_1, \dots, k_q$.

\end{itemize}

Then, there are more than $k_1 + \dots + k_q$ unique Jordan tuples. \end{theorem}

\begin{remark} \label{Rem:NilRad} It is well-known that a \textit{Lie algebra $\mathfrak{g}$ with nilpotent radical has no proper semi-invariants}, e.g. it was mentioned in \cite[Section 2.2]{Ooms08}. 

\begin{itemize}

\item For semisimple (or perfect\footnote{A Lie algebra $\mathfrak{g}$ is perfect if $\mathfrak{g} = [ \mathfrak{g}, \mathfrak{g}]$.}) subalgebras $\mathfrak{g}$ all the characters are trivial, since $[\mathfrak{g}, \mathfrak{g}] = 0$. 

\item For nilpotent Lie algebras $\mathfrak{n}$ there are no (proper) semi-invariants by Engel's theorem. Consider the adjoint action $\operatorname{ad}: \mathfrak{n} \to \operatorname{gl}\left( S^k(\mathfrak{n})\right)$ on the vector space $S^k(\mathfrak{n})$ of homogeneous polynomials of degree $k$. Since $\mathfrak{n}$ is nilpotent and we consider the adjoint action (on polynomials), all the operators  $\operatorname{ad}(x) \in \operatorname{gl}\left( S^k(\mathfrak{n})\right)$ are nilpotent. Thus, by  Engel's theorem all the operators $\operatorname{ad}(x)$ can be simultaneously made strictly upper-triangular. Hence, all the eigenvalues of $\operatorname{ad}(x)$ are equal to $0$, and there are no (proper) semi-invariants. 

\end{itemize}

The general case follows from the semisimple and nilpotent cases, using Levi-Mal'tsev decomposition.

\end{remark}

\subsection{Lie algebras with index $1$ and no semi-invariants} \label{S:UnimInd1}

If $\operatorname{ind} \mathfrak{g} = 1$ (i.e. there is only one Kronecker block), then Theorem~\ref{Th:FinalObs2} can be made even stronger: there are either no Jordan blocks, or all Jordan blocks are equal (see Theorem~\ref{T:Ind1Th} and Remark~\ref{Rem:TrivInd1}). The main result of these section is the following. 

\begin{theorem}  \label{T:Ind1Th} Let $\mathfrak{g}$ be a Lie algebra without (proper) semi-invariants. Assume that $\operatorname{ind} \mathfrak{g} = 1$ and the fundamental semi-invariant $p_{\mathfrak{g}} \not = \operatorname{const}$. Denote \begin{equation} \label{Eq:Ind1K} k =  \frac{1}{2}\left( \dim \mathfrak{g} + \operatorname{ind} \mathfrak{g}\right) - \deg p_{\mathfrak{g}}. \end{equation} 

\begin{enumerate}

\item The algebra of invariants is polynomial $\mathbb{C}[\mathfrak{g}]^{\mathfrak{g}} = \mathbb{C}[f]$. Moreover, \begin{equation} \label{Eq:DefFk} \deg f = k.\end{equation} The fundamental semi-invariant has the form \begin{equation} \label{Eq:PgcFd} p_{\mathfrak{g}}(x) = c f(x)^d, \qquad c \in \mathbb{C}, \quad d \in \mathbb{N}.\end{equation}

\item The JK invariants of $\mathfrak{g}$ consists of 

\begin{itemize}

\item one Kronecker $(2k-1) \times (2k-1)$ block,

\item and $k$ equal Jordan tuples $J_{\lambda_i}(2n_1, \dots, 2n_s)$, $(i=1,\dots, k)$, where \begin{equation} \label{Eq:SumNInd1} n_1 + \dots + n_s = d = \frac{\deg p_{\mathfrak{g}}}{\deg f}.\end{equation}

\end{itemize}

\end{enumerate}

\end{theorem}

\begin{remark}  \label{Rem:TrivInd1} The case $p_{\mathfrak{g}} = 1$ is trivial: there are no Jordan blocks and the JK invariants of $\mathfrak{g}$ consist of one Kronecker $\dim\mathfrak{g} \times \dim \mathfrak{g}$ block. \end{remark}

First, let us recall a well-known statement about $\mathbb{C}[\mathfrak{g}]^{\mathfrak{g}}$ in the case $\operatorname{ind} \mathfrak{g} =1$.

\begin{proposition} \label{P:Ind1Pol} For any Lie algebra $\mathfrak{g}$ with $\operatorname{ind} \mathfrak{g} = 1$ the algebra of invariants is a polynomial ring:  $\mathbb{C}[\mathfrak{g}]^{\mathfrak{g}}  = \mathbb{C}[f]$ (where $f$ may be $1$). \end{proposition}

The proof of Proposition~\ref{P:Ind1Pol}  is briefly discussed in the proof of Lemma 2.3 in \cite{Yakimova17}. For complectness sake we repeat that proof with slighly more details.

\begin{proof}[Proof of Proposition~\ref{P:Ind1Pol} ]We assume that there are nonconstant invariant, since the case $\mathbb{C}[\mathfrak{g}]^{\mathfrak{g}}  = \mathbb{C}$ is trivial. It suffices to prove that for any two homogeneous nonconstant invariants $f(x), g(x) \in \mathbb{C}[\mathfrak{g}]^{\mathfrak{g}}$ there exists an invariant $p(x)$ such that \begin{equation} \label{Eq:CondDivis} f(x) = c_1 p(x)^{k_1}, \qquad g(x) = c_2 p(x)^{k_2}, \qquad c_i \in \mathbb{C}, \quad k_i \in \mathbb{N}.\end{equation} Since $\operatorname{ind} \mathfrak{g} =1$, the invariants $f(x)$ and $g(x)$ are algebraically dependant. Thus, $dg(x)$ and $df(x)$ are linearly dependant. Assume that locally \begin{equation} \label{Eq:DgcDf} dg(x) = c(x) df(x).\end{equation}  Let $\deg f(x) = a$ and $\deg g(x) = b$. By Euler's homogeneous function theorem \begin{equation} \label{Eq:CRat}  c(x) = \frac{b}{a} \cdot \frac{g(x)}{f(x)}.\end{equation}  Substituting\eqref{Eq:CRat} into \eqref{Eq:DgcDf} we get that \[ g(x)^{a} = c _0f(x)^{b}, \qquad c_0 \in \mathbb{C}.\] The degrees all factors in the factorization of the polynomial $h(x) = g(x)^{a} = c f(x)^{b}$ must be divisible by both $a$ and $b$. Hence, the degrees are divisible by $\operatorname{lcm}(a,b)$. The polynomial $p(x) = h(x)^{1/\operatorname{lcm}(a,b)}$ is an invariant such that \eqref{Eq:CondDivis} holds. Proposition~\ref{P:Ind1Pol}  is proved. \end{proof} 

The first part of Theorem~\ref{T:Ind1Th} can be proved even if we relax the condition ``$\mathfrak{g}$ has no semi-invariants'', replacing it with 2 weaker conditions: 
\begin{enumerate}

\item $\mathfrak{g}$ is \textbf{unimodular}  (i.e. $\operatorname{tr}\left( \operatorname{ad} x\right) = 0$ for all $x \in \mathfrak{g}$),

\item The fundamental semi-invariant is an invariant $p_{\mathfrak{g}} \in \mathbb{C}[\mathfrak{g}]^{\mathfrak{g}}$. 

\end{enumerate} The next statement is well-known (see e.g. \cite[Remark~3]{Ooms22}). 

\begin{proposition} If $\mathfrak{g}$ has  no proper semi-invariants (as it is if the radical of $\mathfrak{g}$ is nilpotent), then $\mathfrak{g}$ is unimodular. \end{proposition}

\begin{proof} [Proof of Theorem~\ref{T:Ind1Th}]

\begin{enumerate}

\item Since the fundamental semi-invariant $p_{\mathfrak{g}}$ is a nontrivial invariant, $\mathbb{C}[\mathfrak{g}]^{\mathfrak{g}} \not = \mathbb{C}$. By Proposition~\ref{P:Ind1Pol} the algebra of invariants is polynomial $\mathbb{C}[\mathfrak{g}]^{\mathfrak{g}} = \mathbb{C}[f]$. Since $\mathfrak{g}$ is unimodular, $p_{\mathfrak{g}} \in \mathbb{C}[\mathfrak{g}]^{\mathfrak{g}}$, and $\mathbb{C}[\mathfrak{g}]^{\mathfrak{g}}$ is polynomial, \eqref{Eq:DefFk} follows from \cite[Theorem 9]{BolsZhang}. \eqref{Eq:PgcFd} holds, because $p_{\mathfrak{g}}$ is an invariant.

\item There is one Kronecker block, since $\operatorname{ind} \mathfrak{g} = 1$ (see Proposition~\ref{P:IndGNumKron}). Since $\mathfrak{g}$ has no (proper) semi-invariants, $p_{\mathfrak{g}}$ is irreducible. By Proposition~\ref{Prop:ReducEigen}, the Jordan tuples are equal and \eqref{Eq:SumNInd1} holds  (recall that the fundamental semi-invariant $p_{\mathfrak{g}}$ has the form \eqref{Eq:PgcFd}).

\end{enumerate}

Theorem~\ref{T:Ind1Th} is proved.  \end{proof}

\cite[Theorem 9]{BolsZhang} is based on the results of A. Joseph and D. Shafrir from \cite{Joseph10}. The case $\operatorname{ind} \mathfrak{g} = 1$  is often simpler then the general case. In the next statement we give another proof of the formula \eqref{Eq:DefFk} for $\deg f$. 

\begin{lemma} \label{L:AnotherDeg} Let $\mathfrak{g}$ be Lie algebra such that

\begin{itemize}

\item $\mathfrak{g}$ is \textbf{unimodular},

\item $\operatorname{ind} \mathfrak{g} = 1$, 

\item the fundamental semi-invariant is an invariant $p_{\mathfrak{g}} \in \mathbb{C}[\mathfrak{g}]^{\mathfrak{g}}$.

\end{itemize} 

If $\mathbb{C}[\mathfrak{g}]^{\mathfrak{g}} = \mathbb{C}[f]$, then \begin{equation} \label{Eq:DegFReq} \deg f = \frac{1}{2}\left( \dim \mathfrak{g} + \operatorname{ind} \mathfrak{g}\right) - \deg p_{\mathfrak{g}}. \end{equation} \end{lemma}

\begin{proof}[Proof of Lemma~\ref{L:AnotherDeg}] Fix the standard volume form $\omega$ on $\mathfrak{g}$. The key idea is to express the generating integral $f(x)$ through the Lie--Poisson bracket as follows: \begin{equation} \label{Eq:WedgeF} \underbrace{ \mathcal{A}_x \wedge \dots \wedge  \mathcal{A}_x}_{m} = c \cdot p_{\mathfrak{g}} (x) \cdot \star df(x) , \qquad c \in \mathbb{C}, \quad  m=\frac{\dim \mathfrak{g} -1 }{2}.\end{equation} Here $ \star$ is the Hodge star. The proof is in several steps.

\begin{enumerate}

\item Consider the $1$-form obtained by contracting $\omega$ with the $n$-th exterior power of the Lie--Poisson bivector: \[ \alpha = \omega \star \Lambda^{m} \mathcal{A}_x, \qquad m = \frac{\dim \mathfrak{g} -1 }{2}.\]  Similar to  the proof of Theorem 9 in \cite{BolsZhang} it can be shown that \begin{equation} \label{Eq:Al1} \alpha(x) = p_{\mathfrak{g}}(x) \beta(x),\end{equation} where $\beta(x)$ is a $1$-form that is polynomial in $x$ of degree $k-1$. Here $k$ is given by \eqref{Eq:Ind1K}. We got that  \[  \underbrace{ \mathcal{A}_x \wedge \dots \wedge  \mathcal{A}_x}_{m} = \pm p_{\mathfrak{g}} (x) \cdot \star \beta. \] 

\item We want to prove that $\beta$ is exact: \[\beta = dg(x) \qquad \Leftrightarrow \qquad d \beta = 0.\] It is well-known that for unimodular Lie algebras the standard volume form $\omega$ is preserved by the adjoint action (see e.g. \cite[Exercise 2.6.15]{DufourZung05}). Since $\omega$ is invariant w.r.t. the Lie--Poisson bracket, by \cite[Lemma 7.1]{Izosimov16Flat} we have \begin{equation} \label{Eq:DA1} d \alpha =0, \qquad \alpha \in \operatorname{Ker} \mathcal{A}_x.\end{equation} Thus, $\beta(x) \in \operatorname{Ker} \mathcal{A}_x$. Since $p_{\mathfrak{g}}$ is an invariant, $d p_{\mathfrak{g}} \in \operatorname{Ker} \mathcal{A}_x$. Since \[\dim \operatorname{Ker} \mathcal{A}_x = \operatorname{ind} \mathfrak{g} = 1, \] the $1$-forms $d p_{\mathfrak{g}}$ and $\beta(x)$ are linearly dependant: \begin{equation} \label{Eq:DfB} d p_{\mathfrak{g}} \wedge \beta = 0.\end{equation} Combining \eqref{Eq:Al1} with \eqref{Eq:DA1} and \eqref{Eq:DfB} we get \[ 0 = d\alpha =  p_{\mathfrak{g}} \cdot d \beta  \qquad \Rightarrow \qquad d\beta = 0.\]

\item We proved that \[  \underbrace{ \mathcal{A}_x \wedge \dots \wedge  \mathcal{A}_x}_{m} = \pm p_{\mathfrak{g}} (x) \cdot \star d g(x), \qquad \deg g =  \frac{1}{2}\left( \dim \mathfrak{g} + \operatorname{ind} \mathfrak{g}\right) - \deg p_{\mathfrak{g}}\] (the degree increases by one after integration). It remains to prove that $\deg g(x) = \deg f(x)$. On one hand, $\deg f(x) \geq \deg g(x)$, since by \cite[Theorem 5]{BolsZhang} \[\deg df(x) \geq \frac{1}{2}\left( \dim \mathfrak{g} + \operatorname{ind} \mathfrak{g}\right) - \deg p_{\mathfrak{g}}.\] On the other hand,  $g(x)$ is an invariant, since $dg(x) = \beta(x) \in \operatorname{Ker} \mathcal{A}_x$. By condition, $\mathbb{C}[\mathfrak{g}]^{\mathfrak{g}} = \mathbb{C}[f]$. Therefore, $\deg g(x) \geq \deg f(x)$. We proved\footnote{We also proved \eqref{Eq:WedgeF}. Since $g(x)$ is an invariant, $\mathbb{C}[\mathfrak{g}]^{\mathfrak{g}} = \mathbb{C}[f]$ and $\deg g(x) = \deg f(x)$, we get $g(x) = c f(x)$.} that $\deg f = \deg g$ and got the required formula \eqref{Eq:DegFReq}. \end{enumerate}

Lemma~\ref{L:AnotherDeg} is proved. \end{proof}

\begin{remark} Note that we used the restrictions on $\mathfrak{g}$ in the proof of  Lemma~\ref{L:AnotherDeg}:

\begin{itemize}

\item We need unimodularity (= existence of invariant density) for \eqref{Eq:DA1} to hold. 

\item We, obviously, need $\operatorname{ind}\mathfrak{g} = 1$  for \eqref{Eq:DA1} and \eqref{Eq:DfB}.

\item We need $p_{\mathfrak{g}} \in \mathbb{C}[\mathfrak{g}]^{\mathfrak{g}}$ for \eqref{Eq:DfB}.

\end{itemize}

\end{remark}

\subsubsection{Example: semi-direct sum $ \operatorname{sl}(kl)  \ltimes \left(\mathbb{C}^{kl}\right)^k$}

So far, not many example of Lie algebras with mixed invariants were found. The next example falls under the scope of   Theorem~\ref{T:Ind1Th}.

\begin{example}[K.\,S.~Vorushilov, \cite{Vor2}] Consider the semidirect sum  \[\mathfrak{g} = \operatorname{sl}(kl)  \ltimes \left(\mathbb{C}^{kl}\right)^k\] w.r.t. the standard representation of $\operatorname{sl}(kl)$. The JK invariants are 

\begin{itemize}

\item one Kronecker $\left(kl(l+1) -1\right) \times \left(kl(l+1) -1\right)$ block.

\item $\displaystyle  \frac{kl(l+1)}{2}$ Jordan tuples $J_{\lambda_i}( \underbrace{2, \dots, 2}_{k-1 \textrm{ times }})$.

\end{itemize}

Since this is a ``natural'' example, all the formulas turn out quite nice. If we take an element $x=(Y, L)$, where  $Y \in \operatorname{sl}(n)$ and $L$ is a $n\times k$ matrix ($n =kl$), then we can contruct the $n\times n$ matrix \[ M(x) = \left(L, YL, \dots, Y^{l-1} L\right).\] The Casimir function (=invariant of coadjoint representation) is \[f(x) = \operatorname{det} M(x).\] Note that $\displaystyle \deg f(x) = \frac{kl(l+1)}{2}$. The fundamental semi-invariant \[ p_{\mathfrak{g}}(x) = \operatorname{det} M(x)^{k-1}.\]  In \cite{Vor2}  it was shown that all Jordan blocks are $2\times 2$. It was done by calculating $\dim \mathfrak{g}_y$ for a singular $y \in \operatorname{Sing}_0$ and using Proposition~\ref{P:NumJordBlocks}. Note that the JK invariants have the form described in Theorem~\ref{T:Ind1Th}. $\mathfrak{g}$ has no (proper) semi-invariants and satisfies Theorem~\ref{T:Ind1Th} since it has an abelian radical. \end{example}

\section{Open problems on Jordan--Kronecker invariants}

Let us state several remaining questions about realization of JK invariants. So far, the answer to Question~\ref{Q:JKLie} in the general mixed case is not clear. A natural question, after the results of this paper (in particular, Theorems~\ref{Th:RealOneKronSeveralEigen} and Theorem~\ref{Th:FinalObs}), is the following. 

\begin{problem} \label{Q:RealNJord} Consider JK invariants of a Lie algebra such that:

\begin{enumerate}

\item There are no Kronecker $3 \times 3$ block and no more than one Kronecker $1 \times 1$ block.

\item The Kronecker sizes are $2k_1-1, \dots, 2k_q - 1$.

\item There are $N$ Jordan tuples with multiple maxima $J_{\lambda_i}(2n_{i1}, \dots, 2n_{is_i}), n_{i1} = n_{i2} \geq n_{ij}$. 

\end{enumerate}

Is it true that \[ N \leq k_1 + \dots + k_q?\] 

\end{problem}

We suggest several possible ways to tackle the Problem~\ref{Q:RealNJord}:

\begin{itemize}

\item Consider JK invariants with $1$ or $2$ Kronecker blocks.

\item Consider JK invariants such that all Jordan blocks are $2\times 2$.

\item Consider Lie algebras that have no (proper) semi-invariants. 

\end{itemize} 	

\subsection{JK invariants with $1$ or $2$ Kronecker blocks}

Of course, one can study the general mix case. Yet, it is highly likely, that it would be enough to solve the problem when there is only 1 or 2 Kronecker blocks. After that the answer in the general case should become clear. 

\begin{problem} What JK invariants with 1 or 2 Kronecker blocks can be realized by Lie algebras? Let the sizes of Kronecker blocks be $(2k_i -1) \times (2k_i-1)$. There are 4 distinct cases, where the answer is unknown:

\begin{enumerate}

\item $k_1 > 2$ (one Kronecker block).

\item $k_1 = 1$ and $k_2 > 2$ (a trivial $1 \times 1$ and a non-trivial Kroncker blocks)

\item $k_1 = k_2 > 2$ (equal non-trivial Kronecker blocks)

\item $k_1 > k_2 > 2$ (distinct non-trivial Kronecker blocks)

\end{enumerate}

\end{problem}

The ``simplest'' JK invariants  that were not yet realized are one Kronecker $5\times 5$ block and $4$ equal Jordan tuples with multple maxima. An example of such JK invariants with the smallest dimension is \[J_{\lambda_1} (2, 2), \qquad J_{\lambda_2} (2, 2), \qquad J_{\lambda_3} (2, 2),  \qquad J_{\lambda_4} (2, 2), \qquad k_1 = 3.\] The corresponding Lie algebra $\mathfrak{g}$ should be $21$-dimensional.

The case of one Kronecker block (i.e. $\operatorname{ind} \mathfrak{g} = 1$) can probably be solved completely. One may start with the following simpler question about the algebra of invariant polynomials $\mathbb{C}[ \mathfrak{g}]^{\mathfrak{g}}$ and the field of rational invariants $\mathbb{C}\left( \mathfrak{g}\right)^{\mathfrak{g}}$.

\begin{problem} Let $\mathfrak{g}$ be a Lie algebra with $\operatorname{ind} \mathfrak{g} = 1$ and the fundamental semi-invariant $p_{\mathfrak{g}}(x)$. 

\begin{enumerate}

\item Let $f(x) \in \mathbb{C}[ \mathfrak{g}]^{\mathfrak{g}}$ be a nontrivial irreducible invariant. Is it true that \[\deg f =   \frac{1}{2}\left( \dim \mathfrak{g} + \operatorname{ind} \mathfrak{g}\right) - \deg p_{\mathfrak{g}}?\]

\item Let $\displaystyle f(x) = \frac{p(x)}{q(x)} \in \mathbb{C}\left( \mathfrak{g}\right)^{\mathfrak{g}}$ be a rational invariant, $\operatorname{gcd}(p, q) = 1$. Is it true that \[ \deg p + \deg q =   \frac{1}{2}\left( \dim \mathfrak{g} + \operatorname{ind} \mathfrak{g}\right) - \deg p_{\mathfrak{g}}?\]

\end{enumerate}

\end{problem} 

It would be very peculiar if there is a Lie algebra $\mathfrak{g}$ such that the following 3 conditions hold: 

\begin{itemize}

\item $\operatorname{ind} \mathfrak{g} = 1$, the Kronecker block is $(2k-1) \times (2k-1)$,

\item $ \mathbb{C}[ \mathfrak{g}]^{\mathfrak{g}} =  \mathbb{C}[f]$, where $\deg f > k$.

\item $p_{\mathfrak{g}}(x)$ is an invariant, i.e. $p_{\mathfrak{g}}(x) = f(x)^d$.

\end{itemize}

The number of core Jordan tuples for such Lie algebra $\mathfrak{g}$ will be $ \deg f > k$. We discuss the local structure Poisson pencils with one Kronecker $(2k-1) \times (2k-1)$ block and more than $k$ core Jordan tuples in Appendix~\ref{S:LocalManyCore}. This structure is ``quite rigid'' (see Theorem~\ref{T:GoodCoord1KronManyEigen} and Remark~\ref{Rem:FinalRem}). Roughly speaking, for such Lie-Poisson pencils the core eigenvalues $\lambda_i(x)$ satisfy a system of PDE with $k$ equations on $k$ variable. It would be interesting, if they could also be  the roots of the polynomial (i.e. the fundamental semi-invariant) $p_{\mathfrak{g}}(x-\lambda_i a) = 0$. 

\subsection{Semi-invariants and Jordan blocks}

Jordan blocks are a hassle to work with. We still do not know how to (completely) determine the number and sizes of Jordan blocks from the fundamental semi-invariant $p_{\mathfrak{g}}(x)$. 

\begin{problem} How to (effectively) determine sizes of Jordan blocks and Jordan tuples  $J_{\lambda_i}(2n_{i1}, \dots, 2n_{is_i})$ for a Lie algebra $\mathfrak{g}$? Are there analogues of Lemma~\ref{L:StabSingGen} and Proposition~\ref{P:NumJordBlocks} for nontrivial Jordan blocks (i.e. with size bigger than $2\times2$)?  \end{problem} 

\begin{remark} Let $e_1, \dots, e_n, f_1, \dots, f_n$ be a standar basis of a Jordan $\lambda_i$-block. The vectors $f_i$ satisfy  \[ (A- \lambda_i B) f_1 = 0, \qquad B f_1 = (A - \lambda_i B) f_2, \qquad \dots, \qquad B f_{n-1} = (A - \lambda_i B) f_n\] and a similar system of equations holds for $e_1, \dots, e_n$. In theory, we can determine the Jordan $\lambda_i$-blocks for a Lie algebra $\mathfrak{g}$  by studying the system of equations \begin{gather*} \operatorname{ad}^*_{\xi_1} (x - \lambda_i a) = 0, \\ \operatorname{ad}^*_{\xi_2} (x - \lambda_i a) = \operatorname{ad}^*_{\xi_1} a , \\ \dots \\ \operatorname{ad}^*_{\xi_{l} } (x - \lambda_i a) = \operatorname{ad}^*_{\xi_{l-1}} a.  \end{gather*} This method may not be effective in practice. \end{remark}

Let $f_j(x)$ be irreducible factors of $p_{\mathfrak{g}}$. By Proposition~\ref{Prop:ReducEigen}, if $\lambda_i$ are the roots of the same polynomial $f_j(x - \lambda_i a)=0$, then the Jordan tuples $J_{\lambda_i}(2n_{i1}, \dots, 2n_{is_i})$ are equal. Thus, the Jordan tuples must be determined by the semi-invariants $f_j(x)$. 

One may try to get results about Jordan blocks of a Lie algebra $\mathfrak{g}$ by studying the semi-invarians of $\mathfrak{g}$. For some results about semi-invariants of Lie algebras see \cite{Ooms22} and the references therein. Denote by  $\left(S \mathfrak{g} \right)^\mathfrak{g}_{\mathrm{si}}$ the algebra generated by semi-invariants (in \cite{Ooms10} it was called \textit{the Poisson semi-center} of $S\mathfrak{g}$). The fundamental question is  \begin{question} \label{Q:SemiJK} Can we determine any properties of JK invariants of $\mathfrak{g}$ from $\left(S \mathfrak{g} \right)^\mathfrak{g}_{\mathrm{si}}$? \end{question}

Question~\ref{Q:SemiJK} is too broad. Let us state a more concrete Problem~\ref{Prob:Ooms}, to demonstrate what kind of answers we expect in Question~\ref{Q:SemiJK}. The next statement is well-known.

\begin{theorem}[A.\,I.~Ooms, M.~Van~den~Bergh, \cite{Ooms10}] \label{T:JSSemi} Assume that  $\left(S \mathfrak{g} \right)^\mathfrak{g}_{\mathrm{si}}$ is freely generated by homogeneous elements $f_1, \dots, f_r$. Then \[ \sum_{i=1}^r \deg f_i \leq \frac{1}{2} \left( \dim \mathfrak{g} + \operatorname{ind} \mathfrak{g}\right).\]  Moreover, if there are no (proper) semi-invariants and $\mathbb{C}[\mathfrak{g}]^{\mathfrak{g}}$ is polynomial, then \[ \sum_{i=1}^r \deg f_i   = \frac{1}{2} \left( \dim \mathfrak{g} + \operatorname{ind} \mathfrak{g} \right) - \deg p_{\mathfrak{g}}.\] \end{theorem}

It may interesting to compare the sum $ \displaystyle \sum_{i=1}^r \deg f_i$ with the transcendence degree of of the extended Mischenko-Fomenko subalgebra $\tilde{\mathcal{F}}_a$ from \cite{Izosimov14}. In terms of the JK invariants, we can ask the following.

\begin{problem} \label{Prob:Ooms} Assume that  $\left(S \mathfrak{g} \right)^\mathfrak{g}_{\mathrm{si}}$ is freely generated by homogeneous elements $f_1, \dots, f_r$. Consider the JK invariants of $\mathfrak{g}$. Let $M$ be the number of Jordan tuples $J_{\lambda_i}(2n_{i1}, \dots, 2n_{is_i})$  with unique maximum (i.e. $n_{i1} > n_{ij}$). Is it true that  \[ \sum_{i=1}^r \deg f_i \leq \frac{1}{2} \left( \dim \mathfrak{g} + \operatorname{ind} \mathfrak{g}\right) - \deg p_{\mathfrak{g}} + M?\]   \end{problem} 

\subsubsection{Trivial Jordan blocks}

Perhaps it will be easier to realize the JK invariants, when all Jordan blocks are $2\times 2$. 
 
 \begin{hypothesis} \label{Hyp:Jord2}
 If JK invariants  \[J_{\lambda_1}(2n_{11}, \dots, 2n_{1s_1}), \quad  \dots, \quad  J_{\lambda_p}(2n_{p1}, \dots, 2n_{ps_p}), \quad 2k_1 -1, \quad \dots,  \quad 2k_q - 1, \] where all Jordan tuples have multiple maxima, can be realized by a Lie algebra, then the JK invariants \[J_{\lambda_1}(\underbrace{2, \dots, 2}_{n_{i1} + \dots + n_{1s_1}}), \quad  \dots, \quad  J_{\lambda_p}(\underbrace{2, \dots, 2}_{n_{p1} + \dots + n_{ps_p}}), \quad 2k_1 -1, \quad \dots,  \quad 2k_q - 1 \]   can also be realized by a Lie algebra. 
 \end{hypothesis}
 
 Our reasoning in Hypothesis~\ref{Hyp:Jord2} is as follows. Roughly speaking, the obstructions for realization of JK invariants from Sections~\ref{S:ObstJord}, \ref{S:GenObst} and \ref{S:LieNoInvRest} follow from the next 2 properties of Jordan blocks and eigenvalues:

\begin{enumerate}

\item If a Jordan tuples has multiple maximum, then $d\lambda_i \in \mathcal{K}$. 

\item If $\lambda_i$ are roots of the same polynomial $f_j(x - \lambda_i a)=0$, then their Jordan tuples are equal. 

\end{enumerate}

The first property is a restriction on the eigenvalues $\lambda_i(x)$. The second property requires some Jordan tuples to be equal. When we replace each Jordan tuple $J_{\lambda_i}(2n_{i1}, \dots, 2n_{is_i})$ with $J_{\lambda_i}(\underbrace{2, \dots, 2}_{n_{i1} + \dots + n_{1s_1}})$, the eigenvalues don't change and equal Jordan tuples remain equal. We do note claim that this operation for Lie algebras is possible. This is an intuition why Hypothesis~\ref{Hyp:Jord2} holds, not a formal proof.

\subsection{Lie algebras without semi-invariants}

It is easier to work with invariants, then with semi-invariants. We already discussed some restriction on JK invariants of Lie algebras without semi-invariants in Sections~\ref{S:LieNoInvRest} and \ref{S:UnimInd1}. If there are several Kronecker blocks, one may try to constuct a counter-example to Problem~\ref{Q:RealNJord} by studying the following question.

\begin{problem} \label{Prob:NoSemiInv} Let $g(x) \in \mathbb{C}[x_1, \dots, x_r]$ be a homogeneous polynomial. Is there a Lie algebra $\mathfrak{g}$ such that 

\begin{itemize}

\item  $\operatorname{ind} \mathfrak{g} = r$, i.e. the number of variable of $g(x)$,

\item  $\mathfrak{g}$ has no (proper) semi-invariants, 

\item the algebra of invariants $\mathbb{C}[\mathfrak{g} ]^{\mathfrak{g}}$ is polynomial on $f_1, \dots, f_r$,

\item the fundamental semi-invariant $p_{\mathfrak{g}} = g(f_1, \dots, f_r)$?

\end{itemize}

\end{problem}

The answer to Problem~\ref{Prob:NoSemiInv}  is probably negative. Still, it may provide some obstructions to realization of JK invariants. In a sense, after the results of this paper we can replace the question:

\begin{itemize}

\item ``What JK invariants can be realized by Lie algebras?''

\end{itemize}

with an equally non-trivial question: 

\begin{itemize}

\item ``What fundamental semi-invariants $p_{\mathfrak{g}}$ are possible?''

\end{itemize}

\subsection{JK invariants of subalgebras}

A totally different approach would be to consider JK invariants for subalgebras.

\begin{question} Let $U\subset V$ be codim-1 linear subspace. What JK invariants for $U$ and $V$ are possible? The same question for Lie algebras.
\end{question}

\appendix

\section{Local structure of compatible symplectic forms} \label{S:TurielSympApp}

In this section we briefly retell F.\,J.~Turiel's results \cite{turiel} on local structure of nondegenerate Poisson pencils (=compatible symplectic forms). 

\begin{itemize}

\item  The main theorems are stated in Section~\ref{S:TurThState}.  

\item In order to formulate these theorems we introduce the notion of T-regular points in Section~\ref{S:LocTur4TReg22} (in \cite{turiel} they were called ``regular''). 

\end{itemize}

We already used a weaker version of Turiel's results in Theorem~\ref{Th:Turiel4}. In this section we specify Theorem~\ref{Th:Turiel4} : the ``generic points'' from this theorem are actually T-regular points (see Theorem~\ref{Th:Turiel4Mega}).  A natural question about Theorem~\ref{Th:Turiel4} is 

\begin{itemize}

\item ``why must each Jordan tuple have a unique maximum?'' 

\end{itemize}

First, we demonstrate that effect in a simple example in Section~\ref{S:SimpEx}. Then we explain it in more details in Sections ~\ref{S:TurThState} and \ref{S:EigenLargest}.

\subsection{Example for a pair of Jordan blocks} \label{S:SimpEx}

We start with an example. Let $\mathcal{P} = \left\{ \mathcal{A}  + \lambda \mathcal{B} \right\}$ be a Poisson pencil on $\mathbb{C}^6(p_0, p_1, q_0, q_1, p'_1, q'_1)$ given by \begin{equation} \label{Eq:ExPoisStart} \mathcal{A} = \left(\begin{array}{cccc|cc}  0 &  0 & p_0 & 0  & 0 & 0 \\ 0 & 0 & p_1 & p_0 & 0 & 0\\ -p_0 & -p_1 &  0 & 0 & 0 & 0 \\ 0 & -p_0 & 0 & 0  & 0 & 0 \\ \hline 0 & 0 & 0 & 0  & 0 & p_0 \\ 0 & 0 &  0 & 0  & -p_0 & 0 \end{array} \right), \qquad \mathcal{B} = \left(\begin{array}{cccc|cc}   0 &  0 & 1 & 0  & 0 & 0 \\ 0 & 0 & 0 & 1  & 0 & 0 \\ -1 & 0 &  0 & 0  & 0 & 0 \\ 0 & -1 & 0 & 0  & 0 & 0 \\ \hline 0 & 0 & 0 & 0  & 0 & 1 \\ 0 & 0 &  0 & 0  & -1 & 0  \end{array} \right). \end{equation} The JK-decomposition of \eqref{Eq:ExPoisStart} at a generic point consists of the Jordan $4 \times 4$ block and the Jordan $2\times 2$ block. The eigenvalue is $\lambda_0 = p_0$.

\begin{itemize} 

\item Note that for all regular bracket $\mathcal{A}_{\lambda} \in \mathcal{P}$ the corresponding Hamiltonian vector fields are proportional: \begin{equation} \label{Eq:HamEigen} \left(\mathcal{A} + \lambda \mathcal{B} \right)  dp_0 = - (p_0 + \lambda) \frac{\partial}{\partial q_0}.\end{equation} These Hamiltonian vector fields span the distribution \[\Delta_{\mathrm{eigen}} = \operatorname{span}  \left(\frac{\partial}{\partial q_0} \right).\] Thus, the quotient by the Hamiltonian action  is \[\mathbb{C}^6/\Delta_{\mathrm{eigen}}  \approx \mathbb{C}^5(p_0, p_1, q_1, p'_1, q'_1).\] 

\item The induced Poisson pencils on the quotient $\mathbb{C}^6/\Delta_{\mathrm{eigen}}$ are \begin{equation} \label{Eq:IndFlatEx} \mathcal{A}' = \left(\begin{array}{ccc|cc}  0 &  0 & 0  & 0 & 0 \\ 0 & 0 & p_0 & 0 & 0\\ 0 & -p_0  & 0  & 0 & 0 \\ \hline 0 & 0  & 0  & 0 & p_0 \\ 0 & 0 & 0  & -p_0 & 0 \end{array} \right), \qquad \mathcal{B}' =\left(\begin{array}{ccc|cc}  0 &  0 & 0  & 0 & 0 \\ 0 & 0 & 1 & 0 & 0\\  0 & -1  & 0  & 0 & 0 \\ \hline 0 & 0  & 0  & 0 & 1 \\ 0 & 0 & 0  & -1 & 0 \end{array} \right). \end{equation} Note that on each level set $p_0 = \left\{ const\right\}$ the Poisson pencil \eqref{Eq:IndFlatEx} induces a flat nondgenerate Poisson pencil.

\item The JK-decomposition of \eqref{Eq:IndFlatEx} at a generic point consist of 

\begin{itemize}

\item  One trivial $1\times 1$ Kronecker block (corresponding to the common Casimir $p_0$).

\item Two Jordan $2\times 2$ blocks with eigenvalue $p_0$. 

\end{itemize}

Note that, comparing to  \eqref{Eq:ExPoisStart}, the size of one Jordan block decreased by $2$. 

\item The recursion operator $P = \mathcal{A}\mathcal{B}^{-1}$ has the form \[ P = \left(\begin{array}{cccc|cc}  p_0 &  0 &  &   &  &  \\ p_1 &p_ 0 &  & & & \\  &  &  p_0 & p_1 &  &  \\ &  & 0 & p_0  &  &  \\ \hline  &  &  &   & p_0 & 0 \\  &  &   &   & 0 & p_0 \end{array} \right)\] The nilpotent operator $N = P - \lambda_0 \operatorname{Id}$ has degree $d_0 = 2$.  

\item The differential of the eigenvalue has the form \[ d\lambda_0 = dp_0 = \left(1, 0, \dots, 0 \right).\]  In \eqref{Eq:ExPoisStart} the largest Jordan block is stricly larger than other blocks because \begin{equation} \label{Eq:ImEx} d \lambda_0 \in \operatorname{Im} \left(P^* - \lambda_0 \operatorname{Id}\right)^{d_0-1}.\end{equation}  In general case we proof the analogue of \eqref{Eq:ImEx} in Lemma~\ref{L:DifImd} below.

\end{itemize}

\subsection{Hamiltonian vector fields for eigenvalues}

First, we generalize \eqref{Eq:HamEigen}. Recall that for any Poisson pencil $\mathcal{P}$ its eigenvalues $\lambda_i(x)$ satisfy  \eqref{Eq:Eigen1}. This condition from Lemma~\ref{L:EigenDiff} can be written as \[ \mathcal{A}_{-\lambda_i} d \lambda_i = 0. \]  Note that for any $\mu \in \mathbb{C}$ we have \[ \mathcal{A}_{\mu - \lambda_i} d \lambda_i = \left( \mu \mathcal{B} + \mathcal{A}_{-\lambda_i}\right) d \lambda_i  = \mu \mathcal{B} d \lambda_i. \] We get the following.

\begin{proposition} \label{P:HamEigen} Let $x_0$ be a JK-regular point and $\lambda(x)$ be an eigenvalue for a Poisson pencil $\mathcal{P} = \left\{ \mathcal{A} + \lambda \mathcal{B} \right\}$. Consider the corresponding Hamiltonian vector fields \[ X_{\mu} = A_{\mu} d \lambda(x).\] All the vectors $X_{\mu}(x_0)$ are proportional. Moreover, \[ \operatorname{span} \left( X_{\mu_1} (x_0) \right) = \operatorname{span} \left( X_{\mu_2}(x_0) \right), \qquad \forall \mu_1, \mu_2 \not \in \Lambda_{x_0}. \] \end{proposition}

\subsection{Eigenvalue reduction}

In this section we consider Poisson pencils $\mathcal{P}$ with trivial $1\times 1$ Kronecker blocks (i.e. locally all regular brackets $\mathcal{A}_{\lambda} \in \mathcal{P}$ have common Casimir functions $z_1, \dots, z_r$). Let us consider the distribution $\Delta_{\mathrm{eigen}}$  that is spanned by the Hamiltonian vector fields of the eigenvalues $\lambda_i(x)$.

\begin{proposition} \label{P:EigenDist} Let $\mathcal{P} = \left\{\mathcal{A} + \lambda \mathcal{B} \right\}$ be a Poisson pencil on a manifold $M$, and  $\mathcal{K}$ be the core distibution. Assume that 

\begin{enumerate}

\item $\mathcal{P}$ has a constant algebraic type on $M$ (i.e. all points are JK-regular),

\item all Kronecker blocks are trivial $1\times 1$ blocks,

\item each eigenvalue $\lambda_i(x), i=1, \dots, p$ is \[\text{ either } \quad \lambda_i = \operatorname{const}, \qquad \text{ or } \qquad d \lambda_i (x) \not \in \mathcal{K}.\] 

\end{enumerate}

Then the distribution generated by the Hamiltonian vector fields of the eigenvalues \begin{equation} \label{Eq:EigenDist} \Delta_{\mathrm{eigen}}(x) =\underset{\mu \not \in \Lambda_x}{ \operatorname{span}} \left( \mathcal{A}_\mu d\lambda_1 (x), \quad \dots, \quad \mathcal{A}_\mu d\lambda_p(x) \right) \end{equation} is well-defined (i.e. it does not depend on the choice of $\mu \not \in \Lambda_x$). Moreover, $\Delta_{\mathrm{eigen}}$ is an integrable distribution. \end{proposition}

\begin{proof}[Proof of Proposition~\ref{P:EigenDist}] 

\begin{itemize}

\item By Propositon~\ref{P:HamEigen} the rhs of \eqref{Eq:EigenDist} doesn't depend on the choice of a regular value $\mu \not \in \Lambda_x$. Thus \eqref{Eq:EigenDist} correctly defines a subspace \[\Delta_{\mathrm{eigen}}(x) \subset T_x M\] at each point $x \in M$.

\item Next, we show that $\Delta_{\mathrm{eigen}}$ is a distribution, i.e. that \[\dim \Delta_{\mathrm{eigen}}(x)  = \operatorname{const}.\] The case $\lambda_i = \operatorname{const}$ is trivial ($\mathcal{A}_{\mu} d \lambda_i(x) =0$). If $d \lambda_i(x) \not \in \mathcal{K}$, then $\mathcal{A}_{\mu} d \lambda_i(x) \not =0$. Also,  by Lemma~\ref{L:EigenDiff} and the JK theorem the non-zero vectors $\mathcal{A}_{\mu} d \lambda_i(x)$ are linearly independant. Thus, $\dim \Delta_{\mathrm{eigen}}(x) $ is a constant, equal to the number of eigenvalues such that  $d \lambda_i(x) \not \in \mathcal{K}$.

\item The distribution is integrable, since it is involutive: \[ \left[\mathcal{A}_{\mu} d \lambda_i, \mathcal{A}_{\mu} d \lambda_j\right]  = \pm \mathcal{A}_{\mu} d \left\{ \lambda_i, \lambda_j \right\}_{\mu} = 0. \] As usual in symplectic geometry, the sign $\pm$ depends on the sign conventions for Hamiltonian vector fields. The last equality follows from Corollary~\ref{Cor:EigenCommute}.

\end{itemize}

Proposition~\ref{P:EigenDist} is proved. \end{proof}

Since  $\Delta_{\mathrm{eigen}}$ is integrable, we can consider the (local) quotient\footnote{Since we study Poisson pencils locally, we can replace $M$ with its sufficiently small open subset, if necessary. Hence, we can assume that $M/\Delta_{\mathrm{eigen}}$ is a smooth manifold.} $M/\Delta_{\mathrm{eigen}}$. Next, we can induce a Poisson pencil on $M/\Delta_{\mathrm{eigen}}$ from $\mathcal{P}$ using the following statement.

\begin{proposition} Let $(M, \mathcal{A})$ be a Poisson manifold, $f$ be a function such that \[ df(x)  \not \in \operatorname{Ker} \mathcal{A}(x), \qquad \forall x \in M,\] and the quotient by the Hamiltonian action $N = M/ \operatorname{span} \left( \mathcal{A} df\right)$ is a smooth manifold. Then there exists a unique Poisson bracket $\hat{\mathcal{P}}$ on $N$ such that the map \[ (M, \mathcal{A} ) \to (N, \hat{\mathcal{P}})\] is a Poisson map. \end{proposition}

Denote the induced pencil on $M/\Delta_{\mathrm{eigen}}$ by $\mathcal{P}_{\mathrm{eigen}}$.

\begin{definition}  We call the map \[ \pi: (M, \mathcal{P}) \to \left(M/\Delta_{\mathrm{eigen}}, \mathcal{P}_{\mathrm{eigen}}\right)\] the \textbf{eigenvalue reduction}. \end{definition}

\begin{remark} \label{Rem:SympRed} If the Poisson pencil $\mathcal{P}$ is nondegenerate,  then we can perform the \textbf{symplectic reduction} of $(M, \mathcal{P})$ by the Hamiltonian action of eigenvalues $\lambda_i(x)$. Recall that a symplectic reduction is a two-stage procedure: 

\begin{enumerate}

\item First, we take the level set $Q_c = \left\{ \lambda_i(x) = c_i = \operatorname{const} \right\}$.

\item Then we consider the quotient of $Q_c$ by the (local) action $G$ of Hamiltonian vector field $X_{\mu, i} = A_{\mu} d \lambda_i(x)$. The vector fields $X_{\mu, i}$ commute, since the eigenvalues $\lambda_i(x)$ are in involution by Corrolary~\ref{Cor:EigenCommute}.

\end{enumerate}

By Proposition~\ref{P:HamEigen} the Hamiltonian vector fields $X_{\mu, i}$ are proportional for all regular $\mu$, hence the quotient is the same (for all regular $\mu$). The eigenvalue reduction is ``the second half'' of the symplectic reduction. Moreover, we have the following slightly informal ``commutative diagram'': 
 \[\begin{CD}
 (M,\Pen)  @< \supset << Q_c \\
@V\operatorname{proj} VV @VV\operatorname{proj}V\\
M/\Delta_{\mathrm{eigen}} = M/G @< \supset << Q_c / G
\end{CD}\] Here the horizontal maps are restrictions to the level sets $\left\{\lambda_i(x) = c_i\right\}$ and the vertical maps are the quotients by the Hamiltonian action of $\lambda_i(x)$. More formally, if $\mathcal{P}$ is nondegenerate, then almost all forms $\mathcal{A}_{\mathrm{eigen}, \lambda} \in \mathcal{P}_{\mathrm{eigen}}$ have common symplectic leaves. These symplectic leaves are the level sets $\left\{ \lambda_i(x) = \operatorname{const} \right\}$ in $\left(M/\Delta_{\mathrm{eigen}}, \mathcal{P}_{\mathrm{eigen}}\right)$, and they are the  result of the symplectic reduction $Q_c /G$ of $(M, \mathcal{P})$ by the Hamiltonian action of the eigenvalues $\lambda_i(x)$.
\end{remark}

\subsection{T-regular points}  \label{S:LocTur4TReg22}
 
\cite{turiel} describes local form for nondegenerate Poisson pencils $(M, \mathcal{P})$ in a neighbourhood of a generic point $x \in M$. In  \cite{turiel} these generic points are called ``regular''. In this paper we call them T-regular\footnote{We try not to use widespread vague terms like ``regular'' in definitions or statements. Everybody wants their objects to be ``regular'', ``generic'' or ``nice''. The desires of different people rarely coincide.}. Our definition of T-regular points is a reformulation of the definition of ``regular'' points from~\cite{turiel}. We show that these definitions are equivalent in Section~\ref{S:EquivCond}.

\begin{definition} \label{D:TReg} Let $\mathcal{P} = \left\{ \mathcal{A} + \lambda \mathcal{B} \right\}$ be a Poisson pencil on $M$ such that at all points $x \in M$ all Kronecker blocks are trivial $1\times 1$ blocks. Let $\mathcal{K}$ be the core distibution.  A point $x_0 \in M$ is \textbf{T-regular}, if in a neighbourhood $Ox_0$ it satisfies the following 3 conditions:

\begin{enumerate}

\item $x$ is JK-regular, i.e. the number and sizes of Jordan blocks are the same in $Ox_0$.

\item For each eigenvalue $\lambda_i(x)$ \[  \text{ either } \qquad  \lambda_i(x) = \operatorname{const}, \qquad \text{ or } \qquad d \lambda_i (x_0) \not \in \mathcal{K}.\] 

The next condition is required only if $d \lambda_i (x_0) \not \in \mathcal{K}$.

\item \label{Item:PStable}  The point $\pi(p)$ is JK-regular after the eigenvalue reduction \[\pi: (M, \mathcal{P}) \to \left(M/\Delta_{\mathrm{eigen}}, \mathcal{P}_{\mathrm{eigen}}\right).\]

\end{enumerate}

\end{definition}

In this paper we mostly need the following obvious property of T-regular points.

\begin{proposition}  T-regular points form an open dense subset of $M$. \end{proposition}

\subsection{Turiel's theorems on nondegenerate Poisson pencils} \label{S:TurThState}

\begin{definition} We call a T-regular point $x_0 \in (M, \mathcal{P})$  \textbf{noncritical} if for all eigenvalues  $d \lambda_i(x_0) \not = 0$. \end{definition}

\begin{theorem}[F.\,J.~Turiel, \cite{turiel}] \label{Th:Turiel3}
Let $x' \in (M', \mathcal{P}')$ and $x'' \in (M'', \mathcal{P}'')$ be noncritical T-regular points for nondegenerate Poisson pencils. These pencils are locally isomorphic iff  the JK invariants and all the corresponding eigenvalues coincide at $x'$ and $x''$.\end{theorem}

By Theorem~\ref{Th:Turiel4} not all JK invariants in Theorem~\ref{Th:Turiel3} are possible. (Theorem~\ref{Th:Turiel4}  is formulated for one eigenvalue $\lambda(x)$, but the general case can be reduced to the case with one eigenvalue by Theorem~\ref{Th:Turiel1}.) The next theorem specifies that ``generic points'' from Theorem~\ref{Th:Turiel4} are T-regular points. 

\begin{theorem}[F.\,J.~Turiel, \cite{turiel}] \label{Th:Turiel4Mega} Let $\mathcal{P}$ be a nondenerate Poisson pencil on a manifold $M$ with one eigenvalue $\lambda(x)$ such that $d \lambda (x) \not = 0$. Then for a T-regular point the Jordan tuple $J_{\lambda}(2n_{1}, \dots, 2n_{s})$ has a unique maximum, i.e. $n_{1} > n_{j}$.\end{theorem}

\begin{remark} Under the conditions of Theorem~\ref{Th:Turiel4Mega} assume that the nilpotent field of endomorphisms $N = P - \lambda(x) \operatorname{Id}$ has degree $d$, i.e. $N^{d-1} \not = 0, N^d = 0$.  This restriction in Theorem~\ref{Th:Turiel4Mega} comes from the condition \[  d \lambda_i \in \operatorname{Im} ( P^* - \lambda(x) \operatorname{Id})^{d-1},\] proved in Lemma~\ref{L:DifImd} below. Let us briefly explain that effect (omitting the details). 

 \begin{itemize}
 
 \item  By Lemma~\ref{L:EigenDiff} we have  \begin{equation} \label{Eq:LamKerTur} d \lambda(x) \in \operatorname{Ker} ( P^* - \lambda(x) \operatorname{Id}).\end{equation} 
 
 \item We can perform the symplectic reduction by the Hamiltonian action of the eigenvalue $\lambda(x)$ (see Remark~\ref{Rem:SympRed}). After that symplectic reduction we get a nondegenerate Poisson pencil  $\mathcal{P}_{\mathrm{flat}}$ with constant eigenvalues. By  Theorem~\ref{Th:Turiel4} this pencil $\mathcal{P}_{\mathrm{flat}}$  is flat. 
 
 \item Thus (by the Caratheodory--Jacobi--Lie theorem) we can find local coordinates $(x_1, \dots, x_{2n-2}, p, q)$ such that 

\begin{itemize} 

\item the eigenvalue is one of the coordinates $q =\lambda(x)$,

\item and the matrices of the Poisson brackets take the form:  \begin{equation} \label{Eq:MatrOm12} \mathcal{A} = \left(\begin{array}{c|cc} \Omega_1 & \beta & 0\\ \hline -\beta^T & 0 & q \\ 0 & -q & 0\end{array}  \right), \qquad \mathcal{B} = \left(\begin{array}{c|cc} \Omega_0 & 0 & 0 \\ \hline 0 & 0 & 1 \\ 0 & -1 & 0\end{array}  \right),  \end{equation}  where $\Omega_i$ are constant matrices.  

\end{itemize}

Roughly speaking, the column $\beta$ in \eqref{Eq:MatrOm12} increases the size of one Jordan block. 

\item Since \eqref{Eq:LamKerTur} holds, by Proposition~\ref{P:NJordBasis}, if the Jordan blocks with eigenvalues $\lambda(x)$ in $\mathcal{P}_{\mathrm{flat}}$ have sizes \[ 2k_1 \geq \dots \geq  2k_s,\] then the Jordan blocks with eigenvalues $\lambda(x)$ in $\mathcal{P}$ have the sizes \[ 2k_1, \quad \dots, \quad 2k_{j-1}, \quad 2k_j +2,\quad  2k_{j+1}, \quad \dots, \quad 2k_s.\] 

\item Simply speaking, we need to increase the size of one Jordan $\lambda(x)$-block in $\mathcal{P}_{\mathrm{flat}}$ by $2$.  The condition \[  d \lambda_i \in \operatorname{Im} ( P^* - \lambda(x) \operatorname{Id})^{d-1}\] from Lemma~\ref{L:DifImd} below guarantee that we increase the size of the largest block. Hence we get Jordan tuples with multple maxima in Theorem~\ref{Th:Turiel4}. 
 
 \end{itemize}
 
 \end{remark}

\begin{remark} In \cite{Kozlov15} Turiel's theorems were erroneously stated not for the T-regular points, but for the JK-regular points (i.e. without Condition~\ref{Item:PStable} in Definition~\ref{D:TReg}). In \cite{Kozlov15} the local coordinates from \cite{turiel} are studied, thus in that paper one must consider T-regular points. Both JK-regular and T-regular points are generic, therefore the main results of \cite{Kozlov15} remain true for T-regular points. \end{remark}

\subsection{Eigenvalue and the largest Jordan block} \label{S:EigenLargest}

The next fact is not obvious at all. It was proved in ~\cite{turiel} in the next paragraph after Lemma 6. For completeness sake we briefly retell its proof (omitting the details).

\begin{lemma}[F.\,J.~Turiel, \cite{turiel}] \label{L:DifImd} Let $\mathcal{P} = \left\{ \mathcal{A} + \lambda \mathcal{B} \right\}$ be a nondegenerate Poisson pencil on $M$ with one eigenvalue $\lambda(x) < \infty$. Assume that nilpotent field of endomorphism \[N = P - \lambda(x) \operatorname{Id}\] has degree $d$ on $M$, i.e. \[ N^{d-1} \not = 0, \qquad N^ d = 0.\] Then for any T-regular point $x$ we have \[ d \lambda(x) \in \operatorname{Im} (N^*)^{d-1}(x). \] \end{lemma}

In other words, \textit{there exists a Jordan basis for $N$ such that $d \lambda(x)$ belongs to the largest Jordan block}. In the proof of Lemma~\ref{L:DifImd} we use the following well-known statement (this is \cite[Proposition 2]{turiel}).

\begin{proposition} \label{P:CondCompatForms} Let $(M, \omega_0)$ be a symplectic manifold and $P: TM \to TM$ be field of endomorphisms that is self-adjoint w.r.t. $\omega_0$. Consider the $2$-forms $\omega_j = \omega \circ P^j$. 

\begin{enumerate}

\item The forms $(\omega_0, \omega_1)$ are compatible, i.e. $N_P = 0$, if and only if $d \omega_1 = d \omega_2 = 0$. 

\item If $\omega_0$ and  $\omega_1$ are compatible, then $d \omega_k = 0$ for all $k \geq 0$. 

\end{enumerate}

\end{proposition}

\begin{proof}[Proof of Lemma~\ref{L:DifImd}] Instead of Poisson brackets we consider the corresponding symplectic forms $\omega_0 = \mathcal{B}^{-1}$ and $\omega_1 = \mathcal{A}^{-1}$. 

\begin{enumerate}

\item Consider the symplectic reduction by the Hamiltonian action of the eigenvalue $\lambda(x)$ (see Remark~\ref{Rem:SympRed}). After that reduction the pencil $\mathcal{P} = \left\{ \mathcal{A} + \lambda \mathcal{B} \right\}$ becomes flat. Thus, by Theorem~\ref{Th:Turiel4} there exist local coordinates $(x_1, \dots, x_{2n-2}, p, q)$ such that 

\begin{itemize} 

\item the eigenvalue is one of the coordinates $q =\lambda(x)$,

\item and the forms take the form \[ \omega_0 = \sum_{i=1}^{n-1} dx_{2i-1} \wedge dx_{2i-2} + d p\wedge dq, \qquad \omega_1 = q \cdot \omega_0 + \tau + \alpha \wedge dq, \]  where $\tau=\sum_{1 \leq i < j  \leq 2n-2} b_{ij} dx_i \wedge dx_j$. 

\end{itemize}

In other words, the matrices of $\omega_0, \omega_1$ have the form \begin{equation} \label{Eq:MatrOm12} \omega_0 = \left(\begin{array}{c|cc} \Omega_0 & 0 & 0 \\ \hline 0 & 0 & 1 \\ 0 & -1 & 0\end{array}  \right), \qquad \omega_1 = \left(\begin{array}{c|cc} \Omega_1 & 0 & \alpha \\ \hline 0 & 0 & q \\ - \alpha^T & -q & 0\end{array}  \right).  \end{equation} 

\item The recursion operator $P = \omega_0^{-1} \omega_1$ takes the form \[ P = q \cdot \operatorname{Id}  + H  + \frac{\partial}{\partial p} \otimes \alpha - Z \otimes dq. \] In terms of the matrices \eqref{Eq:MatrOm12}, the matrix of the recursion operator is \begin{equation} \label{Eq:MatrOm12P}   P = \left(\begin{array}{c|cc} \Omega_0^{-1} \Omega_1 & 0 & \Omega_0^{-1} \alpha \\ \hline \alpha^T & q & 0 \\ 0 & 0 & q\end{array}  \right).  \end{equation} Thus, in the matrix form \[ H = \Omega_0^{-1} \Omega_1 - q \cdot \operatorname{Id}, \qquad Z = \Omega_0^{-1} \alpha. \] 

\item Consider the coordinates  $(x_1, \dots, x_{2n-2})$ and denote \[ \omega' = \sum_{i=1}^{n-1} dx_{2i-1} \wedge dx_{2i-2}.\] Using Proposition~\ref{P:CondCompatForms} it is not hard to check that $\omega_0$ and $\omega_1$ are compatible if and only if \begin{equation} \label{Eq:CompCond1} \mathcal{L}_Z \omega'  = - \omega', \qquad \mathcal{L}_Z  H = - H. \end{equation} Here the Lie derivative $\mathcal{L}$ is taken in the coordinates  $(x_1, \dots, x_{2n-2})$.

\item Denote $N = P^*(x) - \lambda(x) \operatorname{Id}$. Note that \[ d\lambda (x) = (0, \dots, 0, 1). \] From \eqref{Eq:MatrOm12P}  it is easy to see that \begin{equation} \label{Eq:CondLamH} d \lambda(x) \in \operatorname{Im} (N^*)^{r+1}(x),  \qquad \Leftrightarrow \qquad H^{r} Z \not = 0.\end{equation}

\item Now, assume that $H^r Z  =0$ and $H^r \not = 0$. 

\begin{itemize}

\item On one hand, by \eqref{Eq:CompCond1} and Cartan's magic formula \[ \mathcal{L}_Z \left( \omega' \left(H^r, \cdot \right)\right) =   d \left( \omega' \left(H^r Z, \cdot \right)\right)  = 0.\]

\item On the other hand, by \eqref{Eq:CompCond1}  we have \[ \mathcal{L}_Z \left( \omega' \left(H^r, \cdot \right)\right)  = \left(\mathcal{L}_Z  \omega' \right)  \left(H^r, \cdot \right)  + \omega'  \left(\mathcal{L}_Z   H^r, \cdot \right)   = -(r+1)  \omega' \left(H^r, \cdot \right)\]

\end{itemize}

Thus, $\omega'(H^r, \cdot) \equiv 0$, and we get a contradiction. 

\item Now, assume that $H^{d-2} Z \not = 0$ and $H^{d-1} = 0$. On one hand, by \eqref{Eq:CondLamH} we have \[ d \lambda(x) \in \operatorname{Im} (N^*)^{d-1}(x) \] and hence $\deg N \geq d$. On the other hand, from \eqref{Eq:MatrOm12P}   it is easy to see  that $N^d  =0$. Therefore, $d = \deg N$. 
\end{enumerate}

Lemma~\ref{L:DifImd} is proved. \end{proof}

\subsubsection{Equivalent conditions on eigenvalues} \label{S:EquivCond}

 Condition~\ref{Item:PStable} in Definition~\ref{D:TReg} can be formulated in several equivalent ways. It is not hard to proof the following. 

\begin{proposition} \label{P:EquivCond} Let $\mathcal{P}$ be a nondenerate Poisson pencil on a manifold $M$ such that each eigenvalue $\lambda_i(x)$ is either a constant, or a noncritical function $d \lambda_i (x) \not = 0$. Let $P = \mathcal{A}\mathcal{B}^{-1}$ be the recursion operator. Consider a (sufficiently small) neighbourhood $U$ of a JK-regular point $x_0 \in M$. The following conditions are equivalent: 

\begin{enumerate}

\item \label{Cond:T1} $x_0$ is T-regular, i.e. the point $\pi(x_0)$ is JK-regular after the eigenvalue reduction \[\pi: (M, \mathcal{P}) \to \left(M/\Delta_{\mathrm{eigen}}, \mathcal{P}_{\mathrm{eigen}}\right).\]

\item \label{Cond:T2} Let $d_i$ be the multiplicity of $\lambda_i$ in the minimal polynomial for $P$. Then \[d \lambda_i(x) \in \operatorname{Im} (P^* - \lambda_i \operatorname{Id})^{d_i-1}(x).\]

\item \label{Cond:T3} Consider the distribution \[E = \bigcap_i \operatorname{Ker} d \lambda_i = \bigcap_j \operatorname{Ker} d \left(\operatorname{tr} P^j\right) .\] The algebraic type of the restriction $P\bigr|_{E}$ is constant. 	  

\end{enumerate}

\end{proposition} 

\begin{proof}[Proof of Proposition~\ref{P:EquivCond}] By Lemma~\ref{L:EigenDiff} $d \lambda_i(x) \in \operatorname{Ker} ( P^* - \lambda_i(x) \operatorname{Id})$. Thus, for any point $x \in U$ there exists a Jordan basis such that the differential of each eigenvalue $d\lambda_i(x)$ belongs to a Jordan $n_i(x) \times n_i(x)$ block for the recursion operator $P$. Conditions~\ref{Cond:T1} and \ref{Cond:T3} in Proposition~\ref{P:EquivCond} mean that the sizes of these Jordan block $n_i(x)  = \operatorname{const}$. By Lemma~\ref{L:DifImd}, they are equivalent to Condition~\ref{Cond:T2} that each $d\lambda_i(x)$ belong to the largest Jordan $\lambda_i$-block. Proposition~\ref{P:EquivCond} is proved. \end{proof}

\section{Turiel's decomposition theorem} \label{S:TurielDecompApp}

In this section we formulate Turiel's decompositon theorem from \cite{Turiel11}. In particular, we specify for which ``generic'' points does Theorem~\ref{T:TurielDecompTrivKronOneJord} hold. 

\begin{definition} Let $\mathcal{P}$ be a Poisson pencil on a manifold $M$. We say that $(M, \Pen)$ \textbf{decomposes into a Jordan--Kronecker product} at point $ x\in M$ if there exist a neighbourhood $Ox$ and a decomposition \[ (Ox, \mathcal{O} ) \approx \left(M_{\textrm{Jord}}, \Pen_{\textrm{Jord}}\right) \times \left(M_{\textrm{Kron}}, \Pen_{\textrm{Kron}}\right),\] such that $\Pen_{\textrm{Jord}}$ is a Jordan pencil and $\Pen_{\textrm{Kron}}$ is a Kronecker pencil. \end{definition}

 Turiel's decompositon theorem \cite[Theorem 7.1]{Turiel11} can be divided into the following two theorem:
 
 \begin{itemize}
 
 \item Theorems~\ref{T:TurielTrivKron} (all Kronecker blocks are $1\times 1$),
 
 \item Theorem~\ref{TourielDecompGenCase2} (general mixed case). 
 
 \end{itemize}
 
 In Section~\ref{S:RestMantle} we describe the method that allows us to reduce the general case to the case, when all Kronecker blocks are $1\times 1$.
 
\begin{remark} For simplicity sake we consider only the analytic case. \cite[Theorem 7.1]{Turiel11} holds in the following cases:

\begin{enumerate}

\item The Poisson pencil $\mathcal{P}$ is complex analytic and holomorphic,

\item $\mathcal{P}$ is real analytic,

\item $\mathcal{P}$ is  real $C^{\infty}$-smooth and all eigenvalues $\lambda_i(x)$ are real.

\end{enumerate}

\cite{Turiel11} also describes a counter-example to the decomposition theorem \cite[Theorem 7.1]{Turiel11} in the real $C^{\infty}$-smooth case, when there is a pair of complex-conjugate eigenvalues $\lambda_i (x)= \alpha (x) \pm i \beta(x)$. \end{remark}
 
\subsection{Case of trivial Kronecker blocks} First, we consider the case when all Kronecker blocks are trivial. 

\begin{theorem}[F.\,J.~Turiel, \cite{Turiel11}]  \label{T:TurielTrivKron} Let $\mathcal{P} = \left\{\mathcal{A} + \lambda \mathcal{B} \right\}$ be an analytic  Poisson pencil on a manifold $M$, and $\mathcal{K}$ be its core distribution. Let $x_0 \in (M, \Pen)$  be a point such that in its neighbourhood $Ox_0$  the following 4 conditions hold:

\begin{enumerate}

\item All Kronecker blocks are $1\times 1$. 

\item \label{Cond:JKTur1} $x_0$ is JK-regular.

\item \label{Cond:EigenTur1} Each eigenvalue $\lambda_i(x)$ is either \[ \lambda_i = \operatorname{const}, \qquad \text{ or } \qquad d \lambda_i(x_0) \not\in \mathcal{K}.\] 

\item The point $\pi(p)$ is JK-regular after the eigenvalue reduction \[\pi: (M, \mathcal{P}) \to \left(M/\Delta_{\mathrm{eigen}}, \mathcal{P}_{\mathrm{eigen}}\right).\]

\end{enumerate}
 
 Then $(M, \Pen)$ decomposes into a Jordan--Kronecker product at $x_0 \in M$. 

\end{theorem}

\begin{remark} Consider the dual-core distribution $\displaystyle \mathcal{K}^0 = \bigcap_{\lambda - \text{regular}} \Imm (\mathcal{A} + \lambda \mathcal{B})$.  In \cite[Theorem 7.1]{Turiel11}  Condition~\ref{Cond:EigenTur1} from Theorem~\ref{T:TurielTrivKron} was written as \begin{equation} \label{Eq:CondEigen} \dim \operatorname{span} \left(d \lambda_1, \dots, d \lambda_n \right) = \dim
\operatorname{span} \left(d\lambda_1 |_{\mathcal{K}^0}, \dots, d \lambda_n |_{\mathcal{K}^0}\right). \end{equation}  It is easy to show, using local coordinates from  Theorems~\ref{T:BiPoissRedCoreMantle}  and \ref{T:TrivKronFact}, that \eqref{Eq:CondEigen} holds if and only if each eigenvalue $\lambda_i(x)$ is either a constant $\lambda_i = \operatorname{const}$, or $d \lambda_i(x_0) \not\in \mathcal{K}$.  \end{remark}

Note that in Theorem~\ref{T:TurielTrivKron}  the JK invariants must be the same in a neighbourhood of $x_0$.  Next example shows, roughly speaking, that for decomposable pencils the number and sizes of Jordan blocks should not depend on common Casimir functions $z_i$ (i.e. functions corresponding to the Kronecker $1\times 1$ blocks). 

\begin{example}  Consider a Poisson pencil on $\mathbb{C}^5(x_1,x_2, y_1, y_2,z)$ with matrices: \[ \mathcal{A} =  \left( \begin{array}{cc|cc|c} 0 & 0 & 1 & z & 0 \\ 0 & 0 & 0 & 1 & 0 \\ \hline -1 & 0 & 0 & 0 & 0 \\ -z & -1 & 0 & 0 & 0 \\ \hline0& 0 & 0 & 0 & 0 \end{array} \right), \qquad \mathcal{B} =  \left( \begin{array}{cc|cc|c} 0 & 0 & 1 & 0 & 0 \\ 0 & 0 & 0 & 1 & 0 \\ \hline -1 & 0 & 0 & 0 & 0 \\ 0 & -1 & 0 & 0 & 0 \\ \hline0& 0 & 0 & 0 & 0 \end{array} \right) .\] The symplectic leaves of all brackets $\mathcal{A}_{\lambda} = \mathcal{A} + \lambda \mathcal{B}$ coincide and have the form $z = z_0$. On each symplectic leaf $z = z_0$ the brackets $\mathcal{A}_{\lambda}$ are constant. But the JK invariants depend on $z$ and are as follows:

\begin{itemize}

\item $4\times 4$ Jordan block and $1\times 1$ Kronecker block if $z \not =0$,

\item two $2\times 2$ Jordan blocks and $1\times 1$ Kronecker block if $z =0$.

\end{itemize} 

Hence, the pencil does not decompose at the origin $0$. 
\end{example}

\begin{remark} Under the conditons of Theorem~\ref{T:TurielTrivKron}  the compatible symplectic structures on common symplectic leaves of $\mathcal{A}_{\lambda}$ satisfy the conditions of Theorem~\ref{Th:Turiel3}.  Roughly speaking, we can prove Theorem~\ref{T:TurielTrivKron} by taking the local coordinates for compatible symplectic structure from \cite[Theorem 3]{turiel} on each symplectic leaf. \end{remark}

\subsection{Reduction of Kronecker blocks} \label{S:RestMantle}

Theorem~\ref{T:BiPoissRedCoreMantle} allows us to reduce some problems about Poisson pencils $\mathcal{P}$ to the case when the Kronecker blocks of $\mathcal{P}$ are trivial $1\times 1$ blocks. Let $(x, s, y)$ be the local coordinates from Theorem~\ref{T:BiPoissRedCoreMantle}. Note that in the coordinates $(s, y)$ we can consider the Poisson pencil $\mathcal{P}_{\mathrm{red}} = \left\{ \mathcal{A}_{\mathrm{red}, \lambda}  \right\} $ with the matrices \begin{equation} \label{Eq:KronRedMat} \mathcal{A}_{\mathrm{red}, \lambda} = \left( \begin{matrix} C_{\lambda}( s, y) & 0 \\ 0 & 0\end{matrix} \right). \end{equation} 

Note that the mantle distribution $\mathcal{M} = \operatorname{span} \left(ds, dy\right)$. Thus, \eqref{Eq:KronRedMat} is the restriction of $\mathcal{P}$ to the mantle $\mathcal{M}$. Thus, from the JK theorem we get the following simple statement.

\begin{proposition} \label{P:KronRed} Let $(x, s, y)$ be the coordinates from Theorem~\ref{T:BiPoissRedCoreMantle}. Denote $\pi: (x, s, y) \to (s, y)$. Consider the JK decompositions of $\mathcal{P}$ at a point $p$ and $\mathcal{P}_{\operatorname{red}}$ at the point $\pi(p)$. The following holds:
\begin{itemize}

\item The Jordan tuples for $\Pen_p$ and for $\Pen_{\mathrm{red}, \pi(p)}$ coincide.

\item If the Kronecker sizes of $\Pen_p$ are $2k_1 -1, \dots, 2k_q-1$, then the Kronecker blocks for $\Pen_{\mathrm{red}, \pi(p)}$ are $k_1 + \dots + k_q$ trivial $1\times 1$ Kronecker blocks. 

\end{itemize}

\end{proposition}

In particular, the characteristic numbers (=eigenvalues) of $\mathcal{P}$ and $\mathcal{P}_{\mathrm{red}}$ coincide. 

\begin{remark} The pencil $\Pen_{\mathrm{red}}$ can be invariantly defined as the pencil induced by the projection \begin{equation} \label{Eq:Kronred} \pi: (M, \Pen) \to \left(M/ \mathcal{M}^0, \Pen_{\mathrm{red}}\right).\end{equation} Here ``induced'' means that \eqref{Eq:Kronred}  is a Poisson map for each bracket $\mathcal{A}_{\lambda} \in \Pen$, i.e. \[ \left\{\pi^*f, \pi^*g\right\}_{\lambda} = \left\{f, g\right\}_{\mathrm{red}, \lambda}, \] for any smooth functions $f, g$ on $M/ \mathcal{M}^0$. Note that in the coordinates $(x, s, y)$ the dual-mantle distribution \[ \mathcal{M}^0 = \operatorname{span}\left\{ \frac{\partial}{\partial x}\right\}.\]  Thus, \eqref{Eq:Kronred}  is the projection $\pi: (x, s,y) \to (s, y)$. The projection \eqref{Eq:Kronred}  was used in \cite{Turiel11}. \end{remark}

\begin{definition}  Consider the local coordinates $(x, s, y)$ from Theorem~\ref{T:BiPoissRedCoreMantle}. We call the map \[ \pi: (M, \Pen) \to \left(M/ \mathcal{M}^0, \Pen_{\mathrm{red}}\right), \] such that $\pi(x, s, y) = (s, y)$ and the Poisson pencil $\mathcal{P}_{\mathrm{red}} = \left\{ \mathcal{A}_{\mathrm{red}, \lambda}  \right\} $ is given by \eqref{Eq:KronRedMat}, the \textbf{Kronecker reduction}. \end{definition}

\subsection{General mixed case} 

The decomposition theorem in the general mixed case reduces to the case, when all Kronecker blocks are $1\times 1$. 

\begin{theorem}[F.\,J.~Turiel, \cite{Turiel11}]\label{TourielDecompGenCase2}  Let $\mathcal{P} = \left\{\mathcal{A} + \lambda \mathcal{B} \right\}$ be an analytic Poisson pencil on a manifold $M$. Assume that on $M$ \[\deg p_{\Pen} (x) = \operatorname{const}, \qquad \operatorname{rk} \mathcal{P}(x) = \operatorname{const}.\] Consider the Kronecker reduction \[ \pi: (M, \Pen) \to \left(M/ \mathcal{M}^0, \Pen_{\mathrm{red}}\right). \] Then  $(M, \mathcal{P})$ decomposes into a Jordan--Kronecker product at $x \in M$ if $\left(M/ \mathcal{M}^0, \Pen_{\mathrm{red}}\right)$ satisfies the conditions of Theorem~\ref{T:TurielTrivKron} at the point $\pi(x)$ (and, thus, also decomposes into a Jordan--Kronecker product at $\pi(x)$).
\end{theorem}

\section{Local coordinates for one Kronecker $(2k-1) \times (2k-1)$ blocks and $k$ eigenvalues in the core} \label{S:LocalManyCore}

In this section we discuss local coordinates for Poisson pencils with one Kronecker $(2k-1) \times (2k-1)$  block and ``a lot of'' (at least $k$) core eigenvalues. Here it is more convenient to use the number $m = k-1$, then $k$.

\begin{theorem}  \label{T:GoodCoord1KronManyEigen}  Let $\Pen = \left\{ \mathcal{A}_{\lambda} = \mathcal{A} + \lambda \mathcal{B}\right\}$ be a Poisson pencil on $M$ and $p \in (M, \Pen)$ be a JK-regular point. Assume that the JK invariants at $p$ are \begin{itemize}

\item one Kronecker $(2m+1) \times (2m+1)$ block,

\item $N > m$ Jordan tuples $J_{\lambda_i}\left(2n_{i1}, \dots, 2n_{is_i}\right)$ for $i =1, \dots, N$. 

\end{itemize}

Assume that the first $m+1$ eigenvalues $\lambda_{1}, \dots, \lambda_{m+1}$ each satisfies \[\lambda_i  <  \infty, \qquad  d \lambda_{i} \not = 0, \qquad d \lambda_{i}\in \mathcal{K}\] on $M$. Then in a neighbourhood of  $p \in M$ there exist local coordinates \[x_1,\dots, x_{m},\quad  s_1, \dots, s_{2n_J}, \quad y_1 = \lambda_1, \dots, y_{m+1} = \lambda_{m+1}\] such that the core and mantle distribution are \begin{equation} \label{Eq:CoreMantleDistLoc1Kron} \mathcal{K} = \operatorname{span}\left\{dy_1, \dots, dy_{m+1} \right\}, \qquad \mathcal{M} = \operatorname{span}\left\{ds_1, \dots, ds_{2n_J}, dy_1, \dots, dy_{m+1} \right\},\end{equation} and the matrices of brackets $\mathcal{A}$ and $\mathcal{B}$ have the form  \begin{equation} \label{Eq:AB1Kron} \mathcal{A} = \left( \begin{matrix}  A_{xx}(s, y)  & A_{xs}(s, y)   & A_{xy}(y)   \\ -A_{xs}^T(s, y) & A_{ss}(s, y)  & 0 \\ -A_{xy}^T(y) & 0 & 0\end{matrix} \right), \qquad \mathcal{B} = \left( \begin{matrix}  B_{xx}(s, y)  & B_{xs}(s, y)   & B_{xy}(y)   \\ -B_{xs}^T(s, y) & B_{ss}(s, y)  & 0 \\ -B_{xy}^T(y) & 0 & 0\end{matrix} \right), \end{equation} where \begin{equation} \label{Eq:ABXY} A_{xy} = B_{xy} \left( \begin{matrix} y_1 \\ \vdots \\ y_{m+1} \end{matrix} \right), \qquad B_{xy} = \left( \begin{matrix}  1 & & & c_{m+1}(y_{m+1}) / c_1(y_1) \\   & \ddots &  & \vdots\\ &  &   1 & c_{m+1}(y_{m+1}) / c_m(y_m) \end{matrix} \right) \end{equation} for some functions $c_1(y_1), \dots, c_{m+1}(y_{m+1})$.  \end{theorem}

In other words, in \eqref{Eq:ABXY} we have  \[A_{xy} =  \left( \begin{matrix}  y_1 & & & y_{m+1}  \cdot \frac{c_{m+1}(y_{m+1})}{c_1(y_1)} \\   & \ddots &  & \vdots\\ &  &  y_m & y_{m+1} \cdot \frac{c_{m+1}(y_{m+1}) }{ c_m(y_m)} \end{matrix} \right).\]

\begin{proof}[Proof of Theorem~\ref{T:GoodCoord1KronManyEigen}] The proof is in several steps.

\begin{enumerate}

\item \textit{First, apply Theorem~\ref{T:BiPoissRedCoreMantle}, to get local coordinates \[ x'_1,\dots, x'_{m},\quad  s_1, \dots, s_{2n_J}, \quad y_1 = \lambda_1, \dots, y_{m+1} = \lambda_{m+1}\] such that the matrices $\mathcal{A}_{\lambda}$ have the form \eqref{Eq:CoreMantleMatr}.} Note that by Lemma~\ref{L:EigenDiff} and Proposition~\ref{P:LinIndKronBlocks} the differentials of characteristic numbers $\lambda_1, \dots, \lambda_{m+1}$ are linearly independent \[ d\lambda_1 \wedge \dots \wedge d\lambda_{m+1} \not = 0.\]  Hence, we can include the characterictic numbers $\lambda_1, \dots, \lambda_{m+1}$ in the local coordinates. 

\item  \label{Step:2ManyCoreEigen} \textit{Next, we use the Caratheodory--Jacobi--Lie theorem for the Poisson manifolds, i.e. Theorem 2.1 from  \cite{Miranda08}.} Namely, we apply it for the Poisson manifold $(M, \mathcal{B})$  and the functions $y_1 = \lambda_1, \dots, y_m=\lambda_{m}$, which are in involution w.r.t. $\mathcal{B}$.  This theorem gives us functions $x_1,\dots, x_{m}$ such that \begin{equation} \label{Eq:XDeltaB} \left\{ x_i, \lambda_j \right\}_{\mathcal{B}} = \delta_{ij}, \qquad i,j =1, \dots, m.\end{equation} In order to apply this theorem, we need to prove that the Hamiltonian vector fields $X_i =\mathcal{B} d \lambda_i$, $i=1, \dots, m$ are linearly independant \[ X_1 \wedge \dots \wedge X_m \not =0.\] Let $dz$ be local non-trivial Casimir function of $\mathcal{B}$, i.e. $dz \in \operatorname{Ker} \mathcal{B}$ and $dz \not = 0$. Since there is only 1 Kronecker block, the vector fields $X_i =\mathcal{B} d \lambda_i$ are linearly dependant iff $d \lambda_i$ and $dz$ are linearly dependant. $d \lambda_i$ and $dz$ are linearly independant by Proposition~\ref{P:LinIndKronBlocks}.

\item After the previous step we got functions $x_1,\dots, x_{m}$ that satisfy \eqref{Eq:XDeltaB}. Simply speaking, the functions $x_1,\dots, x_{m}$ are the ``times of flows'' along the Hamiltonian vector fields $X_1, \dots, X_m$: \[ X_j(x_i) =\delta_{ij}.\] \textit{We take \begin{equation} \label{Eq:CoordX2} x_1, \dots, x_m, \qquad s_1, \dots, s_{2n_J}, \quad y_1 = \lambda_1, \dots, y_{m+1} = \lambda_{m+1}\end{equation}  as the new local coordinates.} We can take these functions as local coordinates, since any $y_i$ commutes with all the functions $s_j, y_k$ and \eqref{Eq:XDeltaB}  holds. 

\item \textit{By construction, in the coordinates~\eqref{Eq:CoordX2} the core $\mathcal{K}$ and the mantle $\mathcal{M}$ have the form \eqref{Eq:CoreMantleDistLoc1Kron}. } The brackets $\mathcal{A}_{\lambda}$ have the form  \[ \mathcal{A}_{\lambda} = \left( \begin{matrix}  A_{\lambda, xx}  & A_{\lambda, xs} & A_{\lambda, xy}    \\ -A_{\lambda, xs}^T & A_{\lambda,  ss} (s, y) & 0 \\ -A_{\lambda, xy}^T  & 0 & 0\end{matrix} \right). \] Since\eqref{Eq:XDeltaB}  holds, we have \[  B_{xy} = \left( \begin{matrix}  1 & & &H_1(x, s, y)\\   & \ddots &  & \vdots\\ &  &   1 & H_{m} (x, s, y)\end{matrix} \right).\] Since $\lambda_1, \dots, \lambda_{m+1}$ are eigenvalue, we also have \[A_{xy} = B_{xy} \left( \begin{matrix} y_1 \\ \vdots \\ y_{m+1} \end{matrix} \right) =  \left( \begin{matrix}  y_1 & & &y_{m+1} H_1(x, s, y)\\   & \ddots &  & \vdots\\ &  &   y_m & y_{m+1}  H_{m} (x, s, y)\end{matrix} \right).\] 

\item Similar to Step~\ref{Step:2ManyCoreEigen} we can prove that \textit{$m+1$ Hamitonian vector fields $X_i =\mathcal{B} d \lambda_i, i=1,\dots, m+1$ are in a general position at the point $p\in M$, i.e. \[ H_i(p) \not =0, \qquad i = 1, \dots, m.\]} Thus, below we can consider fractions with $H_i(p)$ in the denominator.

\item \textit{ Let us prove \eqref{Eq:AB1Kron}, i.e. that the blocks of $\mathcal{A}_{\lambda}$ do not depend on some of the coordinates $x$ or $s$.} We can do it using the Jacobi identity in the following order:

\begin{enumerate}

\item $H_i = H_i(s, y)$ because of the Jacobi identity for $(x_{i}, y_j, y_{k})$.

\item  $\mathcal{A}_{\lambda, xs} = \mathcal{A}_{\lambda, xs}(s, y)$ and  $\mathcal{A}_{\lambda, xy} = \mathcal{A}_{\lambda, xy}(y)$ (i.e. $H_i = H_i(y)$) because of the Jacobi identity for $(x_{i}, s_j, y_{k})$. 

\item $\mathcal{A}_{\lambda, xx} = \mathcal{A}_{\lambda, xx}(s, y)$ because of the Jacobi identity for $(x_{i}, x_{j}, y_{k})$ for $k=1, \dots, m$. Indeed, \[ \left\{ \left\{ x_i, x_j \right\}_{\lambda}, y_k \right\}_{\lambda} = \left\{ \left\{ x_j, y_k  \right\}_{\lambda}, x_i\right\}_{\lambda} + \left\{ \left\{ y_k, x_i\right\}_{\lambda},  x_j \right\}_{\lambda} = 0. \] The last equality follows from $\left\{x_i, y_j \right\}_{\lambda} = (y_j + \lambda) \delta_{ij}$.
\end{enumerate}

We proved  \eqref{Eq:AB1Kron}.

\item It remains to prove that \[ H_j(y) = \frac{c_{m+1}(y_{m+1}) }{c_j(y_j)}.\] First, \textit{let us prove that \begin{equation} \label{Eq:HProp1} H_j(y) = H_j(y_j, y_{m+1}), \qquad \frac{\partial }{\partial y_{m+1} } \left(\frac{H_{i}}{H_j} \right) =0.  \end{equation} } Consider the Jacobi identity for $(x_{i}, x_{j}, y_{m+1})$. On one hand, since  $\mathcal{A}_{\lambda, xx} = \mathcal{A}_{\lambda, xx}(s, y)$, we have \[  \left\{ \left\{ x_i, x_j \right\}_{\lambda}, y_{m+1} \right\}_{\lambda} = 0. \] On the other hand, since the Jacobi idenity holds, \[  \left\{ \left\{ x_i, x_j \right\}_{\lambda}, y_{m+1} \right\}_{\lambda} = \left\{ \left\{ x_j, y_{m+1}  \right\}_{\lambda}, x_i\right\}_{\lambda} + \left\{ \left\{ y_{m+1}, x_i\right\}_{\lambda},  x_j \right\}_{\lambda} \] By Leibniz rule, we get:\begin{gather*} 0= \left\{ \left\{ x_j, y_{m+1}  \right\}_{\lambda}, x_i\right\}_{\lambda} + \left\{ \left\{ y_{m+1}, x_i\right\}_{\lambda},  x_j \right\}_{\lambda} = \\ = (y_{m+1} + \lambda) (y_{j} + \lambda) \frac{\partial H_i}{\partial y_j} +   (y_{m+1} + \lambda)^2 \frac{\partial H_i}{\partial y_{m+1}} H_j - \\ - (y_{m+1} + \lambda) (y_{i} + \lambda) \frac{\partial H_j}{\partial y_i} -   (y_{m+1} + \lambda)^2 \frac{\partial H_j}{\partial y_{m+1}} H_i. \end{gather*}  Since the last equality holds for all $\lambda$ we get that \[  \frac{\partial H_i}{\partial y_j} =  \frac{\partial H_j}{\partial y_i} = 0 \qquad \Rightarrow \qquad H_j(y)= H_j(y_j, y_{m+1}), \] and \[ \frac{\partial H_i}{\partial y_{m+1}} H_j  =\frac{\partial H_j}{\partial y_{m+1}} H_i \qquad \Rightarrow \qquad   \frac{\partial }{\partial y_{m+1} } \left(\frac{H_{i}}{H_j} \right) =0. \] \eqref{Eq:HProp1} is proved.

\item \textit{Finally, we prove that \begin{equation} \label{Eq:HEq1} H_j(y) = \frac{c_{m+1}(y_{m+1}) }{c_j(y_j)}.\end{equation} } Consider the function $H_j(y_j, y_{m+1})$ with a fixed $y_{m+1}=y_{m+1}(p)$: \[ T_j(y_j) = H_j(y_j, y_{m+1} (p) ).\] From\eqref{Eq:HProp1} it follows that\[ \frac{H_i(y_i, y_{m+1})}{H_j(y_j, y_{m+1})} = \frac{T_i(y_i)}{T_j(y_j)}.\] Thus, \[\frac{H_i(y_i, y_{m+1})}{T_i(y_i)} = \frac{H_j(y_j, y_{m+1})}{T_j(y_j)} =c_{m_1}(y_{m+1}).\] If we denote $\displaystyle c_i(y_i) = \frac{1}{T_i(y_i)}$, we get \eqref{Eq:HEq1}.

\end{enumerate}

Theorem~\ref{T:GoodCoord1KronManyEigen} is proved. \end{proof}

\begin{remark} It is not hard to change the coordinates $x_1, \dots, x_m$ in  Theorem~\ref{T:GoodCoord1KronManyEigen} in such a way that \eqref{Eq:CoreMantleDistLoc1Kron} and \eqref{Eq:AB1Kron} still hold, but the matrices $A_{xy}$ and $B_{xy}$ take the form \begin{equation} \label{Eq:NiceB1} A_{xy} = B_{xy} \left( \begin{matrix} y_1 \\ \vdots \\ y_{m+1} \end{matrix} \right), \qquad B_{xy} = \left( \begin{matrix}  c_1(y_1) & & & c_{m+1}(y_{m+1}) \\   & \ddots &  & \vdots\\ &  &   c_m(y_m) & c_{m+1}(y_{m+1}) \end{matrix} \right).\end{equation} Maybe, \eqref{Eq:NiceB1} ``looks even better'' than \eqref{Eq:ABXY}. \end{remark}

\begin{remark} Consider the local coordinates from Theorem~\ref{T:GoodCoord1KronManyEigen}. Any core eigenvalue  has the form $\lambda = \lambda(\lambda_1,\dots, \lambda_{m+1})$. The condition \eqref{Eq:Eigen1} on the (core) eigenvalue $\lambda(\lambda_1,\dots, \lambda_{m+1})$ takes the form  \begin{equation} \label{Eq:PDESystNice1} (\lambda_i - \lambda) c_i(\lambda_i) \frac{\partial \lambda}{\partial \lambda_i} + (\lambda_{m+1} - \lambda)c_{m+1}(\lambda_{m+1}) \frac{\partial \lambda}{\partial \lambda_{m+1}} = 0, \qquad i=1,\dots, m. \end{equation}  This system of equations can also be rewritten as \[(\lambda_1 - \lambda) c_1(\lambda_1) \frac{\partial \lambda}{\partial \lambda_1}  = \dots =(\lambda_m - \lambda) c_m(\lambda_m) \frac{\partial \lambda}{\partial \lambda_m}=- (\lambda_{m+1} - \lambda) c_{m+1}(\lambda_{m+1}) \frac{\partial \lambda}{\partial \lambda_{m+1}}.\]  \end{remark}

\begin{remark} \label{Rem:FinalRem} \eqref{Eq:PDESystNice1} is a system of first order quasi-linear PDE, consisting of $m$ equation on the function  $\lambda(\lambda_1,\dots, \lambda_{m+1})$ of $m+1$ variables. If we consider a Lie--Poisson pencil $\mathcal{P}_a = \left\{ \mathcal{A}_{x + \lambda a}\right\}$, then by Proposition~\ref{P:DiffRoot}, an eigenvalue $\lambda(x)$ also satisfies \[ \langle d \lambda(x), a \rangle =1.\] Thus, we get $m+1$ equations on the function  $\lambda(\lambda_1,\dots, \lambda_{m+1})$ of $m+1$ variables. Roughly speaking, that means that a core eigenvalue $\lambda(\lambda_1,\dots, \lambda_{m+1})$ is determined by its value at one point. \end{remark}


\begin{thebibliography}{99}


\bibitem{BolsIzosKonOsh} A.\,V.~Bolsinov, A.\,M.~Izosimov, A.\,Yu.~Konyaev, A.\,A.~Oshemkov, ``Algebra and topology of integrable systems. Research problems'' (Russian) \textit{Trudy Sem. Vektor. Tenzor.
Anal.}, No. 28 (2012), 119--191


\bibitem{BolsIzosTson} A.\,V.~Bolsinov, A.\,M.~Izosimov, D.~Tsonev, ``Finite-dimensional integrable systems: a collection of research problems'', \textit{Journal of Geometry and Physics}, \textbf{115} (2017), 2--15

\bibitem{BolsinovN1} A.\,V.~Bolsinov, A.\,Yu.~Konyaev, V.\,S.~Matveev, ``Nijenhuis Geometry'', {\tt arXiv:1903.04603 [math.DG] }

\bibitem{BolsMatvMirTab} A.\,V.~Bolsinov, V.\,S.~Matveev, E.~Miranda, S.~Tabachnikov,  ``Open problems, questions and challenges in finite-dimensional integrable systems'', \textit{Philosophical Transactions of the Royal Society A: Mathematical, Physical and Engineering Sciences}, Royal Society of London (United Kingdom), \textbf{376}: 2131 (2018), 1-40


\bibitem{BolsZhang}
A.\,V.~Bolsinov, P.~Zhang, ``Jordan-Kronecker invariants of finite-dimensional Lie algebras'',
\textit{Transformation Groups}, \textbf{21}:1 (2016),  51--86


\bibitem{Cooper14} D.~Cooper, J.\,F.~Manning, ``Non-faithful representations of surface groups into $\mathrm{SL}(2,\mathbb{C})$ which kill no simple closed curve'',   {\tt arXiv:1104.4492 [math.GT] }	



\bibitem{DufourZung05} 
J.-P.~Dufour, N.\,T.~Zung, \textit{Poisson structures and their normal forms}, Progress in Mathematics, Volume 242, Birkhauser Verlag, Basel, 2005


\bibitem{Fritzsche} K.~Fritzsche, H.~Grauert, \textit{From Holomorphic Functions to Complex Manifolds}, In: Graduate Texts in Mathematics 213. Springer, New York (2002)

\bibitem{Gantmaher88} F.\,R.~Gantmacher,  {\it Theory of matrices}, AMS Chelsea publishing, 1959.

\bibitem{Gar1} A.\,A.~Garazha, ``A canonical basis of a pair of compatible Poisson brackets on a matrix algebra'', \textit{Sb. Math.}, \textbf{211}:6 (2020), 838-849

\bibitem{Gar2} A.\,A.~Garazha, ``On a canonical basis of a pair of compatible Poisson brackets on a symplectic Lie algebra'', \textit{Uspekhi Mat. Nauk}, \textbf{77}:2(464) (2022), 199-200; \textit{Russian Math. Surveys}, \textbf{77}:2 (2022), 375-377


\bibitem{Izosimov16Flat}  A.\,M.~Izosimov, ``Flat bi-Hamiltonian structures and invariant densities'', \textit{Letters in Mathematical Physics}, \textbf{106}, (2016), 1415-1427

\bibitem{Izosimov14}  A.\,M.~Izosimov, ``Generalized argument shift method and complete commutative subalgebras in polynomial Poisson algebras'',  {\tt arXiv:1406.3777 [math.RT] }

\bibitem{Izosimov14Derived}  A.\,M.~Izosimov, ``The derived algebra of a stabilizer, families of coadjoint orbits, and sheets'', \textit{Journal of Lie Theory}, \textbf{24} (2014), 705-714 

\bibitem{Joseph10} A.~Joseph, D.~Shafrir, ``Polynomiality of invariants, unimodularity and adapted pairs'', \textit{Transformation groups}, \textbf{15}:4, (2010), 851-882

\bibitem{KobayashiNomidzu} S.\,Kobayashi, K.~Nomizu, \textit{Foundations of differential geometry}, Vol. II. Wiley Classics Library. 2009 [1969]


\bibitem{Kozlov15} I.\,K.~Kozlov, ``Invariant foliations of nondegenerate bi-Hamiltonian structures'', \textit{Fundam. Prikl. Mat.}, \textbf{20}:3 (2015), 91--111; \textit{J. Math. Sci.}, \textbf{225}:4 (2017), 596--610

\bibitem{Miranda08} C.~Laurent-Gengoux, E.~Miranda, P.~Vanhaecke, ``Action-angle coordinates for integrable systems on Poisson manifolds'', {\tt arXiv:0805.1679 [math.SG] }

\bibitem{Olver}
P.\,J.~Olver,  ``Canonical forms and integrability of bi-Hamiltonian systems'', \textit{Phys. Lett. A}, 
\textbf{148} (1990), 177-187

\bibitem{Ooms08} A.\,I.~Ooms, ``Computing invariants and semi-invariants by means of Frobenius Lie algebras'', 
\textit{Journal of Algebra}, \textbf{321}:4, (2009), 1293-1312

\bibitem{Ooms22} A.\,I.~Ooms, ``On Dixmier\textquotesingle s Fourth Problem'', \textit{Algebras and Representation Theory}, \textbf{25}, (2022), 561-579 

\bibitem{Ooms10} A.\,I.~Ooms, M. Van den Bergh, ``A degree inequality for Lie algebras with a regular Poisson semi-center'', \textit{J. Algebra}, \textbf{323} (2010), 305-322

\bibitem{SilvaWeinstein99}
A.\,Cannas da Silva; A.~Weinstein, \textit{Geometric models for noncommutative algebras}, AMS Berkeley Mathematics Lecture Notes, 10, 1999.


\bibitem{Thompson} R.\,C.~Thompson, ``Pencils of complex and real symmetric and skew matrices'', \textit{Linear Algebra Appl.}, \textbf{147}, 323-371 (1991)


\bibitem{turiel}
F.\,J.~Turiel,  ``Classification locale simultan\'ee de deux formes symplectiques compatibles'', \textit{Manuscripta Mathematica}, \textbf{82}:3--4 (1994), 349--362

\bibitem{Turiel11} F.\,J.~Turiel, ``The local product theorem for
bihamiltonian structures'', \texttt{arXiv:1107.2243v1 [math.SG]}


\bibitem{Voisin} C.~Voisin, \textit{Hodge Theory and Complex Algebraic Geometry-I},  Cambridge studies in advanced mathematics-76, Cambridge University press (2002)

\bibitem{Voron} A.\,S.~Vorontsov, ``Kronecker indices of Lie algebras and invariants degrees estimate'', \textit{Moscow
University Math. Bulletin}, \textbf{66}: 1 (2011), 25-29

\bibitem{Vor1} K.\,S.~Vorushilov, ``Jordan-Kronecker invariants for semidirect sums defined by standard
representation of orthogonal or symplectic Lie algebras'', \textit{Lobachevskii Journal of
Mathematics}, \textbf{36}:6 (2017), 1121-1130

\bibitem{Vor2} K.\,S.~Vorushilov, ``Jordan-Kronecker invariants of semidirect sums of the form $\operatorname{sl}(n)+(R^n)^k$ and $\operatorname{gl}(n)+(R^n)^k$'', \textit{Fundam. Prikl. Mat.}, \textbf{22}:6 (2019),  3-18

\bibitem{Vor3} K.\,S.~Vorushilov, ``Complete sets of polynomials in bi-involution on nilpotent seven-dimensional Lie algebras'', \textit{Sb. Math.}, \textbf{212}:9 (2021), 1193-1207

\bibitem{Vor4} K.\,S.~Vorushilov,  ``Jordan-Kronecker invariants of Borel subalgebras of semisimple Lie algebras'', \textit{Chebyshevskii Sb.}, \textbf{22}:3 (2021), 32-56

\bibitem{Yakimova17} O.~Yakimova, ``Some semi-direct products with free algebras of symmetric invariants'', \textit{Perspectives in Lie theory}, \textbf{19}, (2017), 266-279

\end{thebibliography}
\end{document}